\documentclass{memo-l}

\usepackage[utf8]{inputenc}
\usepackage[T1]{fontenc}
\usepackage[english]{babel}
\usepackage{hyperref}
\usepackage{amssymb}
\usepackage{color}
\usepackage{dsfont}
\usepackage{enumerate}
\usepackage{mathtools}
\usepackage{mathrsfs} 
\usepackage{bm} 
\usepackage{url}

\newtheorem{theorem}{Theorem}[chapter]
\newtheorem{lemma}[theorem]{Lemma}
\newtheorem{prop}[theorem]{Proposition}
\newtheorem{proposition}[theorem]{Proposition}
\newtheorem{cor}[theorem]{Corollary} 
\newtheorem{corollary}[theorem]{Corollary}

\theoremstyle{definition}
\newtheorem{definition}[theorem]{Definition}
\newtheorem{example}[theorem]{Example}

\theoremstyle{remark}
\newtheorem{remark}[theorem]{Remark}

\numberwithin{section}{chapter}
\numberwithin{equation}{chapter}

\newcommand{\cA}{\mathcal{A}}

\newcommand{\cF}{\mathcal{F}}

\newcommand{\cM}{\mathcal{M}}
\newcommand{\cN}{\mathcal{N}}
\newcommand{\cO}{\mathcal{O}}

\newcommand{\cR}{\mathcal{R}}
\newcommand{\cS}{{\mathcal{S}}}

\newcommand{\cU}{{U}}

\newcommand{\dsone}{\mathds{1}}

\newcommand{\bB}{\mathbb{B}}
\newcommand{\bC}{\mathbb{C}}

\newcommand{\bE}{\mathbb{E}}
\newcommand{\bF}{\mathbb{F}}
\newcommand{\bG}{\mathbb{G}}
\newcommand{\bN}{\mathbb{N}}
\newcommand{\bR}{\mathbb{R}}
\newcommand{\bZ}{\mathbb{Z}}

\newcommand{\xpsi}[2]{\psi^{[#1]}_{#2}(\omega)}
\newcommand{\xpsii}[2]{\psi^{[#1]}_{#2}}
\newcommand{\xPsi}[2]{\Psi^{[#1]}_{#2}}
\newcommand{\dt}[1]{\partial_t^{#1}}
\newcommand{\vN}{VN(\Gamma)}
\newcommand{\add}[1]{\quad \text{ #1 } \quad}

\newcommand{\uu}[2]{u_{#1}^{(#2)}}

\newcommand{\Lplcr}[1]{{L_p(#1;\ell_2^{cr})}}
\newcommand{\Lplc}[1]{{L_p(#1;\ell_2^{c})}}
\newcommand{\Lplr}[1]{{L_p(#1;\ell_2^{r})}}
\newcommand{\Lpinfty}[1]{{L_p(#1;\ell_\infty)}}
\newcommand{\Lpinftyc}[1]{{L_p(#1;\ell_\infty^c)}}
\newcommand{\Lpli}[1]{{L_p(#1;\ell_1)}}

\newcommand{\LpR}{{L_p(\bR^d;L_p(\cN))}}
\newcommand{\pp}{{p^\prime}}

\newcommand{\hatphi}{{\widehat{\varphi}}}
\newcommand{\hlp}[1]{{L_{#1}(\cM, \varphi)}}
\newcommand{\hlpr}[1]{{L_{#1}(\cR, \hatphi)}}
\newcommand{\ddp}[1]{D_{\widehat{\varphi}}^{1/2p} #1 D_{\widehat{\varphi}}^{1/2p}}
\newcommand{\aap}[1]{e^{\frac{a_k}{2p}}#1e^{\frac{a_k}{2p}}}

\DeclareMathOperator{\id}{id}
\DeclareMathOperator{\supp}{supp}
\DeclareMathOperator{\Pol}{Pol}
\DeclareMathOperator{\Tr}{Tr}

\DeclareMathOperator{\Irr}{Irr}
\DeclareMathOperator{\real}{Re}
\DeclareMathOperator{\col}{Col}
\DeclareMathOperator{\row}{Row}

\DeclareMathOperator*{\Sup}{{sup}^+}

\newcommand{\fancyot}{\mathbin{\text{\footnotesize\textcircled{\tiny \sf T}}}}
    \DeclareMathOperator{\ot}{\fancyot}

\makeindex

\begin{document}

\frontmatter

\title{Pointwise convergence of noncommutative Fourier series}


\author{Guixiang Hong}

\address{School of Mathematics and Statistics, Wuhan University, 430072 Wuhan,
	China and Hubei Key Laboratory of Computational Science, Wuhan University, 430072 Wuhan, China.}
\email{guixiang.hong@whu.edu.cn}
\curraddr{}
\thanks{}

\author{Simeng Wang}
\address{Institute for Advanced Study in Mathematics, Harbin Institute of Technology, Harbin, 150001, China}
\email{simeng.wang@hit.edu.cn}
\curraddr{}
\thanks{}

\author{Xumin Wang}
\address{Laboratoire de math\'ematiques de Besan\c{c}on, Universit\'e de Bourgogne Franche-Comt\'e, 16, route de Gray, 25030 Besan\c{c}on cedex, France.}
\curraddr{Laboratoire de Mathématiques Nicolas Oresme,
Université de Caen - Normandie, 14032 Caen, France.}
\email{xumin.wang1124@gmail.com}

\date{17 Jun 2021}



\begin{abstract}
    This paper is devoted to the study of pointwise convergence of Fourier series for group von Neumann algebras and quantum groups. It is well-known that a number of approximation properties of groups can be interpreted as  summation methods and mean convergence of the associated noncommutative Fourier series. Based on this framework, this paper studies the refined counterpart of pointwise convergence of these Fourier series. As a key ingredient, we develop a noncommutative bootstrap method and establish a general criterion of maximal inequalities for approximate identities of noncommutative Fourier multipliers. Based on this criterion, we prove that for any countable discrete amenable group, there exists a sequence of finitely supported positive definite functions  tending to $1$ pointwise, so that the associated Fourier multipliers on noncommutative $L_p$-spaces satisfy the pointwise convergence for all $p>1$. In a similar fashion, we also obtain results for a large subclass of groups (as well as quantum groups) with the Haagerup property and the weak amenability. We also consider the analogues of Fej\'{e}r  and Bochner-Riesz means in the noncommutative setting. Our approach heavily relies on the noncommutative ergodic theory in conjunction with abstract constructions of Markov semigroups, inspired by quantum probability and geometric group theory. 
Finally, we also obtain as a byproduct the dimension free bounds of the noncommutative Hardy-Littlewood maximal inequalities associated with convex bodies.
\end{abstract}

\maketitle

\tableofcontents


\mainmatter
\chapter*{Introduction}

The study of convergence of Fourier series goes back to the very beginning  of  Fourier analysis. Recall that for an integrable function $f$ on the unit circle $\mathbb T $, the Dirichlet summation method is defined as
	$$({D}_N f ) (z) =  \sum_{k=-N}^{N} \hat f (k) z^k,\quad z\in \mathbb T , \quad N\in \mathbb N,$$
	where $\hat f$ denotes the Fourier transform of $f$. 
	This summation method is quite intuitive, but very intricate to deal with. Indeed, the mean convergence of these sums is equivalent to the boundedness of the Hilbert transform, which is a typical example of Calder\'on-Zygmund singular integral operators; the corresponding pointwise convergence problem is much more complicated and was solved by Carleson and Hunt, which is now well-known as the Carleson-Hunt theorem. In order to study the Dirichlet means and their higher-dimensional versions, there have appeared numerous related problems together with other summation methods, which have always been motivating the development of harmonic analysis. For instance, as averages of Dirichlet means, the Fej\'{e}r means stand out
	$$({F}_N f ) (z) =  \sum_{k=-N}^{N} \left( 1-\frac{|k|}{N} \right) \hat f (k) z^k,\quad z\in \mathbb T , \quad N\in \mathbb N.$$
	It is well-known that ${F}_N$ defines a positive and contractive operator on $L_p (\mathbb T)$, and ${F}_N f $ converges almost everywhere to $f$ for all $1\leq p \leq \infty$ (see e.g. \cite{grafakos08classical}). In the case of higher dimensions, people consider the Bochner-Riesz means
		$$({B}^\delta_N f ) (z) =  \sum_{k\in\mathbb Z^d:|k|\leq N} \left( 1-\frac{|k|^2}{N^2} \right)^\delta \hat f (k) z^k,\quad z\in \mathbb T^d , \quad N\in \mathbb N,$$
		where $|k|=\sqrt{|k_1|^2+\dotsm+|k_d|^2}$ and $\delta>0$, which can be viewed as fractional averages of ball multipliers.
		When $\delta>\frac{d-1}{2}$, via verifying the rapid decay of the kernels, it is easy to obtain the weak $L_1$ estimate and thus all the $L_p$ estimates of the maximal Bochner-Riesz means; in this case, it is trivial that the Bochner-Riesz means is uniformly $L_p$-bounded for all $1\leq p\leq\infty$. When $\delta=\frac{d-1}{2}$, it had been a conjecture that the Bochner-Riesz means is of weak type $(1,1)$, which was solved by Christ \cite{Chr88}. 
		Given $0<\delta<\frac{d-1}{2}$, the largest possible scale of $p$ depending on $\delta$ was given by Herz \cite{Her54}  such that the Bochner-Riesz means is $L_p$-bounded. It was verified in two dimensions,  but remains one of the famous open problems in three and higher dimensions,  which is closely related to many other open problems in harmonic analysis, PDEs, additive combinatorics, number theory etc, see e.g. \cite{katztao02kakeya, tao99restriction, tao99restrictionbochnerriesz, tao04restrictionconj} and the references therein. And the $L_p$-boundedness of the maximal Bochner-Riesz means for $p<2$ is even open in two dimensions, see e.g. \cite{liwu2019bochnerriesz}. These problems have been stimulating the further development of analysis and beyond.


	
	
	
	
	
	In recent decades, similar topics have been fruitfully developed in the setting of operator algebras and geometric group theory. The study was initiated in the groundbreaking work of Haagerup \cite{haagerup78map}, motivated by the approximation properties of group von Neumann algebras. Indeed, let $\Gamma$ be a countable discrete group with left regular representation $\lambda : \Gamma \to  B(\ell_2 (\Gamma))$ given by $\lambda(g) \delta_h = \delta_{gh}$, where the $\delta_g$'s form the unit vector basis of $\ell_2 (\Gamma)$. The corresponding group von Neumann algebra $\vN$ is defined to be the weak operator closure of the linear span of
	$\lambda (\Gamma)$. For $f \in \vN$ we set $\tau (f) = \langle \delta_ e ,f\delta_ e \rangle$
	where $e$ denotes the identity of $\Gamma$. 
	Then $\tau$ is a faithful normal tracial state on $\vN$.
	Any such $f$ admits a formal Fourier series
	$$\sum_{g\in \Gamma} \hat f (g) \lambda(g)\quad \text{with } \quad  \hat f (g) = \tau (f\lambda(g^{-1})).$$
	The convergence and summation methods of these Fourier series at the operator algebraic level (i.e. at the $L_\infty(\vN)$ level) are deeply linked with the geometric and analytic properties of $\Gamma$, and in the noncommutative setting they are usually interpreted as various \emph{approximation properties} for groups (see e.g. \cite{brownozawa08bookC*,cherixetal01haagerup}). More precisely, for a function $m:\Gamma\to \mathbb C$ we may formally define the corresponding Fourier multiplier by
	\begin{equation}\label{eq:multiplierdef}
	T_m : \sum_{g\in \Gamma} \hat f (g) \lambda(g) \mapsto \sum_{g\in \Gamma} m(g) \hat f (g) \lambda(g).
	\end{equation} 
	We may consider among others the following approximate properties:
	\begin{enumerate}
		\item $\Gamma$ is \emph{amenable}  if there exists a family of finitely supported functions  $(m_N)_{N\in \bN}$  on $\Gamma$ so that $T_{m_N}$ defines a unital completely positive map on $\vN$ and $T_{m_N} f$ converges to $f$  in the w*-topology for all $f\in \vN$ (equivalently, $m_N$ converges pointwise to $1$). 
		\item $\Gamma$  has the \emph{Haagerup property} if there exists a family of  $\, c_0$-functions $(m_N)_{N\in \bN}$ on $\Gamma$ so that  $T_{m_N}$ defines a unital completely positive map on $\vN$ and $m_N$ converges pointwise to $1$. 
		\item $\Gamma$ is \emph{weakly amenable} if there exists a family of finitely supported functions $(m_N)_{N\in \bN}$ on $\Gamma$ so that $T_{m_N}$ defines a completely bounded map on $\vN$ with $\sup_N \|T_{m_N}\|_{cb} <\infty$ and $m_N$ converges pointwise to $1$.
	\end{enumerate}	
	If we take $\Gamma=\mathbb Z$ (in this case $VN(\mathbb Z)=L_\infty (\mathbb T)$) and $m_N (k) = (1 - |k|/N)_+$,  then $T_{m_N }$ recovers the Fej\'{e}r means $F_N$ and obviously satisfies the above conditions. On one hand, these approximation properties provide a natural framework of  noncommutative Fourier analysis, and on the other hand they also play an essential role in the modern theory of von Neumann algebras and in geometric group theory, such as Cowling-Haagerup's solution \cite{cowlinghaagerup89cb} of the rigidity problem of various group von Neumann algebras, Popa's deformation/rigidity theory \cite{popa07deformation}, and the study of strong solidity and uniqueness of Cartan subalgebras \cite{ozawapopa10cartan,chifansinclair13strongsolidhyperbolic,popavaes14cartanii1}.
	
	Despite the remarkable progress in this field, it is worthy mentioning that only the convergence of $T_{m_N } f$  in the w*-topology was studied in the aforementioned work. A standard argument also yields the convergence in  norm in the corresponding noncommutative $L_p$-spaces $L_p(\vN)$ for $1\leq p <\infty$.  On the other hand, the analogue of almost everywhere convergence in the noncommutative setting was introduced by Lance in his study of noncommutative ergodic theory \cite{lance76ergodic}; this type of convergence is usually called the \emph{(bilaterally) almost uniform} (abbreviated as \emph{b.a.u.} and \emph{a.u.} correspondingly) convergence ; see Section~\ref{sec:max and point}. Keeping in mind the aforementioned impressive results in both classical and noncommutative analysis, it is natural to develop a refined theory of pointwise convergence of noncommutative Fourier series. More precisely, it is known that for the previous maps $T_{m_N}$ and for $f\in L_p(\vN)$, there exists a subsequence $(N_k )_k$ (possibly depending on $f$ and $p$) such that $T_{m_{N_k} } f$ converges a.u. to $f$. From the viewpoint of analysis, the following problem naturally arises: \emph{can we choose $N_k$ to be independent of $f$, or even can we  choose $N_k$ to be $k$}? If $G$ is abelian, this is exactly the classical pointwise convergence problem. As  mentioned previously, the study of the pointwise convergence problem is much more difficult than the mean convergence one; it is  one of the major subjects of harmonic analysis, for instance the study of Bochner-Riesz means and maximal Schr\"odinger operators  \cite{tao03bochnerrieszplane, leeseeger15radial, duzhang19sharpl2schrodinger,liwu2019bochnerriesz}. So the above problem should be regarded as one of the initial steps to develop  Fourier analysis on noncommutative $L_p$-spaces.   
	
	However, compared to the classical setting, the pointwise convergence problem on noncommutative $L_p$-spaces remains essentially unexplored, up to sporadic contributions \cite{jungexu07ergodic,chenxuyin13harmonic}. The reason for this lack of development  might be explained by numerous difficulties one may encounter when dealing with maximal inequalities for noncommutative Fourier multipliers. Indeed,  in the commutative setting, the pointwise convergence problem almost amounts to the validity of maximal inequalities \cite{stein61limit}, and  the arguments for maximal inequalities depend in their turn on the explicit expressions or the pointwise estimates of the kernels. However, the kernels of noncommutative Fourier multipliers are only formal elements in a noncommutative $L_1$-space, which are in general no longer related to classical functions and cannot be compared pointwise, so the usual methods for classical maximal inequalities do  not apply to the noncommutative setting any more. Although the notion of noncommutative maximal inequality has been formulated successfully thanks to the theory of vector-valued noncommutative $ L_p $ spaces \cite{pisier98noncommutativeLp,junge02doob}, the  approaches to these inequalities are very limited, except the noncommutative Doob inequality in martingale theory \cite{junge02doob} and its analogue in ergodic theory \cite{jungexu07ergodic,hongwangliao2017noncommutative,condegonzalezparcet20sigma}, where some additional nice properties of the underlying operators are available; in particular, there is no maximal inequality in the literature for the summation methods such as Fej\'er means and Bochner-Riesz means \emph{etc.} of Fourier expansion of elements in noncommutative $L_p$-spaces associated to von Neumann algebra generated by noncommutative groups.
	
	In this paper we will  provide a first approach to the maximal inequalities and pointwise convergence theorems for noncommutative Fourier series. To our best knowledge, the current trend of investigation on noncommutative Fourier multipliers mainly relies on various transference methods and quantum probability theory (see e.g. \cite{neuwirthricard11transfer,chenxuyin13harmonic,jungemeiparcet14smooth,jungemeiparcet18riesz}). The method presented in this paper is completely independent of all the preceding work, so is entirely new. The strategy turns out to be efficient in a very general setting; roughly speaking, it allows us to deal with all Fourier-like structures including quantum groups, twisted crossed products and free Gaussian systems. In many cases, we may give an explicit answer to the pointwise convergence problem raised previously. 	
	In the following part of this section we will describe some of our main results. 
	\section*{Criteria for maximal inequalities of Fourier multipliers}
	\addtocontents{toc}{\protect\setcounter{tocdepth}{1}}
	Our key technical results give two criteria for maximal inequalities of noncommutative Fourier multipliers. These criteria only focuses on the regularity and decay information of symbols of multipliers in terms of length functions.  Hence it is relatively easy to verify.  As mentioned previously, the criteria can be extended to \emph{all Fourier-like expansions in general von Neumann algebras}.  For simplicity we only present the ones for group von Neumann algebras $\vN$ as  illustration, and we refer to Theorem~\ref{theorem:criterion1}, Theorem \ref{theorem:criterion2} and Theorem \ref{theorem:criterion2+: Haagerup case} for a complete statement. 
	
	Let $\Gamma$ be a discrete group and let $\ell:\Gamma\to [0,\infty)$ be a conditionally negative definite function on it. We consider a family of real valued unital positive definite functions $({m_t})_{t\in \mathbb R _+}$. 
	It is known that the associated operators $(T_{m_t})_{t\in \mathbb R _+}$ defined as in \eqref{eq:multiplierdef} extend to contractive maps on $L_p(\vN)$ for all $1\leq p \leq \infty$ (see Section \ref{sec:maximal} for more details). In this framework we present the following result. We refer to Section~\ref{sec:preliminaries} for the notions of noncommutative $L_p$-spaces $L_p(\vN)$ and 
	noncommutative maximal norms $\|\sup_n^+ x_n \|_p$ for $(x_n)_n\subset L_p(\vN)$. 

\smallskip
	
{\bf Criterion 1:} 
		Let $\Gamma,\ell$ and $(T_{m_t})_{t\in \mathbb R _+}$ be as above. Assume that there exist $\alpha ,\beta>0$ and $\eta\in \mathbb N _+$ such that for all $g\in \Gamma$ and $1\leq k\leq \eta$ we have
		$$|1-m_t (g)|\leq \beta \frac{\ell(g)^\alpha}{t},\quad |m_t (g)|\leq\beta \frac{ t}{\ell(g)^\alpha},\quad \left|\frac{d^k m_t (g)}{dt^k} \right|\leq\beta \frac{1}{t^k},$$	
		then for all $1+\frac{1}{2\eta}< p\leq \infty$ there exists a constant $c>0$ such that for all $f\in L_p (\vN)$,
		$$\|{\sup_{t\in\mathbb R _+} }^+ T_{m_t} f \|_p \leq c \|f\|_p  \quad\text{and}\quad  T_{m_t} f\to f \text{ b.a.u. (a.u. if $p\geq 2$) as }t\to\infty,$$
		and for all $1< p\leq \infty$ there exists a constant $c>0$ such that for all $f\in L_p (\vN)$,
		$$\|{\sup_{N\in\mathbb N} }^+ T_{m_{2^N}} f \|_p \leq c \|f\|_p  \quad\text{and}\quad  T_{m_{2^N}} f\to f \text{ b.a.u. (a.u. if $p\geq 2$) as }N\to\infty.$$ 
\smallskip

	Similar results hold for uniformly bounded (but not necessarily positive) Fourier multipliers $(T_{m_t})$ if we restrict ourselves to the case $p\geq 2$. The study of Criterion 1 relies on the analysis of lacunary subsequences $(T_{m_{2^N}})_{N\in \mathbb N}$. This type of lacunarity seems to be insufficient in the further study of abstract analysis on groups. The criterion  below is more suitable for the abstract setting, which applies to other sequences without being of the form $(T_{m_{2^N}})_{N\in \mathbb N}$ in Criterion 1 and will play a prominent role in the remaining part of this work.   
	\smallskip
	
{\bf Criterion 2:} 
		Let $\Gamma$ and $\ell $ be as above. Let $({m_N})_{N\in \mathbb N}$ be a sequence of real valued unital positive definite functions.  If there exist $\alpha,\beta >0$ such that for all $g\in \Gamma$,
		$$|1-m_N (g)|\leq \beta \frac{\ell(g)^\alpha}{2^N},\quad |m_N (g)|\leq  \beta\frac{ 2^N}{\ell(g)^\alpha},$$
		then for all $1<p\leq \infty$ there exists a constant $c>0$ such that for all ${f\in L_p (\vN)}$,
		$$\|{\sup_{N\in\mathbb N} } ^+ T_{m_N} f \|_p \leq c \|f\|_p  \quad\text{and}\quad  T_{m_N} f\to f \text{ b.a.u. (a.u. if $p\geq 2$) as }N\to\infty.$$
\smallskip

	The main idea of the proof will be to compare the Fourier multipliers with certain quantum Markov semigroups, and then apply the ergodic theory of the latter developed by \cite{jungexu07ergodic}. This is first loosely inspired by the study of variational inequalities (in particular the comparison between averaging operators and martingales in \cite{chendinghongxiao17somejump, dinghongliu17jump, hongma17qvariation}), and then by Bourgain’s approach to the dimension-free bounds of Hardy-Littlewood maximal inequalities \cite{bourgain86onhigh,bourgain86ontheLp,bourgain87ondimensionfree, deleavalguedonmaurey18dimfreesurvey}.  Bourgain's work is based on a careful study of the $L_p (\ell_\infty)$-norm estimate of differences between ball averaging operators and Poisson semigroups on Euclidean spaces. In this paper we will develop similar techniques for noncommutative Fourier multipliers and abstract quantum Markov semigroups. {This method based on ergodic theory seems to be new even for the study of commutative Fourier series. Indeed, this approach yields new results, insights and problems in classical harmonic analysis. We will carry out all this in a forthcoming paper \cite{hongwangwang21inprogress}.}  
	
	As a key point of the proof, we will develop a \emph{bootstrap} argument in the noncommutative setting for the first time.  The so-called bootstrap methods have had a deep impact on classical harmonic analysis since the original work \cite{nagelsteinwainger78difflacunary}; see e.g. \cite{duoandikoetxearubio86max,carbery89convexbody,bourgainmireksteinwrobel18dimfreevar}.
	Though not explicitly mentioned in the original papers, the aforementioned work by Bourgain \cite{bourgain86onhigh,bourgain86ontheLp,bourgain87ondimensionfree} can be essentially compared with the previous ones and recognized with hindsight as a certain bootstrap argument with independent techniques. Motivated by Bourgain's method, we will develop a bootstrap argument based on the almost orthogonality principle: deducing the desired $L_p$-estimates for $p<2$ from the $L_2$-estimates by a delicate study of suitable decompositions of $(T_{m_N} f - T_{e^{-t_N \sqrt{\ell}}} f )_N $ and certain differences of $(T_{m_N} f)_N$. It is relevant to remark that there is no straightforward way to extend the classical bootstrap arguments directly to the noncommutative setting. In particular, as a noncommutative variant of the vector-valued $L_p$-space $L_p (\Omega;\ell_2  )$ for the study of noncommutative square functions, the space $L_p (\mathcal{M};\ell_2 ^{cr} )$ (to be defined in Chapter~\ref{sec:order}) does not coincide with the interpolation space of the form  $(L_p(\mathcal{M};\ell_{q_1}),L_p(\mathcal{M};\ell_{q_2}))_\theta$ unless the underlying von Neumann algebra $\mathcal M$ is commutative, but the corresponding interpolation method for the commutative case plays a key role in realizing classical bootstrap arguments; also, as a fundamental tool, the Littlewood-Paley-Stein square function estimate of sharp growth order for $p < 2$ is itself a quite involved topic in noncommutative analysis. Our proof is consequently more intricate than the classical ones, and involves more modern techniques or ideas from operator space theory, maximal inequalities and noncommutative square function estimates. 
	
	\section*{Pointwise convergence of Fourier series}
	Based on the preceding criteria, we may provide answers to the pointwise convergence problems for a wide class of noncommutative Fourier multipliers. We prove that any countable discrete amenable group $\Gamma$ admits a sequence of finitely supported unital positive definite functions $(m_N)_{N\in \bN}$ such that for all $f\in L_p (\vN)$ with $1<p\leq  \infty$,
	$$T_{m_N} f\to f \text{ b.a.u. (a.u. if $p\geq 2$), as }N\to\infty.$$ 
	The result also holds for general groups $\Gamma$ with the ACPAP (see Chapter~\ref{sec:multipliers on CQG}) when ${f\in L_p (\vN)}$ with $2\leq p\leq  \infty$; these groups form a large subclass of groups with the Haagerup property and the weak amenability. In Chapter~\ref{sec:multipliers on CQG} we will present the result and its proof for general \emph{quantum groups}, and we refer to Theorem \ref{lemma: def of conditionally negative definite function on G} and Corollary \ref{prop:general case for ACPAP case subsequence} for more details.
	
	The reader might wonder if similar results could be established for groups without such approximation properties, such as $\mathrm{SL}(3,\mathbb Z)$ and general lattices in higher rank semisimple Lie groups. However, it is well-known that the group von Neumann algebra $VN(\mathrm{SL}(3,\mathbb Z))$ as well as the associated noncommutative $L_p$-spaces for large finite $p$ do not admit any completely bounded approximation property, and the case for small $p$ close to $2$ is still open (see e.g. \cite{delasallelafforgueduke,delasalleparcetricardduke}). As we explained in the beginning of this introduction, this amounts to saying that the group does not admit any natural summation methods on the the associated noncommutative $L_p$-spaces for large $p$, and that the convergence of Fourier series for the case of $p$ close to $2$ is open even in the sense of $\|\,\|_p$-norm convergence. Therefore the theory of convergences of noncommutative Fourier series on these groups constitutes a completely different flavor, which seems to go beyond the scope of this paper. Nevertheless, Criteria 1 and 2 do not require explicit approximation properties of groups and one may still apply our method to study similar problems for multipliers associated with suitable  $1$-cocycles and conditionally negative definite functions on these groups. 
	
	Our approach to the above result differs greatly from usual strategies in the study of pointwise convergence problems. The key idea is to construct an \emph{abstract} Markov semigroup whose symbols are sufficiently close to $(m_N)_N$ so that Criterion 2 becomes applicable. In hindsight, the construction is essentially inspired by geometric group theory and operator algebras; in particular we would like to mention several related works in this setting  \cite{jolissaintmartin04haagerupsemigroup,casperaskalski15Haagerup,dawsfimaskalskiwhit16haagerup,ciprianisauvageot17amenability}, where an interplay between Fourier multipliers, approximation properties and abstract Markov semigroups has been highlighted. Our method applies to quite general classes of Fourier multipliers as soon as the symbols satisfy a suitable convergence rate. Together with the comments after Criterion 2, our approach might be viewed as an application of ergodic theory of genuinely abstract semigroups to pointwise convergence problems.
		
	Our method is also useful for the study of pointwise convergence of  Dirichlet means in the noncommutative setting. Taking an increasing sequence $(K_N)_{N\in \bN}$ of finite subsets of $\Gamma$, one may consider the partial sums $D_N f = \sum_{g\in K_N} \hat f (g) \lambda (g)$ for $f\in \vN$. As in the classical case, in general $f$ cannot be approximated by  $D_N f $ in the uniform norm $\| \ \|_\infty$ even for elements $f$ in the reduced C*-algebra $C^* _r (\Gamma)$ generated by $\lambda(\Gamma)$. On the other hand, the problem of convergence of $D_N f$ in $L_p$-norms in the noncommutative setting is also very subtle (see e.g. \cite{jungenielsenetal04schauderbasis,bozejkofendler06divergence}). In \cite{bedosconti09twisted,bedosconti12twisted2,chenwang15truncationfourier}, the uniform convergence of $(D_N)_{N\in \bN}$ on some smooth dense subalgebras of $C^* _r (\Gamma)$ was studied. However, if we replace the uniform convergence by the almost uniform one and choose appropriately the family $(K_N)_{N\in \bN}$, it seems that the result can be largely improved; in particular we may obtain the almost uniform convergence for more general measurable operators contained in $L_2 (\vN)$. We show that any countable discrete group $\Gamma$ with the ACPAP admits an increasing sequence $(K_N)_{N\in \bN}$ of finite subsets of $\Gamma$ such that
		$$\sum_{g\in K_N} \hat f (g) \lambda (g) \to f \text{ a.u. as }N\to\infty,\qquad
		f\in L_2 (\vN).$$ 
	 The proof will be given in the general setting of quantum groups in Proposition~\ref{prop:dirichlet}. 
	 
	\section*{Concrete examples}
	Our method is also useful for the study of concrete multipliers in the noncommutative setting. We will consider the following concrete examples in  Chapter~\ref{sec:more ex}:

		i)   \emph{Generalized Fej\'{e}r means}: Consider a $2$-step nilpotent group $\Gamma$ with finite generating set $S$. Let $3/2<p\leq \infty$ and  $f\in L_p (\vN)$, we have 
		$$\sum_{g\in S^N}\frac{|S^N \cap gS^N|}{|S^N|}  \hat f (g) \lambda (g) \to f \ \text{ b.a.u. (a.u. if $p\geq 2$), as }N\to\infty;$$
		the a.u. (or b.a.u.) convergence holds indeed for all $1<p\leq\infty$ if we consider the lacunary subsequence associated with $(S^{2^j})_j$. We will also study the similar means on all non-abelian discrete amenable (quantum) groups in Section \ref{sec:fejer nc}. 
		
		ii)  \emph{Noncommutative Bochner-Riesz means}: Let $\mathbb F $ be a non-abelian free group of finitely many generators and $|\  |$  be its natural word length function. For any $\delta > 1-2/p$ we have 
		$$ \sum_{g\in \Gamma : |g|\leq N} \left(1-\frac{|g|^2}{N^2}\right)^\delta \hat f (g) \lambda(g)\to f \ \text{ a.u. as }N\to\infty,\quad f\in L_p (VN(\mathbb F))$$ 
		for all $2\leq p\leq \infty$. The result also holds for general hyperbolic groups with a conditionally negative word length function.
		
		iii) \emph{Smooth positive definite radial kernels on free groups}: Let $\mathbb F $ and $|\  |$ be as above. 
		Let $\nu$ be an arbitrary positive Borel measure supported on $[-1, 1]$ with $\nu([-1,1])=1$ and write $d\nu_t (x)=  d\nu ( tx)$ for all $t>0$.
		For any $t>0$, set $$m_t( g )=\int_{\mathbb R} x^{|g|} d\nu_t (x-e^{-\frac{2}{t}})=\int_{-1}^1 \left(\frac{y}{ t}+e^{-\frac{2}{t}} \right)^{|g|} d\nu(y),  \qquad g\in\mathbb F.$$
		Then  
		for all $1<p\leq\infty $ and all $f\in L_p (VN(\mathbb F ))$,
		$$ \qquad  T_{m_t} f\to f \text{ b.a.u. (a.u. if $p\geq 2$) as }t\to\infty.$$
		Note that if $\nu$ is the Dirac measure on $0$, then this statement amounts to the almost uniform convergence of Poisson semigroups on $VN(\mathbb F )$ proved in \cite[Theorem 0.3]{jungexu07ergodic}.

	iv) \emph{Dimension free bounds of noncommutative Hardy-Littlewood maximal inequalities}:  Let $B$ be a  symmetric convex body in $\bR^d$ and $\cN$ a semifinite von Neumann algebra. We may consider the operator-valued Hardy-Littlewood operators $\Phi_r :\LpR\to\LpR$ defined by $$\Phi _r(f)(x)=\frac{1}{\mu(B)}\int_{B} f(x-ry)dy.$$ 
	We may identify the Bochner $L_p(\cN)$-valued $L_p$-spaces $L_p(\bR^d;L_p(\cN))$ as a noncommutative $L_p$-space $L_p(L_\infty (\bR^d) \overline{\otimes} \cN)$ associated with the von Neumann algebra $L_\infty (\bR^d) \overline{\otimes} \cN$ and study the associated noncommutative maximal inequalities of $(\Phi_r)_{r>0} $. 
	For classical spaces $L_p(\bR^d)$, Bourgain proved in \cite{bourgain86onhigh,bourgain86ontheLp,bourgain87ondimensionfree} that this kind of  maximal operators satisfy dimension free maximal inequalities.
	The noncommutative version of Hardy-Littlewood maximal inequalities was studied in \cite{mei07operatorvalued} for the case where $B$ is the ball respect to Euclidean metrics and in \cite{hongwangliao2017noncommutative} for general doubling metric spaces. The dimension free bounds in this noncommutative setting were studied by the first author in \cite{hong18hl} for Euclidean balls; because of various difficulties in noncommutative analysis as mentioned before, the general case for convex bodies remained unexplored before our work.  Our main theorems imply as a byproduct the desired maximal inequalities for general convex bodies in $\mathbb R ^d$ with dimension free estimates.
More precisely, 
	we show that there exist constants $C_p>0$ independent of $d$ and $B$ such that the following holds: for any $1<p< \infty$,
		$$\|{\sup_{j\in \bZ}}^+ \Phi _{2^j}(f)\|_p\leq C_p\|f\|_p, \qquad  f\in \LpR,$$
 and for any $\frac{3}{2}<p< \infty$,
		$$\|{\sup_{r> 0}}^+ \Phi _{r}(f)\|_p\leq C_p\|f\|_p, \qquad  f\in \LpR.$$ 

\bigskip	

	The remaining chapters are organized as follows. In Chapter~\ref{sec:preliminaries} we will recall the background and in Chapter~\ref{sec:order} we will prove some preliminary results on noncommutative vector-valued $L_p$-spaces and square function estimates. In Chapter~\ref{sec:maximal} we will establish the key criterion for maximal inequalities of Fourier multipliers, i.e., Criteria 1 and 2. In Chapter~\ref{sec:multipliers on CQG} we will prove several pointwise convergence theorems for Fourier series on quantum groups. Finally in Chapter \ref{sec:more ex} we will establish various examples of noncommutative maximal inequalities and pointwise convergence theorems, as well as the dimension free
	bounds of noncommutative Hardy-Littlewood maximal inequalities.
	
	\medskip
	
	\textbf{Notation:} In all what follows, we write $X\lesssim Y$ if $X \leq  CY$ for an \emph{absolute} constant
	$C > 0$, and $X\lesssim _{\alpha,\beta,\cdots} Y$ if $X \leq  CY$ for a constant
	$C > 0$ depending only on the parameters indicated. Also, we write $X\asymp Y$ if $C^{-1}Y\leq X\leq CY$ for an \emph{absolute} constant $C>0$.

	\subsection*{Acknowledgments}
	The subject of this paper was suggested to the authors by Professor Quanhua Xu several years ago. The authors would like to thank Professor Quanhua Xu for many fruitful discussions. The authors would also like to thank Adam Skalski for  helpful conversations on Proposition \ref{prop:exist multipliers commute with Q}, and thank Ignacio Vergara for explaining  the works \cite{haagerupknudby15LKfreegroups,vergara2019positive}. Part of the work was done in the Summer Working Seminar on Noncommutative
	Analysis at   the Institute for Advanced Study in Mathematics of
	Harbin Institute of Technology (2018 and 2019). GH is partially supported by the Fundamental Research Funds for the Central Universities, NSF of China No. 12071355 and 11601396. SW was partially supported  by the Fundamental Research Funds for the Central Universities No. FRFCUAUGA5710012222, NSF of China No. 12031004, a public grant as part of the FMJH and the ANR grant ANR-19-CE40-0002. XW was partially supported  by the China Scholarship Council, the Projet Blanc I-SITE BFC (contract ANR-15-IDEX-03), the ANR grant ANR-19-CE40-0002 and the CEFIPRA project 6101-1.

\chapter{Preliminaries}\label{sec:preliminaries}
	
	Let $\cM$  denote a semifinite von Neumann algebra equipped with a normal semifinite faithful trace $\tau$.  Let ${{\cS_\cM}}_+$ denote the set of all $x\in \cM_+$ such that $\tau({\supp x})<\infty$, where $\supp x$ denotes the support projection of $x$. Let ${\cS_\cM}$ be the linear span of ${\cS_\cM}_+$. Then $\cS_\cM$ is
	a w*-dense $*$-subalgebra of $\cM$. Given $1\leq  p < \infty$, we define
	$$\|x\|_p=[\tau(|x|)^p]^{1/p}, \qquad x\in {\cS_\cM},$$
	where $|x| = (x^*x)^{1/2}$ is the modulus of $x$. Then $({\cS_\cM}, \|\cdot\|_p)$ is a normed  space, whose completion is the noncommutative $L_p$-space associated with  $(\cM, \tau)$, denoted by $L_p(\cM, \tau)$ or simply by $L_p(\cM)$. As usual, we set $L_\infty(\cM, \tau) =\cM$ equipped with the operator norm. Let $L_0(\cM)$ denote the space of all closed densely defined operators on $H$ measurable with respect to $(\cM, \tau)$, where $H$ is the Hilbert space on which $\cM$ acts. Then the elements of $L_p(\cM)$ can be viewed as closed densely defined operators on $H$. A more general notion of Haagerup $L_p$-spaces on arbitrary von Neumann algebras can be found in Subsection~\ref{sect:nontracial matrix}.   We refer to \cite{pisierxu03nclp} for more information on noncommutative $L_p$-spaces.
	We   say that a map $T:L_p(\mathcal M, \tau)\to L_p(\mathcal M, \tau)$ is \emph{$n$-positive} (resp. \emph{$n$-bounded}) for some $n\in \mathbb N$ if $T\otimes \mathrm{id}_{M_n}$ extends to a positive (resp. bounded) map on $ L_p(\mathcal M \otimes M_n, \tau \otimes \mathrm{Tr}) $, where $M_n$ denotes the algebra of all $n\times n$ complex matrices and $\mathrm{Tr}$ denotes the usual trace on it, and we say that $T$ is \emph{completely positive} (resp. \emph{completely bounded}) if it is $n$-positive  (resp.  $n$-bounded) for all $n\in \mathbb N$. We will denote by $\|T\|_{cb}$ the supreme of the norms of $T\otimes \mathrm{id}_{M_n}$ on $ L_p(\mathcal M \otimes M_n, \tau \otimes \mathrm{Tr}) $ over all $n\in\mathbb N$.
	
	
	%
	
	\section{Noncommutative \texorpdfstring{$\ell_\infty$}{l infinite}-valued \texorpdfstring{$L_p$}{Lp}-spaces} 
	In classical analysis, the pointwise properties of measurable functions are often studied by estimating the norms of maximal functions of the form $\| \sup_n |f_n| \|_p$.
	However, these maximal norms in the noncommutative setting require a specific definition, since $\sup_n |x_n|$ does not make  sense for a sequence $(x_n)_n$ of operators. This difficulty
	is overcome by considering the spaces $L_p(\cM;\ell_\infty)$, which are the noncommutative analogs of the
	usual Bochner spaces $L_p(X;\ell_\infty)$. These spaces were first introduced by Pisier \cite{pisier98noncommutativeLp} for injective von Neumann algebras and then extended to general von Neumann algebras by Junge \cite{junge02doob}. See also \cite[Section 2]{jungexu07ergodic} for more details.
	
	Given $1\leq p\leq \infty$, we define $L_p(\cM;\ell_\infty)$ to
	be the space of all sequences $x = (x_n)_{n\in \bN}$ in $L_p(\cM)$ which admit a factorization of the following form: there exist $a, b\in L_{2p}(\cM)$ and a bounded sequence $y=(y_n)\subset \cM$   such that
	$$x_n = ay_nb, \qquad  n\in \bN.$$
	The norm of $x$ in $L_p(\cM;\ell_\infty)$ is given by 
	$$\|x\|_{L_p(\cM;\ell_\infty)}=\inf\left\{\|a\|_{2p}\sup_{n\in  \bN}\|y\|_\infty\|b\|_{2p}    \right\}$$
	where the infimum runs over all factorizations of $x$ as above. We will adopt the convention that the norm $\|x\|_{L_p(\cM;\ell_\infty)}$ is denoted by $\|\sup_n^+ x_n\|_p$. As an intuitive description, it is worth remarking that a selfadjoint  sequence $(x_n)_{n\in \bN}$ of $L_p(\cM)$ belongs to $L_p(\cM;\ell_\infty)$ if and only if there exists a positive element $a\in L_p(\cM)$ such that $-a\leq x_n\leq a$ for any $n\in \bN$. In this case, we have
	\begin{equation}\label{eq: sup+ norm for positive sequence}
		\|{\sup_{n\in \bN}}^+ x_n\|_p=\inf\{\|a\|_p : a\in L_p(\cM)_+, -a\leq x_n\leq a, \forall n\in \bN   \}.
	\end{equation}
	The subspace $L_p(\cM, c_0)$ of $L_p(\cM;\ell_{\infty})$ is defined as the family of all sequences $(x_n)_{n\in \bN}\subset L_p(\cM)$ such that there are $a,b\in L_{2p}(\cM)$ and $(y_n)\subset \cM$ verifying 
	$$x_n=ay_n b\add{and} \lim _{n\to \infty} \|y_n\|_\infty=0.$$
	It is easy to check that $L_p(\cM,c_0)$ is a closed subspace of $L_p(\cM;\ell_\infty)$. It is indeed the closure of the subspace of all finitely supported sequences.
	

	On the other hand, we may also consider the space $L_p(\cM;\ell_\infty^c)$ for ${2\leq p\leq \infty}$. This space $L_p(\cM;\ell_\infty^c)$ is defined to be the family of all sequences $(x_n)_{n\in \bN}\subset L_p(\cM)$  which admits $a\in L_p(\cM)$ and $(y_n)\subset L_{\infty}(\cM)$ such that 
	$$x_n=y_n a\add{and}\sup_{n\in \bN}\|y_n\|_\infty<\infty.$$
	$\|(x_n)\|_{L_p(\cM;\ell_\infty^c)}$ is then defined to be the infimum of $ \sup_{n\in \bN} \|y_n\|_\infty \|a\|_p   $ over all factorization of $(x_n)$ as above. It is easy to check that $\|\  \|_{L_p(\cM;\ell_\infty^c)}$ is a norm, which makes $L_p(\cM;\ell_\infty^c)$ a Banach space. Moreover, ${(x_n)\in L_p(\cM;\ell_\infty^c)}$ if and only if $(x_n^*x_n)\in L_{p/2}(\cM;\ell_{\infty})$. Indeed, we have 
	\begin{equation}\label{eq:relation between l infty c and l infty}
	\|(x_n)\|_{L_p(\cM;\ell_{\infty}^c)}=	\|(x_n^*x_n)\|_{L_{p/2}(\cM;\ell_{\infty})}^{1/2}.
	\end{equation}
	We define similarly the subspace $L_p(\cM;c_0^c)$ of $L_p(\cM;\ell_\infty^c)$.
	
	We define the Banach space $L_p(\cM;\ell_\infty^r):=\{(x_n):  (x_n^*)\in  L_p(\cM;\ell_\infty^c) \}$ for $2\leq p\leq \infty$ with the norm $\|(x_n)\|_{L_p(\cM;\ell_\infty^r)}=\|(x_n^*)\|_{L_p(\cM;\ell_\infty^c)}$. The following interpolation theorem was firsted studied by Pisier in \cite{pisier96OH} and then generalized by Junge and Parcet in \cite{jungeparcet10mixed}.

	\begin{lemma}[{\cite[Theorem A]{jungeparcet10mixed}}]\label{theorem:interpolation l infty}
		For any $2\leq p\leq \infty$, we have isometrically
		$$L_p(\cM; \ell_{\infty})=\left(L_p(\cM;\ell_\infty^c), L_p(\cM;\ell_\infty^r)    \right)_{1/2}.$$
	\end{lemma}

	Another Banach space $L_{p}(\cM ;\ell_1)$ is also defined in \cite{junge02doob}. 
	Given ${1\leq p\leq \infty}$, a sequence $x={(x_n)_{n\in \bN}}$ belongs to $L_p(\cM;\ell_1)$ if there are $u_{kn}, v_{kn} \in L_{2p}(\cM)$ such that 
	$$x_n=\sum_{k\geq 0} u_{kn}^* v_{kn}, \qquad   n\geq 0$$
	and 
	$$\|(x_n )_n\|_{L_p(\cM;\ell_1)}:=\inf \left\{ \left\|\sum_{k,n\geq 0} u_{kn}^*u_{kn}\right\|_{p}^{1/2} \left\|\sum_{k,n\geq 0} v_{kn}^*v_{kn}\right\|_{p}^{1/2}  \right\}<\infty.$$ 
	Specially, for a positive sequence $x=(x_n)$, we have 
	$$\|x\|_{L_p(\cM;\ell_1)}=\|\sum_{n\geq 0}x_n\|_p.$$
	
	The following proposition will be useful in this paper.
	\begin{prop}[\cite{jungexu07ergodic,jungeparcet10mixed}]\label{prop: prop of LMinfty and LMl1}
		Let $1\leq p,p^\prime\leq \infty$ and $1/p + 1/p^\prime =1$.
		
		\emph{(1)}  $L_p(\cM;\ell_\infty)$ is the dual space   of $L_{p^\prime}(\cM;\ell_1)$ when $p^{\prime}\neq\infty$. 
		The duality bracket is given by 
		$$\langle x, y\rangle=\sum_{n\geq 0}\tau(x_ny_n), \quad x\in L_p(\cM;\ell_\infty),\ y\in L_{p^\prime}(\cM;\ell_1).$$
		In particular for any positive sequence $(x_n)_n$ in $L_p(\cM;\ell_\infty)$, we have
		$$\|{\sup_n}^+ x_n\|_p=\sup \{\sum_n\tau(x_ny_n): y_n\in L_{p^\prime}^+(\cM) \text{ and } \|\sum_n y_n\|_{p^\prime}\leq 1   \}.$$
		
		\emph{(2)} Each element in the unit ball of $L_p(\cM;\ell_\infty)$ (resp. $L_p(\cM;\ell_1)$) is a sum of sixteen (resp. eight) positive elements in the same ball.
		
		\emph{(3)} Let $1\leq p_0<p<p_1\leq \infty$ and $0<\theta<1$ be such that $\frac{1}{p}=\frac{1-\theta}{p_0}+\frac{\theta}{p_1}$. Then the following complex interpolation holds: we have isometrically
		$$\Lpinfty{\cM}=\left( L_{p_0}(\cM;\ell_\infty),L_{p_1}(\cM;\ell_\infty)\right)_\theta .$$
		Similar complex interpolations also hold  for $L_{p}(\cM;\ell_\infty^c)$ with $2\leq p\leq \infty$.
	\end{prop}



	\begin{remark}\label{prop:subsequence approxiamte to whole sequence}
	Note that in the above definitions of the spaces $L_p(\cM;\ell_\infty)$, $L_p(\cM; c_0)$, $L_p(\cM; c_0^c)$, $L_p(\cM;\ell_\infty^c)$ or $L_p(\cM;\ell_\infty^r)$, we defined the corresponding norms  only for a family of operators with a countable index set. In fact, we may also define them for a family with  any uncountable index set $I$, by
		$$\left\|\mathop{{\sup_{i\in I}}^{+}}x_{i}\right\|_p= \sup_{(i_n)_{n\in \bN}\subset I} \left\|\mathop{{\sup_{n\in \bN}}^{+}}x_{i_n}\right\|.$$
	Then we denote the associated spaces by $L_p(\cM;\ell_\infty(I))$, $L_p(\cM; c_0(I))$, $L_p(\cM; c_0^c(I))$, $L_p(\cM;\ell_\infty^c(I))$ or $L_p(\cM;\ell_\infty^r(I))$ respectively, or simply by $L_p(\cM; \ell_\infty )$, $L_p(\cM; c_0 )$ and so on if no confusion can occur.
	The above proposition still holds for these spaces. 
	
		It is known that a family $(x_{i})_{i\in I}\subset L_{p}(\mathcal{\mathcal{M}})$, whether $I$ is  countable or uncountable,
		belongs to $L_{p}(\mathcal{\mathcal{M}};\ell_{\infty})$ if and only	if 
		\begin{eqnarray*}
			\sup_{J\subset I\text{ finite}}\Big\|\mathop{{\sup_{i\in J}}^{+}}x_{i}\Big\|_{p} & < & \infty,
		\end{eqnarray*}
		and in this case  
		\begin{equation}\label{eq:sup+ is sup of finite set}
		\Big\|\mathop{{\sup_{i\in I}}^{+}}x_{i}\Big\|_{p}=\sup_{J\text{ finite}}\Big\|\mathop{{\sup_{i\in J}}^{+}}x_{i}\Big\|_{p}.
		\end{equation}
		Similar observations hold for $L_p(\cM; \ell_\infty^c )$. As a consequence, for any   $1\leq p< \infty$ and any $(x_t)_{t\in \mathbb R_+}\in L_p(\cM; \ell_\infty)$ such that the map $t\mapsto x_{t}$ from $\mathbb R _+$ to $L_p(\cM )$ is continuous,  we have
		$$\| (x_t)_{t\in \bR_+}\|_{\Lpinfty{\cM}}=\lim_{a\to 1^+}\|( x_{a^j})_{j\in \bZ}\|_{\Lpinfty{\cM}}.$$
	To see this, we note that  $\| (x_t)_{t\in \bR_+}\|_{\Lpinfty{\cM}}\geq \limsup_{a\to 1^+}\|( x_{a^j})_{j\in \bZ}\|_{\Lpinfty{\cM}}$; thus by \eqref{eq:sup+ is sup of finite set}  it suffices to show that for any (finitely many) elements $t_1,\ldots , t_n,$ $$\|  (x_{t_k})_{1\leq k\leq n}\|_{\Lpinfty{\cM}}\leq \liminf_{a\to 1^+}\|( x_{a^j})_{j\in \bZ}\|_{\Lpinfty{\cM}}.$$ 
		This is obvious since for any $\varepsilon>0$, by continuity we may find a scalar $a_0\in \bR_+$ sufficiently close to $1$ such that for all $1\leq k\leq n$,  $\|x_{t_k} - x_{a_0^{j_k}}\|_p <\varepsilon/n$ with some $j_k \in \mathbb Z$ and 
		$$\|( x_{a_0^j})_{j\in \bZ}\|_{\Lpinfty{\cM}}\leq \liminf_{a\to 1^+}\|( x_{a^j})_{j\in \bZ}\|_{\Lpinfty{\cM}}+\varepsilon,$$
		which implies
		\begin{align*}
			\|  (x_{t_k})_{1\leq k\leq n}\|_{\Lpinfty{\cM}} 
			&\leq \|  (x_{a_0^{j_k}})_{1\leq k\leq n}\|_{\Lpinfty{\cM}} + \sum_{1\leq k\leq n} \|x_{t_k} - x_{a_0^{j_k}}\|_p\\
			& \leq \liminf_{a\to 1^+} \|  (x_{a^{j }})_{j\in\mathbb Z}\|_{\Lpinfty{\cM}} + 2\varepsilon.
		\end{align*} 
		Similarly, for $2\leq p<\infty$, we have
		$$\| (x_t)_{t\in \bR_+}\|_{\Lpinftyc{\cM}}=\lim_{a\to 1^+}\|( x_{a^j})_{j\in \bZ}\|_{\Lpinftyc{\cM}}.$$
	\end{remark}


	\section{Maximal inequalities and pointwise convergence} \label{sec:max and point}
	The standard tool in the study of pointwise convergence is the following type of \emph{maximal inequalities}. 
	\begin{definition}
		Let $1\leq p\leq \infty$. We consider a family of linear maps ${\Phi_n :L_p(\cM)\to L_0(\cM)}$ for $n\in \mathbb N$.  
		
		(1)  We say that $(\Phi_n)_{n\in \bN}$ is of \textit{strong type $(p, p)$} with constant $C$ if  
		$$\|{\sup_{n \in \bN} }^+ \Phi_n(x)\|_p\leq C\|x\|_p,\qquad x\in L_p(\cM).$$
		
		(2)   We say that $(\Phi_n)_{n\in \bN}$ is of \textit{weak type $(p, p)$} ($p<\infty)$ with constant $C$ if  for any $x\in L_p(\cM)$ and any $\alpha>0$ there is a projection $e\in \cM$ such that
		$$\|e\Phi_n(x)e\|_\infty\leq \alpha, \quad    n\in \bN \add{and} \tau(e^\perp)\leq \left[ C\frac{\|x\|_p}{\alpha}\right]^p.$$
		
		(3) We say that $(\Phi_n)_{n\in \bN}$ is of \textit{restricted weak type $(p,p)$} ($p<\infty)$ with constant $C$ if   for any projection $f\in  \cM $ and any $\alpha>0$, there is a projection $e\in \cM$ such that 
		$$\|e\Phi_n(f)e\|_\infty \leq \alpha \quad   n\in \bN      \quad \text{ and }\quad  \tau(e^\perp)\leq  \left( \frac{C }{\alpha}\right) ^p \tau(f).$$
	\end{definition}
	It is easy to see that for any $1<p<\infty$,
	$$\text{ strong type } (p, p)\Rightarrow\text{ weak type } (p, p)\Rightarrow\text{ restricted weak type } (p, p).$$
	
	Here is a simple but useful proposition.	
	\begin{prop}\label{prop:positive maps Cp=A(p infty)}
		Let $(\Phi_n)_{n\in \bN}$ be a sequence of positive linear  maps on $L_p(\cM)$. Then 
		$$\|(\Phi_n)_n:L_p(\cM;\ell_{\infty})\to L_p(\cM;\ell_{\infty})\|\asymp \|(\Phi_n)_n:L_p(\cM)\to L_p(\cM;\ell_{\infty})\|.$$
	\end{prop}	
	\begin{proof}
		By setting $x_n=x$, it is obvious to see that 
		$$\|(\Phi_n)_n:L_p(\cM)\to L_p(\cM;\ell_{\infty})\|\leq\|(\Phi_n)_n:L_p(\cM;\ell_{\infty})\to L_p(\cM;\ell_{\infty})\|.$$
		For the inverse direction, we consider positive elements first. Take an element ${(x_n)_n\in L_p(\cM;\ell_{\infty})_+}$. For any $\varepsilon>0$, by \eqref{eq: sup+ norm for positive sequence}, we can find an element $a\in L_p(\cM)_+$ such that, 
		$$0\leq x_n\leq a,\ \forall n \quad\text{and}\quad  \|a\|_p\leq 	\|{\sup_{n\in \bN}}^+ x_n\|_p+\varepsilon . $$
		By linearity and positivity of $(\Phi_n)_n$, we have $0\leq \Phi_nx_n\leq \Phi_n a $.
		Therefore
		$$\|{\sup_n}^+ \Phi_nx_n\|\leq \|{\sup_n}^+ \Phi_na\|\leq  \|(\Phi_n)_n:L_p(\cM)\to L_p(\cM;\ell_{\infty})\|(\|{\sup_n}^+ x_n\|_p+\varepsilon). $$
		Thus, by arbitrariness of $\varepsilon$ and Proposition~\ref{prop: prop of LMinfty and LMl1}~(2), we get 
		\begin{equation*}
			\|(\Phi_n)_n:L_p(\cM;\ell_{\infty})\to L_p(\cM;\ell_{\infty})\|\leq 16  \|(\Phi_n)_n:L_p(\cM)\to L_p(\cM;\ell_{\infty})\|. \qedhere 
		\end{equation*}
	\end{proof}
	
	The Marcinkiewicz interpolation theorem  plays an important role in the study of   maximal inequalities. Its analogue for the noncommutative setting was first established by Junge and Xu in \cite{jungexu07ergodic}, and then was generalized in \cite{bekjanchenoscekowski17ncmaximal} and \cite{dirksen15interpolation}. We present Dirksen's version here. 
	\begin{theorem}[{\cite[Corollary 5.3]{dirksen15interpolation}}]\label{theorem:interpolation of dirksen}
		Let $1\leq p< r<q\leq \infty$. Let $(\Phi_n)_{n\in \bN}$ be a family of positive linear maps from $L_p (\cM) + L_q (\cM)$ into $L_0  (\cM)$. If $(\Phi_n)_{n\in \bN}$ is of restricted weak type $(p, p)$ and  of {strong type $(q, q)$} with constants $C_p$ and $C_q$, then it is of strong type $(r,r)$ with constant
		$$C_r\lesssim \max\{C_p, C_q\}(\frac{rp}{r-p}+\frac{rq}{q-r})^2.$$
	\end{theorem}

	We need an appropriate analogue for the noncommutative setting of the usual almost everywhere convergence. This is the notion of almost uniform convergence introduced by Lance \cite{lance76ergodic}.
	\begin{definition}\label{def:auconvergence}
		Let $x_n, x\in L_0(\cM)$. $(x_n)_{n\in \bN}$ is said to converge \textit{almost uniformly} (\emph{a.u.} in short) to $x$ if for any $\varepsilon>0$ there is a projection $e\in \cM$ such that 
		$$\tau(e^\perp)<\varepsilon \add{and} \lim _{n\to \infty}\|(x_n-x)e\|_\infty=0.$$
		$(x_n)_{n\in \bN}$ is said to converge \textit{bilaterally almost uniformly} (\emph{b.a.u.} in short) to $x$ if for any $\varepsilon>0$ there is a projection $e\in \cM$ such that 
		$$\tau(e^\perp)<\varepsilon \add{and} \lim _{n\to \infty}\|e(x_n-x)e\|_\infty=0.$$
	\end{definition}
	It is obvious that the a.u. convergence implies the b.a.u. convergence, so we will be mainly interested in the former. Note that in the commutative case,  both notions  are equivalent to the usual almost everywhere convergence in terms of Egorov's Theorem in the case of probability space. 

	It is nowadays a standard method of deducing  pointwise convergence from maximal inequalities. We will use the following  facts.
	
	\begin{lemma}[\cite{defantjunge04maximal}]\label{lemma:Lp M c0 imply b.a.u}
		\emph{(1)} If a family $(x_i)_{i\in I}$ belongs to $L_p(\cM, c_0)$ with  some $1\leq p<\infty$, then $x_i$ converges b.a.u. to $0$.
		
		\emph{(2)} 	If a family $(x_i)_{i\in I}$ belongs to $L_p(\cM, c_0^c)$ with some $2\leq p<\infty$, then $x_i$ converges a.u. to $0$.
	\end{lemma}

	\begin{prop}\label{prop:Phix-x in Lp M C0}
		\emph{(1)} Let $1\leq p<\infty$ and  $(\Phi_n)_{n\in \bN}$ be a sequence of positive linear maps on $L_p(\cM) $. Assume that $(\Phi_n)_{n\in \bN}$ is of weak type $(p, p)$. If  $(\Phi_nx)_{n\in \bN}$ converges b.a.u. to $0$ for all elements $x $ in a dense subspace of $L_p(\cM) $, then $(\Phi_nx)_{n\in \bN}$ converges b.a.u. for all $x\in L_p(\cM)$.

		\emph{(2)}  
		Let $1\leq p <\infty$ and $(\Phi_n)_{n\in \bN}$ be  a sequence of linear maps on $L_p(\mathcal M )$. 
		Assume that $(\Phi_n)_{n\in \bN}$ satisfies the following \emph{one sided weak type $(p, p)$} maximal inequalities, i.e. there exists $C>0$ such that for any $x\in L_p(\cM)$ and $\alpha>0$ there exists a projection $e\in \cM$ such that
		\begin{equation}\label{eq:one sided weak}
		 \|\Phi_n(x)e\|_\infty\leq \alpha, \quad    n\in \bN \add{and} \tau(e^\perp)\leq \left[ C\frac{\|x\|_p}{\alpha}\right]^p.
		\end{equation}
		If  $(\Phi_nx)_{n\in \bN}$ converges a.u. to $0$ for all elements $x $ in a dense subspace of $L_p(\cM) $, then $(\Phi_nx)_{n\in \bN}$ converges a.u. for all $x\in L_p(\cM)$.
%
	\end{prop}
	\begin{proof}
		The assertion (1) is given by \cite[Theorem 3.1]{chilinlitvinov2016ergodic}.  
		
		The second part  is standard and is implicitly established in the proof of \cite[Remark 6.5]{jungexu07ergodic} and \cite[Theorem 5.1]{chenxuyin13harmonic} 
		 for which we provide a brief argument  for the convenience of the reader. Let $x\in L_p(\cM)$ and $\varepsilon>0$. For any $m\geq 1$, take $y_m\in L_p(\cM)$ such that $\|x-y_m\|_p< 2^{-2m/p}C^{-1}\varepsilon^{1/p}$   and $(\Phi_n y_m)_n$ converges a.u. to $0$  as $n \to \infty$.
	Denote $z_m=x-y_m$.	By the estimation of one side weak type $(p, p)$, we may find a projection $e_m\in \cM$ such that
		$$\sup_n\|\Phi_n (z_m)e_m\|_\infty \leq 2^{-m/p} \add{and}  \tau(e_m^\perp)\leq \left[C\frac{\|z_m\|_p}{2^{-m/p}}\right]^p<2^{-m}\varepsilon .$$
	We may also find a projection $f_m \in \cM$ such that 
	$$\tau(f_m^\perp)<2^{-m}\varepsilon \add{ and } \lim_{n\to \infty}\|\Phi_n(y_m)f_m\|_\infty=0.$$
	Let $e=\bigwedge_m (e_m \wedge f_m)$. Then 
	$$\tau(e^\perp)\leq \sum_{m\geq 1} (\tau(e_m^\perp)+\tau(f_m^\perp))< \varepsilon$$  and for any $m\geq1$,	
	$$\limsup_{n\to \infty}\|\Phi_n(x)e\|_\infty \leq\lim_{n\to \infty}(\|\Phi_n(y_m)f_m\|+\|\Phi_n(z_m)\|)\leq 2^{-m/p},$$
	which means that $\lim_{n\to \infty}\|\Phi_n(x)e\|_\infty =0$. Therefore, $\Phi_n(x)$ converge a.u. to $0$.
	\end{proof}

	We recall the following well-known fact, which is of essential use for our arguments. The following maximal inequalities and the pointwise convergence on $L_p$-spaces are given in \cite{jungexu07ergodic}. 
	We recall that a map $T$ is said to be \emph{symmetric} if $\tau(T(x)^*y)=\tau(x^*T(y))$ for any $x, y \in \cS_\cM$.
	\begin{proposition}\label{prop:maximal inequality of semigroup}
		Let $(S_t)_{t\in \bR_+}$ be a semigroup  of unital completely positive  trace preserving and symmetric maps on $\mathcal M$. We have
		$$\| (S_t(x))_t\|_{L_p(\cM;\ell_{\infty})} \leq c_p\|x\|_p, \qquad   x\in L_p(\cM),\quad 1<p< \infty,$$
		and	$$\| (S_t(x))_t\|_{L_p(\cM;\ell_{\infty}^c)} \leq \sqrt{c_p}\|x\|_p, \qquad   x\in L_p(\cM),\quad 2<p< \infty,$$
		where $c_p\leq C p^2(p-1)^{-2}$ with $C$ an absolute constant. 
	\end{proposition}
	

\chapter[Square function estimates]{Noncommutative Hilbert space valued \texorpdfstring{$L_p$}{Lp}-spaces and square function estimates}\label{sec:order}
In this chapter we will collect some preliminary results on noncommutative square functions, which are among the essential tools in this paper. Some of the results proved in this chapter might be folkloric for experts, but we include them here for the convenience of the reader. 

The noncommutative Hilbert space valued $L_p$-spaces provide a suitable framework for studying square functions in the noncommutative setting. In this paper we will only use the following concrete representations of these spaces; for a more general description we refer to the papers \cite{lustpisier91nonKhintchine,lustpiquard86khintchine,jungelemerdyxu06Hinftycalut}.

First, for a finite sequence $(x_n)_n\subset L_p(\cM)$, we define
$$\|(x_n)\|_{L_p(\cM;\ell_2^c)} =\left\|\left(\sum_n x_n^*x_n\right)^{1/2}\right\|_p $$
and $$ \|(x_n)\|_{L_p(\cM;\ell_2^r)} =\left\|\left(\sum_n x_nx_n^*\right)^{1/2}\right\|_p.$$
We alert the reader that the two norms above are not comparable at all if $p\neq 2$. 
Let $L_p(\cM;\ell_2^c)$ (resp. $L_p(\cM;\ell_2^r)$ ) be the completion of the space of all finite sequences in $L_p(\cM)$ with respect to $\|\ \|_{L_p(\cM;\ell_2^c)}$ (resp. $\|\ \|_{L_p(\cM;\ell_2^r)}$). The space $L_p(\cM;\ell_2^{cr})$ is defined in the following way.  If $ 2\leq p\leq \infty$, we set
$$\Lplcr{\cM}=\Lplc{\cM}\cap \Lplr{\cM}$$
equipped with the   norm 
$$\|(x_n)\|_{\Lplcr{\cM}}=\max \{\|(x_n)\|_{\Lplc{\cM}}, \|(x_n)\|_{\Lplr{\cM}}\}.$$
If $1\leq p\leq 2$, we set
$$\Lplcr{\cM}=\Lplc{\cM}+ \Lplr{\cM}$$
equipped with the   norm 
$$\|(x_n)\|_{\Lplcr{\cM}}=\inf \{\|(y_n)\|_{\Lplc{\cM}}+\|(z_n)\|_{\Lplr{\cM}}\}$$
where the infimum runs over all decompositions $x_n=y_n+z_n$ in $L_p(\cM)$. 

Second, for the Borel measure space $(\mathbb R _+ \setminus \{0\} , \frac{dt}{t} )$,  we may consider the norms  
$$\|(x_t)_{t }\|_{L_p(\cM;L_2^{c}(\frac{dt}{t}))}=\left\|\left(\int_0^\infty   x_t^*x_t \frac{dt}{t} \right)^{1/2} \right\|_{p}$$ and
$$\|(x_t)_{t }\|_{L_p(\cM;L_2^{r}(\frac{dt}{t}))}=\left\|\left(\int_0^\infty   x_t x_t^* \frac{dt}{t} \right)^{1/2} \right\|_{p}. $$
We refer to \cite[Section 6.A]{jungelemerdyxu06Hinftycalut} for the rigorous meaning of the integral appeared in the above norm. Then we may define the spaces  $ L_p(\cM;L_2^{c}(\frac{dt}{t}))$, $L_p(\cM;L_2^{r}(\frac{dt}{t}))$ and $L_p(\cM;L_2^{cr}(\frac{dt}{t}))$ in a similar way.

We recall the following basic properties.
\begin{proposition}
	\emph{(1) (Duality)}  Let $1\leq p<\infty$ and $p^\prime$ such that $1/ p^\prime +1/p =1$. Then 
	$$(\Lplc{\cM})^*=L_{p^\prime}(\cM;\ell_2^{r}),\quad (\Lplr{\cM})^*=L_{p^\prime}(\cM;\ell_2^{c})$$
	and
	$$ (\Lplcr{\cM})^*=L_{p^\prime}(\cM;\ell_2^{cr}).$$ 
	The duality bracket is given by 
	$$\langle (x_n)_n, (y_n)_n \rangle = \sum_n \tau(x_n y_n),\qquad (x_n)_n\subset L_p(\cM),\  (y_n)_n\subset L_{p^\prime}(\cM ).$$
	
	
	\emph{(2) (Complex interpolation \cite{pisier82holomor})} Let $1\leq p,q\leq \infty$ and $0<\theta<1$. Let $\frac{1}{r}=\frac{1-\theta}{p}+\frac{\theta}{q}$. Then we have the isomorphism with absolute constants
	$$\left( \Lplcr{\cM}, L_{q}(\cM; \ell_2^{cr})\right)_\theta=L_{r}(\cM; \ell_2^{cr}). $$ 
	Similar complex interpolations also hold for $\Lplc{\cM}$ and $\Lplr{\cM}$.
	
\end{proposition}

A sequence of independent random variables $(\varepsilon_n)$ on a probability space $(\Omega, P)$ is called a \emph{Rademacher} sequence if $P(\varepsilon_n=1)=P(\varepsilon_n=-1)=\frac{1}{2}$ for any $n\geq 1$.
The following noncommutative Khintchine inequalities are well-known.
\begin{proposition}[\cite{lustpiquard86khintchine,lustpisier91nonKhintchine,pisier98noncommutativeLp}]\label{prop:Khintchine}
	Let $(\varepsilon_n)$ be a  Rademacher sequence on a probability space $(\Omega, P)$. 
	Let $1\leq p<\infty$ and $(x_n)$ be a  sequence in $L_p(\cM;\ell_2^{cr})$. 
	
	\emph{(1)} If $1\leq p\leq 2$, then there exists an absolute constant $c>0$ such that
	$$c\|(x_n)_n\|_{\Lplcr{\cM}} \leq \|\sum_n \varepsilon_n x_n\|_{L_p(\Omega;L_p(\cM))}\leq \|(x_n)_n\|_{\Lplcr{\cM}}.$$
	
	\emph{(2)} If $2\leq p< \infty,$ then there exists an absolute constant $c>0$ such that
	$$\|(x_n)_n\|_{\Lplcr{\cM}}\leq \|\sum_n \varepsilon_n x_n\|_{L_p(\Omega;L_p(\cM))} \leq c \sqrt{p} \|(x_n)_n\|_{\Lplcr{\cM}}.$$
\end{proposition}

The following proposition will be useful for our further studies. 
\begin{prop}\label{claim: sup+ < CRp norm}
	Let $(x_n)_{n\in\bN}\in L_p(\cM;\ell_\infty)$. Then there exists an absolute constant $c>0$ such that		for any $1\leq p\leq \infty$,
	$$\| (x_n)_n\|_{L_p(\cM; \ell_{\infty})}\leq  \|(x_n)_n\|_{L_p(\cM; \ell_2^{cr})};$$ 
	and for any $2\leq p\leq \infty$,
	$$\| (x_n)_n\|_{L_p(\cM; \ell_{\infty}^c)}\leq \|(x_n)_n\|_{L_p(\cM; \ell_2^{c})}.$$
\end{prop}
\begin{proof}
	We start with the proof of the first inequality.
	It is trivial for the case $p=\infty$:
	$$\|(x_n)_n\|_{L_\infty(\cM;\ell_{\infty})}=\sup_n \|x_n\|_\infty\leq \left\|\left( \sum_n x_n^*x_n \right)^{1/2} \right\|_\infty\leq\|(x_n)_n\|_{L_\infty(\cM;\ell_2^{cr})}.$$
	We note that by the H\"{o}lder inequality (see also \cite[Lemma 3.5]{junge02doob}), for any ${1\leq p\leq \infty}$ and for any sequence $(x_n)_n$ in $L_p(\cM;\ell_1)$,
	$$\|\sum_n x_n\|_p\leq \|(x_n)_n\|_{L_p(\cM; \ell_1)}.$$
	On the other hand, for  any $1\leq p\leq \infty$, by the definition of $\|\ \|_{L_p(\cM;\ell_1)}$, one can easily get  $\|(\varepsilon_n x_n)\|_{L_p(\cM; \ell_1)}=\|( x_n)\|_{L_p(\cM; \ell_1)}$ with  $\varepsilon_n\in \{\pm1\}$.
Now we set $(\varepsilon_n)_n$ to be a Rademacher sequence on a probability space $(\Omega, P)$. 
It is folkloric that 
\begin{align*}
\|(x_n)\|_{L_\infty (\cM; \ell_2^{cr})}&=\|(\varepsilon_n x_n)\|_{L_\infty (\cM; \ell_2^{cr})} \leq \|\sum_n \varepsilon_nx_n\|_{L_2(\Omega;L_\infty(\cM))}\\
&\leq \|\sum_n \varepsilon_nx_n\|_{L_\infty(\Omega;L_\infty(\cM))}
=\sup_{\omega \in \Omega} \left\|\sum_n \varepsilon_n(\omega) x_n\right\|_\infty\\
&\leq \sup_{\omega \in \Omega}\|(\varepsilon_n(\omega) x_n)_n\|_{L_\infty(\cM;\ell_1)}=\|( x_n)_n\|_{L_\infty(\cM;\ell_1)}.
\end{align*}	
%
Let $(y_n)_n\in L_1(\cM;\ell_\infty)$. By duality and the above inequality we have
	\begin{align*}
		\|(y_n)_n\|_{L_1(\cM; \ell_{\infty})}&=\sup \left\{\frac{\sum_n \tau(x_ny_n)}{\|(x_n)\|_{L_{\infty}(\cM,\ell_1)}}:   (x_n)\in  L_{\infty}(\cM,\ell_1)      \right\}\\
		&\leq \sup \left\{ \frac{\sum_n \tau(x_ny_n)}{\|(x_n)\|_{L_{\infty}(\cM; \ell_2^{cr})} } : (x_n)\in  L_1(\cM; \ell_2^{cr})      \right\}\\
		&\leq \|(y_n)\|_{L_1(\cM; \ell_2^{cr})}.
	\end{align*}
By interpolation we immediately get the first inequality in the lemma.
	
	The above arguments  tell  that   $\|\ \|_{L_p(\cM;\ell_{\infty})}\leq \|\  \|_{L_p(\cM;\ell_{2}^{cr})}$ for $1\leq p\leq \infty$.  
	As before, by a duality argument we indeed get
$$\|\ \|_{L_p(\cM;\ell_{\infty})}\leq 	\|\  \|_{L_p(\cM;\ell_{2}^{cr})} \leq \|\ \|_{L_p(\cM;\ell_{1})}.$$
	Therefore, we obtain the second inequality:
	\begin{align*}
		&\quad\  \|(x_n)\|_{L_p(\cM;\ell_{\infty}^c)}\\
		&=	\|(x_n^*x_n)\|_{L_{p/2}(\cM;\ell_{\infty})}^{1/2} \leq \|(x_n^*x_n)\|_{L_{p/2}(\cM;\ell_1)}^{1/2}\\
		&=\left\|\left(\sum_n x_n^*\otimes e_{1,n}\right) \left(\sum_n x_n\otimes e_{n,1}\right)\right\|_{L_{p/2}(\cM\otimes B(\ell_2))}^{1/2}\\
		&\leq \left\|\left(\sum_n x_n\otimes e_{n,1}\right)^* \right\|_{L_{p}(\cM\otimes B(\ell_2))}^{1/2}\left\|\left(\sum_n x_n\otimes e_{n,1}\right)\right\|_{L_{p}(\cM\otimes B(\ell_2))}^{1/2}\\ 
		&=  \|x_n\|_{L_p(\cM;\ell_2^c)}. \qedhere 
	\end{align*}	
\end{proof}

In the noncommutative setting usually we do not have the analogue of the complex interpolation $(L_p(\ell_{q_1}),L_p(\ell_{q_2}))_\theta = L_p (\ell^{cr}_2)$ with $1/2 = (1-\theta)/q_1 + \theta/q_2$, which is an essential obstruction to generalize many classical methods on maximal inequalities in \cite{bourgain86onhigh,bourgain87ondimensionfree,carbery89convexbody}. Nevertheless, we   still have the following weaker property, which will be enough for our purpose in this paper. 
More precisely, we will compare the norms of positive symmetric maps on $L_p(\cM;\ell_2^{cr})$ with those on $L_{\frac{p}{2-p}}(\cM;\ell_{\infty})$. Note that if $T$ is a symmetric and selfadjoint map on $\cM$ (by selfadjointness we mean that $T(x^*) = T(x)^*$ for all $x\in\cM$), then 
\begin{equation}\label{eq:symmtric bracket}
	\tau(T (x)y)=\tau([T (x^*)]^*y)=\tau((x^*)^*T (y))=\tau(xT (y)) \qquad x, y \in S_{\cM}.
\end{equation}
Therefore, $T $ equals its predual operator on $L_{1}(\cM)$. In particular, $T$ extends to $L_{1}(\cM)$ with the same norm, and by interpolation it also extends to a bounded map on $L_{p}(\cM)$ with $1<p<\infty$. In this context we state the following property (note that a positive map is automatically selfajoint).

\begin{lemma}\label{lemma: A(p,2)<A(p/p-2) sqrt}
	Let	$(\Phi_j)_j$ be a sequence of unital positive and symmetric maps on $\cM$. Denote $\Phi: (x_j)_{j}\mapsto (\Phi_jx_j)_{j}$.  Let $1<p< 2$. Then
	$$\|\Phi: L_p(\cM;\ell_2^{cr})\to L_p(\cM;\ell_2^{cr})\|\leq 2\|\Phi:L_{\frac{p}{2-p}}(\cM;\ell_{\infty})\to L_{\frac{p}{2-p}}(\cM;\ell_{\infty})\|^{1/2}.$$
	Similar inequalities hold for the spaces $L_p(\cM;\ell_2^{c })$ and $L_p(\cM;\ell_2^{ r})$ if $(\Phi_j)_j$ is a sequence of unital $2$-positive maps.
\end{lemma}

\begin{proof}
	Since $\Phi_j$ are unital positive maps, by Kadison's Cauchy-Schwarz inequality \cite{kadison52schwarz}, for any selfadjoint  element $x_j\in L_p(\cM)$, we have
	$$\Phi_j(x_j)^2\leq \Phi_j(x_j^2).$$
	Assume  that $(x_j)_j \in L_{p^\prime}(\cM; \ell_2^{cr})$ is a sequence of selfadjoint elements. Then the conjugate index $p^\prime$ is greater than $2$ and
	\begin{align*}
		\|(\Phi_{j}x_j)_j\|_{L_{p^\prime}(\cM; \ell_2^{cr})}&= \left\Vert \left(\sum_{j}( \Phi_{j}x_j)^2 \right)^{1/2} \right\Vert_{p^\prime}
		\leq \left\Vert \left(\sum_{j} \Phi_{j}(x_j^2) \right)^{1/2} \right\Vert_{p^\prime}\\
		&=\left\Vert \left( \Phi_{j}(x_j^2) \right)_j \right\Vert_{L_{p^\prime/2}(\cM;\ell_1)}^{1/2}\\
		&\leq\|\Phi:L_{p^\prime /2}(\cM;\ell_{1})\to L_{p^\prime/2}(\cM;\ell_{1})\|^{1/2} \|(x_j)_j\|_{L_{p^\prime}(\cM; \ell_2^{cr})}. 
	\end{align*}
	For general $(x_j)_j \in L_{p^\prime}(\cM; \ell_2^{cr})$, we may decompose it into two sequence of selfadjoint elements. Note that $\|(x_j)_j\|_{L_{p^\prime}(\cM; \ell_2^{cr})}=\|(x_j^*)_j\|_{L_{p^\prime}(\cM; \ell_2^{cr})}$ for $p^\prime> 2$.
	Therefore, 
	$$\|\Phi: L_{p^\prime}(\cM;\ell_2^{cr})\to L_{p^\prime}(\cM;\ell_2^{cr})\|\leq 2\|\Phi:L_{p^\prime /2}(\cM;\ell_{1})\to L_{p^\prime/2}(\cM;\ell_{1})\|^{1/2}.$$
	As explained in \eqref{eq:symmtric bracket}, the dual operator of $\Phi$ equals itself and we obtain
	$$\|\Phi: L_p(\cM;\ell_2^{cr})\to L_p(\cM;\ell_2^{cr})\|\leq 2 \|\Phi:L_{\frac{p}{2-p}}(\cM;\ell_{\infty})\to L_{\frac{p}{2-p}}(\cM;\ell_{\infty})\|^{1/2},$$
	as desired.
	
	For the spaces $L_p(\cM;\ell_2^{c })$ and $L_p(\cM;\ell_2^{ r})$,  similar arguments still work for non selfadjoint elements $(x_j)$ if the maps $\Phi_j$ are $2$-positive, 	since in this case we can use the following Cauchy-Schwarz inequality   $|\Phi_j(x_j)|^2\leq \Phi_j(|x_j|^2) $ (see e.g. \cite[Proposition 3.3]{paulsen02completely}).
\end{proof}
%
%

The square function estimates for noncommutative diffusion semigroups has been established in \cite{jungelemerdyxu06Hinftycalut}. In this section we will slightly adapt the arguments of \cite{jungelemerdyxu06Hinftycalut} so as to obtain a refined version of this result for our further purpose.
Throughout this subsection, $(S_t)_{t \in \bR_+}$ always denotes  a semigroup of unital completely positive trace preserving  and symmetric maps on $\mathcal M$ with the negative infinitesimal generator $A$.  Let $(P_t)$ denote the subordinate Poisson semigroup of $(S_t)$, i.e. the negative generator of $P_t$ is $-(-A)^{1/2}$.

For notational simplicity, the spaces $L_p(\cM;L_2^c(\bR;\frac{dt}{t}))$, $L_p(\cM;L_2^r(\bR;\frac{dt}{t}))$ and $L_p(\cM;L_2^{cr}(\bR;\frac{dt}{t}))$ will be denoted by $L_p(L^c_2(\frac{dt}{t}))$, $L_p(L^r_2(\frac{dt}{t}))$ and $L_p(L^{cr}_2(\frac{dt}{t}))$ respectively  in this section.

To state our theorem, we recall the  dilation property. Let $(\cM, \tau),(\cN, \tau^\prime)$ be two von Neumann algebra where $\tau$ and $\tau^\prime$ are normal faithful semifinite traces.  If $\pi\colon (\mathcal M,\tau) \to (\mathcal N,\tau')$ is a normal unital
faithful trace preserving $*$-representation,   it (uniquely)
extends to an isometry from $L_p(\mathcal M)$ into $L_p(\mathcal N)$ for any
${1\leq p<\infty}$. We call
the adjoint $\bE\colon\mathcal N\to\mathcal M $ of the embedding
$L_1(\mathcal M )\hookrightarrow L_1(\mathcal N)$ induced by $\pi$\textit{ the conditional
	expectation associated with  $\pi$}. Moreover $\bE\colon\mathcal N\to\mathcal M $
is unital and completely positive. 
\begin{definition}\label{def:Rota}
	Let  $T\colon \mathcal M \to \mathcal M$ be a bounded operator. We say
	that $T$ satisfies  \emph{Rota's dilation property} if there exist a
	von Neumann algebra $\cN$ equipped with a normal semifinite faithful trace, a
	normal unital faithful $*$-representation $\pi\colon \mathcal M \to \cN$
	which preserves traces, and a decreasing sequence $(\cN_m)_{m\geq
		1}$ of von Neumann subalgebras of $\cN$ such that
	\begin{equation}\label{8Rota}
		T^{m} = \hat \bE \circ \bE_m\circ\pi,\qquad m\geq 1,
	\end{equation}
	where $\bE_m \colon \cN \to \cN _m\subset  \cN$ is the canonical
	conditional expectation onto $\cN_m$, and where $\hat \bE \colon \cN \to\mathcal M $
	is the conditional expectation associated with  $\pi$.
\end{definition}

We recall the following typical examples of operators with Rota's dilation property.

\begin{lemma}\label{lem:eg rota}
	\emph{(1) (\cite{jungericardshlyakhtenko2014noncommutative,dabrowski2010dilation})} If $\mathcal M$ is a finite von Neumann algebra and $\tau$ is a normal faithful state on $\mathcal M$, then for all $t \in \bR_+$, the operator $S_t$ satisfies Rota's dilation property.
	
	\emph{(2) (\cite{stein70booktopics})} If $\mathcal M$ is a commutative von Neumann algebra and $\mathcal L$ is another semifinite von Neumann algebra, then for all $t \in \bR_+$, the operator $S_t \otimes \mathrm{Id}_{\mathcal L}$ on $\mathcal M \overline \otimes \mathcal L$ satisfies Rota's dilation property.
\end{lemma} 

We aim to prove the following square function estimates, which are  essentially established in \cite{jungelemerdyxu06Hinftycalut,jiaowang17semigroups}, without specifying the order $(p-1)^{-6}$.  However, we will see that the methods in \cite{jungelemerdyxu06Hinftycalut}, together with the sharp constants of various martingale inequalities, are enough to obtain this order. The outline of our proof is slightly different from that of \cite{jungelemerdyxu06Hinftycalut}, but all the ingredients are already available in the latter. 

\begin{prop}\label{thm:order of square funct}
	Assume that for all $t \in \bR_+$, the operator $S_t$ satisfies Rota's dilation property. Then for all $1<p<2$ and $x\in L_p(\mathcal M)$ we have
	\begin{align}\label{order of square funct}
		\inf \left\lbrace \left\|(\int_0^\infty \left| t{\partial_t}P_t(x^c)\right| ^2\frac{dt}{t})^{\frac12}\right\|_p+  \left\|(\int_0^\infty \left| (t{\partial_t}P_t(x^r))^* \right| ^2\frac{dt}{t})^{\frac12}\right\|_p\right\rbrace \leq c (p-1)^{-6} \|x\|_p
	\end{align}
	where the infimum runs over all $x_c, x_r\in L_p(\cM)$ such that $x=x_c+x_r$, and $c>0$ is an absolute constant.
\end{prop}

\begin{remark}
	The order $(p-1)^{-6}$ stated in the above theorem is not optimal; after we finished the preliminary version of this paper, the result has been improved in the recent preprints \cite{xu21holomorphic,xuzhang21inprogress}.
	However, this order $(p-1)^{-6}$ is sufficient for our purpose in the sequel.
\end{remark}

Our study of the semigroup $(  P_t )_t$ is based on the analysis of the ergodic averages  as follows:
$$M_t=\frac{1}{t}\int^t_{0}S_udu.$$
We will need the following claim.
\begin{lemma}\label{lem:average control poisson}
	For any $y\in L_p(\cM)$,
	\begin{equation*}
		\|(t\partial P_ty)_t\|_{L_p(L^c_2(\frac{dt}{t}))}\leq c \|(t\partial M_ty)_t\|_{L_p(L^c_2(\frac{dt}{t}))},
	\end{equation*}
	where  $c>0$ is an absolute constant.
	The inequality remains true if we replace the norm of $L_p(L^{c}_2(\frac{dt}{t}))$ by $L_p(L^r_2(\frac{dt}{t}))$ or $L_p(L^{cr}_2(\frac{dt}{t}))$.
\end{lemma}
\begin{proof}
	This claim is well known to experts. We only give a sketch of its arguments. Set $\varphi(s)=\frac{1}{2\sqrt{\pi}}\frac{e^{-1/4s}}{s^{3/2}}$. Using integration by parts  we have
	$$P_t=\frac{1}{t^2}\int^\infty_{0}\varphi(\frac s{t^2})(\partial(sM_s))ds=-\int^\infty_0s\varphi^\prime(s)M_{t^2s}ds.$$
	Therefore
	\begin{align}\label{expression of poisson}
		t\partial P_t=-2\int^\infty_0 t^2s^2\varphi^\prime(s)\partial M_{t^2s}ds=-2\int^\infty_0s\varphi^\prime(s)(t^2s\partial M_{t^2s})ds
	\end{align}
	which yields the claim by noting that $\|\cdot \|_{L_p(L^c_2(\frac{dt}{t}))}$ is a norm and $s\varphi^\prime(s)$ is absolutely integrable.
\end{proof}

We need the following auxiliary result.

\begin{prop}\label{prop:square funct estimate}
	Assume that for all $t \in \bR_+$, the operator $S_t$ satisfies Rota's dilation property. Then for $1<p<2$ and $x\in L_p(\mathcal M)$, we have
	\begin{align}\label{square funct estimate}
		\| ( t{\partial}P_t(x) )_{t>0} \|_{L_p (L_2^{cr}(\frac{dt}{t}))} \leq c (p-1)^{-2} \|x\|_p,
	\end{align} 
	where $c>0$ is an absolute constant.
\end{prop}

\begin{proof}
	This result has been essentially obtained in \cite{jungelemerdyxu06Hinftycalut}, together with the optimal estimates for martingale inequalities in \cite{rand02ncmartigale,jungexu05ncmartigale}. Indeed, let $(\bE_n)_{n\in \bN}$ be a monotone sequence of conditional expectations on $\mathcal M$ and 
	$$x_n=\bE_{n+1}(x) - \bE_n (x)$$
	be a sequence of martingale differences with $x\in L_p (\mathcal M)$. By the estimate for noncommutative martingale transform in \cite[Theorem 4.3]{rand02ncmartigale} and the Khintchine inequality in Lemma \ref{prop:Khintchine}, we have
	$$\|(x_n)_{n\in \bN}\|_{L_p (\ell_2^{cr})} \leq \frac{c}{p-1} \|x\|_p,$$ 
	and by the noncommutative Stein inequality \cite[Theorem 8]{jungexu05ncmartigale} we have for any sequence $(y_n)_{n\in \bN} \subset L_p (\mathcal M)$,
	$$\|(\bE_n y_n)_{n\in \bN}\|_{L_p (\ell_2^{cr})} \leq \frac{c}{p-1}  \|( y_n)_{n\in \bN}\|_{L_p (\ell_2^{cr})} , \quad 1<p<2,$$
	where $c>0$ is an absolute constant.
	Then tracing the order in the proof of \cite[Corollary 10.9]{jungelemerdyxu06Hinftycalut}, we obtain that for all $\varepsilon>0$,
	$$\|(\sqrt{m}D^{\varepsilon}_m(x))_{m\geq1}\|_{L_p(\ell^{cr}_2)}\leq  \frac{c'}{(p-1)^{2}}\|x\|_p,$$
	where  $c' >0$ is an absolute constant, and where
	$$D^{\varepsilon}_m=\rho^{\varepsilon}_m-\rho^{\varepsilon}_{m-1}\quad \text{and}
	\quad
	\rho^{\varepsilon}_m=\frac{1}{m+1}\sum^m_{k=0}S_{k\varepsilon}.$$
	By a standard discretization argument (see e.g. \cite[Lemma 10.11]{jungelemerdyxu06Hinftycalut}), we get the following inequality
	$$\|(t\partial M_t x)_t\|_{L_p(L^{cr}_2(\frac{dt}{t}))}\leq c^\prime(p-1)^{-2}\|x\|_p.$$
	Then	the desired result follows from Lemma~\ref{lem:average control poisson}.
\end{proof}

%

Estimate (\ref{square funct estimate}) is weaker than (\ref{order of square funct}).  The remaining task for proving Proposition~\ref{thm:order of square funct} is to show 
\begin{align}\label{c-c estimate}
	&\ \quad \inf \left\lbrace  \left\|(\int_0^\infty \left| t{\partial}P_t(x^c)\right| ^2\frac{dt}{t})^{\frac12}\right\|_p+ 
	\left\|(\int_0^\infty \left| (t{\partial}P_t(x^r))^* \right| ^2\frac{dt}{t})^{\frac12} \right\|_p\right\rbrace  \\
	&\lesssim (p-1)^{-4} \| ( t{\partial}P_t(x) )_{t>0} \|_{L_p (L_2^{cr}(\frac{dt}{t}))}  \nonumber
\end{align}
where the infimum runs over all decompositions $x=x^c+x^r$ in $L_p(\cM)$.  
This inequality is essentially proved in \cite[Theorem 7.8]{jungelemerdyxu06Hinftycalut}; the order $(p-1)^{-4}$ is not stated there, but it follows from a careful computation  on all the related constants appearing in the argument therein. For the convenience of the reader, we will recall some parts of the proof and clarify all the constants in the proof which are concerned with the precise order. 

For notational simplicity, we say that a family $\cF\subset B(L_p(\cM))$ is \emph{Col-bounded}  (resp. \emph{Row-bounded}) if  there is a constant $C$  such that for any sequence $(T_k)_k\subset \cF$, we have
\begin{equation}\label{eq: def col-bdd} 
	\|(T_k )_k : L_p(\cM;\ell_2^c)\to L_p(\cM;\ell_2^c) \| \leq C ,
\end{equation} 
$$\left(\text{resp. }   \|(T_k )_k : L_p(\cM;\ell_2^r)\to L_p(\cM;\ell_2^r) \| \leq C \right) ,$$
and the least constant $C$   will be denoted by $\col(\cF)$ (resp. $\row(\cF)$).

We quote a useful result from {\cite[Lemma 3.2]{clementdepagtersukochev00Schauder}} (see also \cite[Lemma 4.2]{jungelemerdyxu06Hinftycalut}).
\begin{lemma}\label{lemma:Col-bounded by convex combination}
	Let $\cF\subset B(L_p(\cM))$ be a Col-bounded (resp. Row-bounded) collection with  $\col(\cF)=M$ (resp. $\row(\cF)=M$). Then the closure of the complex absolute convex hull of $\cF$ in the strong operator topology is also  Col-bounded (resp. Row-bounded) with the constant $\col(\cF)\leq 2M$ (resp. $\row(\cF)\leq 2M$).
\end{lemma}

For any $\theta\in (0, \pi)$, we let 
$$\Sigma_\theta=\{z\in \bC^*: |\mathrm{Arg}(z)|<\theta  \}.$$
Without the concrete order of growth of the constant on $p$, the following lemma is contained in \cite[Theorem 5.6]{jungelemerdyxu06Hinftycalut}. Note that the present area $\bC\backslash {\Sigma_{\nu_p}}$   in the following lemma is contained in the optimal area $\bC\backslash {\overline{\Sigma_{\omega_p}}}$ described in  \cite[Theorem 5.6]{jungelemerdyxu06Hinftycalut}.
\begin{lemma} \label{lemma:theorem 5.6 of JLMX}
	Let $1<p<2$. Let 	 $(S_t)_t$ be  a semigroup of unital completely positive trace preserving and symmetric maps.
	Let $A$ be the  negative infinitesimal generator  of  $(S_t)_t$.
	Then 
	the set $\cF_p=\{z(z-A)^{-1}: z\in \bC\backslash {\Sigma_{\nu_p}} \}\subset B(L_p(\cM))$ with $\nu_p=\frac{(p+1)\pi}{4p}$ is Col-bounded and Row-bounded with constants 
	$$\mathrm{Col}(\cF_p)\leq c(p-1)^{-2} \add{and} \mathrm{Row}(\cF_p)\leq c(p-1)^{-2},$$
	where $c$ is an absolute constant. 
\end{lemma}
\begin{proof}

	Let $s_1,\cdots, s_n$ be some nonnegative real numbers. For any $z\in \bC^*$ with $0\leq \mathrm{Arg}(z)\leq \frac{\pi}{2} $, we define a map $U(z)$ with
	\begin{align*}
		U(z):L_2(\cM;\ell_2^c)\cap L_q(\cM;\ell_2^c)&\to L_2(\cM;\ell_2^c)+ L_q(\cM;\ell_2^c)\\
		(x_k)_k&\mapsto(S_{zs_k}(x_k))_k.
	\end{align*}	
	Note that for any $x\in L_2(\cM)$, the function $z\mapsto U(z)x$ is continuous and bounded in the area $\{z\in \bC^*: 0\leq \mathrm{Arg}(z)\leq \pi/2 \}$ by \cite[Proposition 5.4 and Lemma 3.1]{jungelemerdyxu06Hinftycalut}.	This $U(z)$ is well defined. 
	On the one hand, for any $t>0$, we have
	$$\|U(te^{i\frac{\pi}{2}}): L_2(\cM;\ell_2^c) \to L_2(\cM;\ell_2^c)\|\leq 1.$$
	On the other hand, by duality and Lemma \ref{lemma: A(p,2)<A(p/p-2) sqrt} and Proposition \ref{prop:maximal inequality of semigroup}, we may find an absolute constant $c$ such that for any $ t>0, \ 2<q<\infty$,
	\begin{align*}
		\|U(t):L_q(\cM;\ell_2^c)\to L_q(\cM;\ell_2^c)\|
		&= \|U(t):L_{\frac{q}{q-1}}(\cM;\ell_2^c)\to L_{\frac{q}{q-1}}(\cM;\ell_2^c)\| \\
		&\leq  \|U(t):L_{\frac{q}{q-2}}(\cM;\ell_\infty)\to L_{\frac{q}{q-2}}(\cM;\ell_\infty)\|^{1/2}\\
		& \leq cq.
	\end{align*} 
	Let $p^\prime=\frac{p}{p-1}$ be the conjugate number of $p$. In the following we fix  ${\beta_p} = \frac{\pi}{2p'}$ and write
	${q=\pp(\frac{\pi-2{\beta_p}}{\pi-\pp{\beta_p}})=2(\pp-1)}$ and $\alpha=\frac{2{\beta_p}}{\pi}=\frac{1}{\pp}.$
	These numbers satisfy $\frac{1-\alpha}{q}+\frac{\alpha}{2}=\frac{1}{\pp}$. By complex interpolation, we know that
	$$L_\pp(\cM; \ell_2^c)=[L_q(\cM;\ell_2^c), L_2(\cM;\ell_2^c)]_\alpha.$$
	By the `sectorial' form of Stein's interpolation principle (see for instance \cite[Lemma 5.3]{jungelemerdyxu06Hinftycalut}), we have
	$$\|U(e^{i{\beta_p}}):L_\pp(\cM; \ell_2^c)\to L_\pp(\cM; \ell_2^c) \|\leq (cq)^{1-\alpha}\leq c q.$$ 
	Thus,
	$$\|(S_{s_ke^{i{\beta_p}}}(x_k))_k\|_{L_\pp(\cM; \ell_2^c)}\leq cq\|(x_k)\|_{L_\pp(\cM; \ell_2^c)}.$$
	Similarly, we have 
	$$\|(S_{s_ke^{-i{\beta_p}}}(x_k))_k\|_{L_\pp(\cM; \ell_2^c)}\leq cq\|(x_k)\|_{L_\pp(\cM; \ell_2^c)}.$$
	Namely, $(S_z)_{z\in \partial {\Sigma_{\beta_p}}}$ is Col-bounded with constant $cq$.	Then we get that the set 
	$$\{S_z: L_\pp(\cM)\to L_\pp(\cM): z\in \Sigma_{\beta_p}\}$$
	is also Col-bounded. Indeed, by a standard argument (see e.g. \cite[Proposition~2.8]{lutz01fourier}), we see that any $S_z$ with $z\in \Sigma_{\beta_p}$ can be approximated by convex combinations of $\{S_z: z\in\partial {\Sigma_{\beta_p}}\} $.  Therefore, by Lemma~\ref{lemma:Col-bounded by convex combination} we get that   $$\col(\{S_z: L_\pp(\cM)\to L_\pp(\cM): z\in \Sigma_{\beta_p}\})\leq 2cq=2c\pp(\frac{\pi-2{\beta_p}}{\pi-\pp{\beta_p}}) \leq 2c (p-1)^{-1}.$$
	By duality, we have
	\begin{equation}\label{eq:row of Gp}
		\row\left(\{S_z: L_p(\cM)\to L_p(\cM): z\in \Sigma_{\beta_p}\}\right)\leq 2c (p-1)^{-1}.
	\end{equation}
	
	Set $\omega_p=\frac{\pi}{p}-\frac{\pi}{2}$ and $\nu_p=\frac{(p+1)\pi}{4p}$.  Then $0<\frac{\pi}{2}-\nu_p<{\beta_p}<\frac{\pi}{\pp}=\frac{\pi}{2}-\omega_p$. 
	By the Laplace formula, we have that for any $z\in \bC\backslash \overline{\Sigma_{\pi/2}}$,
	$$(z-A)^{-1}=-\int_0^\infty e^{tz}S_tdt.$$
	Denote $\Gamma_{\beta_p}^+=\{u=te^{i{\beta_p}}: t\in \bR_+\}$. By \cite[Proposition 5.4 and Lemma~3.1]{jungelemerdyxu06Hinftycalut}, $u\mapsto S_u$ is analytic on the area ${\Sigma_{\frac{\pi}{2}-\omega_p}}$.
	Note that $\Gamma_{\beta_p}^+\subset {\Sigma_{\frac{\pi}{2}-\omega_p}}$.	Then by the Cauchy theorem, we have that for any $z\in \bC\backslash \overline{\Sigma_{\pi/2}}$,
	\begin{equation}\label{eq:cauchy formula}
		z(z-A)^{-1}=-\int_0^\infty ze^{tz} S_tdt=-\int_{\Gamma_{\beta_p}^+} ze^{uz} S_u du.
	\end{equation}
	Note that for any $z\in \bC$ with $ \nu_p\leq \mathrm{Arg}(z)\leq \pi/2$, we have 
	\begin{align*}
		\real(uz)=|z|t\cos(\mathrm{Arg}(z)+{\beta_p})=-|z|t\sin(\mathrm{Arg}(z)- \pi/ 2p ),
	\end{align*}and hence,
	$$\int_{\Gamma_{\beta_p}^+} \|ze^{uz} S_u du\|\leq \sup_{u\in \Gamma_{\beta_p}^+}\|S_u\| \int_0^\infty re^{-rt\sin(\mathrm{Arg}(z)- \pi/ 2p )}dt\leq \frac{\sup_{u\in \Gamma_{\beta_p}^+}\|S_u\|}{\sin(\nu_p- \pi/ 2p )}<\infty .$$
	Hence, 
	(\ref{eq:cauchy formula})  holds for any $z\in \bC\backslash \Sigma_{\nu_p}$.
	Note that we have $\frac{\nu_p- \pi/ 2p}{\sin(\nu_p- \pi/ 2p)}\lesssim 1$ since ${0<\nu_p- \pi/ 2p<\frac{\pi}{8}}$.	By Lemma~\ref{lemma:Col-bounded by convex combination} and (\ref{eq:row of Gp}), we get 
	\begin{align*}
		\row(\cF_p)&\leq 2 \left|\int_{\Gamma_{\beta_p}^+} ze^{uz} du\right|\cdot \row(\{S_z: L_p(\cM)\to L_p(\cM): z\in \Sigma_{\beta_p}\}) \\
		&\lesssim \frac{(p-1)^{-1}}{\sin(\nu_p- \pi/ 2p) }\lesssim \frac{(p-1)^{-1}}{\nu_p- \pi/ 2p }\lesssim (p-1)^{-2}.
	\end{align*}
	A similar proof shows that 
	\begin{equation*}
		\col(\cF_p)\lesssim (p-1)^{-2}. \qedhere
	\end{equation*}
\end{proof}

Now let us prove the desired proposition.
\begin{proof}[Proof of \eqref{c-c estimate} and Proposition~\ref{thm:order of square funct}]
	Set $F(z)=ze^{-z}$, $G(z)=4F(z)$ and $\widetilde{G}(z)=\overline{G(\bar{z})}$. 
	Let $B=-(-A)^{1/2}$ be the negative  infinitesimal generator  of  $(P_t)_t$. We have 
	$$F(tB)x=tBe^{-tB}x=-t\frac{\partial}{\partial t}(P_t(x)).$$
	We set $\omega_p=\frac{\pi}{p}-\frac{\pi}{2}$,  $\nu_p=\frac{(p+1)\pi}{4p}$ and $\xi_p=\frac{(3p+1)\pi}{8p}$. These numbers satisfy ${\omega_p<\nu_p<\xi_p<\frac{\pi}{2}}$.
	Note that $(P_t)_{t \in \bR_+}$ is again  a semigroup of unital completely positive trace preserving  and symmetric maps.  
	By  \cite[Corollary 11.2]{jungelemerdyxu06Hinftycalut}, the operator ${B:L_p(\cM)\to L_p(\cM)}$ admits a bounded $H^\infty(\Sigma_{\xi_p})$ functional calculus. By  \cite[Theorem~7.6~(1)]{jungelemerdyxu06Hinftycalut}, $B$ satisfies the dual square function estimate $(\cS_{\tilde{G}}^*)$, since  $ G\in H_0^\infty(\Sigma_{\xi_p})$.   Note that $ F\in H_0^\infty(\Sigma_{\xi_p})$ and $\int_0^\infty G(t)F(t)\frac{dt}{t}=1$. By the proof of \cite[Theorem~7.8]{jungelemerdyxu06Hinftycalut}, we get that
	\begin{align}\label{eq:1st estimation}
		&\quad\ \inf \left\lbrace \left\|(\int_0^\infty \left| t{\partial}P_t(x^c)\right| ^2\frac{dt}{t})^{\frac12}\right\|_p+  \left\|(\int_0^\infty \left| (t{\partial}P_t(x^r))^* \right| ^2\frac{dt}{t})^{\frac12}\right\|_p\right\rbrace
		\\
		& \leq 2C \| ( t{\partial}P_t(x) )_{t>0} \|_{L_p (L_2^{cr}(\frac{dt}{t}))}	\nonumber
	\end{align}
	where the infimum runs over all decompositions of $x=x^c+x^r$ in $L_p(\cM)$ and $C=\max\{\|T_c\|, \|T_r\|\}$, with $T_c$ and $T_r$ being defined as 
	\begin{align*}
		T_c:L_p(L_2^{c}(\frac{dt}{t}))&\to \ L_p(L_2^{c}(\frac{dt}{t}))\\
		(x_t)_t &\mapsto (\int_0^\infty F(sB)G(tB)x_t\frac{dt}{t})_s
	\end{align*}
and
	\begin{align*}
		T_r:L_p(L_2^{r}(\frac{dt}{t}))&\to \ L_p(L_2^{r}(\frac{dt}{t}))\\
		(x_t)_t &\mapsto (\int_0^\infty F(sB)G(tB)x_t\frac{dt}{t})_s.
	\end{align*}
	Let $ \nu_p<\gamma < \xi_p $. Denote
	$$f_\gamma(t)=\begin{cases}
	-te^{i\gamma},& t\in \bR_-,\\
	te^{-i\gamma}, &t\in \bR_+,
	\end{cases}$$
	and  let $\Gamma_\gamma=\{f_\gamma(t):t\in \bR\} \subset \bC$. Set
	\begin{align*}
		T_\Phi :L_p(L_2^{c}(\bR, \left|\frac{dt}{t}\right|))&\to L_p(L_2^{c}(\bR, \left|\frac{dt}{t}\right|))\\
		(x_{t})_{t\in \bR}&\mapsto \left(\frac{f_\gamma(t)(f_\gamma(t)-B)^{-1}x_t}{2\pi i}\right)_{t\in \bR}.
	\end{align*}
	Denote $K_1=\int_{\Gamma_\gamma}|F(z)|\left|\frac{dz}{z}\right|$ and $K_2=\int_{\Gamma_\gamma}|G(z)|\left|\frac{dz}{z}\right|.$
	By the proof of \cite[Theorem~4.14]{jungelemerdyxu06Hinftycalut}, we have 
	$$\|T_c\|\leq K_1K_2\|T_\Phi\|.$$
	And the proof of \cite[Proposition~4.4]{jungelemerdyxu06Hinftycalut} shows that
	$$\|T_\Phi\|\leq \col(\cO)$$
	where $\cO=\left\lbrace \frac{1}{\mu(I)}\int_I f_\gamma(t)(f_\gamma(t)-B)^{-1}d\mu(t): I\subset \bR, 0<\mu(I)<\infty \right\rbrace $.
	Moreover, by Lemma~\ref{lemma:Col-bounded by convex combination},
	$$\col(\cO)\leq 2\col(\{z(z-B)^{-1}: z\in \Gamma_\gamma\}).$$
	Since $\gamma>\nu_p$,	by Lemma~\ref{lemma:theorem 5.6 of JLMX}, $$\col(\{z(z-B)^{-1}: z\in \Gamma_\gamma\})\lesssim (p-1)^{-2}.$$
	On the other hand, 
	\begin{align*}
		K_1 &=\int_{\Gamma_\gamma}|F(z)|\left|\frac{dz}{z}\right|=2\int_0^\infty |te^{-i\gamma}e^{-t(\cos(\gamma)-i\sin(\gamma))}|\frac{dt}{t}\\
		&=2\int_0^\infty e^{-t\cos(\gamma)}dt\lesssim \frac{1}{\cos(\gamma)},
	\end{align*}
	and $K_2=4K_1$. 	
	Note that $\gamma<\xi_p=\frac{(3p+1)\pi}{8p}$ and $0<\frac{\pi}{2}-\xi_p<\frac{\pi}{16}$. We have 
	$$\frac{1}{\cos(\gamma)}\leq \frac{1}{\sin(\frac{\pi}{2}-\xi_p)} 
	\lesssim (p-1)^{-1}.$$
	Therefore, $$\|T_c\|\lesssim (p-1)^{-4}.$$
	Similarly, $$\|T_r\|\lesssim (p-1)^{-4}.$$ Thus $C\lesssim (p-1)^{-4}$ and we obtain \eqref{c-c estimate}. As mentioned previously, this implies Proposition~\ref{thm:order of square funct}. The proof is complete.
\end{proof}

\chapter{Criteria for maximal inequalities of Fourier multipliers}\label{sec:maximal}
	
This chapter is devoted to the study of general criteria for maximal inequalities and pointwise convergences given by Criteria 1 and 2.
Our argument does not essentially rely on the group theoretic structure. Note that there are a number of typical structures with Fourier-like expansions in noncommutative analysis, which are not given by group algebras. Hence instead of the framework in Criterion 1 or 2, we would like to state and prove results in a quite general setting.

To proceed with our study, we will only require the following simple framework. In the sequel, we fix a von Neumann algebra $\cM$ equipped with a normal semifinite faithful trace $\tau$, and an isometric isomorphism of Hilbert spaces ${U : L_2(\cM)\to L_2(\Omega, \mu;H)}$ for some  distinguished regular Borel measure space $(\Omega, \mu)$ and Hilbert space $H$. Assume additionally that $U^{-1}(C_c (\Omega;H))$ is a dense subspace in $L_p (\mathcal M)$ for all $1\leq p \leq \infty$ (for $p=\infty$ we refer to the w*-density), where $C_c (\Omega;H)$ denotes the space of $H$-valued continuous functions with compact supports. Given a measurable function $m \in L_\infty(\Omega;\mathbb C)$, we denote by $T_m$ the linear operator  on $L_2(\cM)$ determined by
\begin{equation}\label{eq:definition of multiplier operoter}
	T_m : L_2(\cM)\to L_2(\cM),\quad U(T_m x) = m 
	U(x) ,\quad x\in L_2(\cM).
\end{equation}
We call $m$ the \emph{symbol} of  $T_m$.
The operator  $T_m$ is obviously a generalization of a classical Fourier multiplier. Moreover, for a discrete group $\Gamma$, taking 
$$\mathcal M =VN(\Gamma),\quad (\Omega,\mu) = (\Gamma,\text{counting measure}),\quad H=\mathbb C ,\quad  U:\lambda(g)\mapsto \delta_g,$$
the above framework coincides  with that considered in Criterion 1.

\begin{example}\label{eg:more vna}
	Apart from group von Neumann algebras, this framework applies to various typical models in the study of noncommutative analysis. As an illustration we recall briefly some of them.
	
	(1) Twisted crossed product (\cite{bedosconti09twisted,bedosconti12twisted2}): Let $\Gamma$ be a discrete group with a twisted dynamical system $\Sigma$ on a von Neumann algebra $\mathcal N \subset B(L_2 (\mathcal N))$. Then we may consider the von Neumann algebra $\mathcal M$ generated by the associated regular covariant representation of $\Gamma$ and the natural representation of  $\mathcal N$ on $\ell_2 (\Gamma;L_2 (\mathcal N))$. Each $x\in \mathcal M$ admits a Fourier series
	$\sum_{g\in \Gamma} \hat x (g) \lambda_\Sigma (g) $ with $\hat x (g)\in \mathcal N$. Take $\Omega =\Gamma$, $H=  L_2 (\mathcal N)$ and $U:x\mapsto \hat x $. It is easy to see that for any $m\in \ell_\infty (\Gamma)$, the multiplier in \eqref{eq:definition of multiplier operoter} is given by 
	$$ T_m : \sum_{g\in \Gamma} \hat x (g) \lambda_\Sigma (g) \mapsto
	\sum_{g\in \Gamma} m(g) \hat x (g) \lambda_\Sigma (g),$$
	which is the usual Fourier multiplier considered in \cite{bedosconti09twisted}. As a particular case, this also coincides with the Fourier multipliers on quantum tori studied by \cite{chenxuyin13harmonic}.
	
	(2) Rigid C*-tensor category (\cite{popavaes15reprentation,arnodelaatwahl18Fourier}): Let $\mathcal C$ be a rigid C*-tensor category, $A(\mathcal C)$  its Fourier algebra and $\mathcal M$  the von Neumann algebra generated by the image of the left regular representation of $\mathbb C  [\mathcal C  ]$. Set $(\Omega,\mu)=(\mathrm{Irr}(\mathcal C),d)$ where $d$ denotes the intrinsic
	dimension, and set $U:L_2 (\mathcal M)\to \ell_2 (\Omega)$ by $U(\alpha)=\delta_\alpha$ for $\alpha\in \mathrm{Irr}(\mathcal C)$. Then for any $m\in \ell_\infty (\Omega)$, it is easy to check that $T_m$ is the dual map of the multiplication operator $\theta \mapsto m\theta$ for $\theta \in A(\mathcal C)$, which gives the Fourier multiplier studied in \cite{popavaes15reprentation,arnodelaatwahl18Fourier}.
	
	(3) Clifford algebras, free semicircular systems and $q$-deformations: Let $\mathcal M $ be a $q$-deformed von Neumann algebra $\Gamma_q (H)$ in the sense of Bo\.{z}ejko and Speicher \cite{bozejkospeicher91anexample,bozejkospeicher94cp}. The case $q=0$ corresponds to Voiculescu's free Gaussian von Neumann algebra and the case $q=-1$  to the usual Clifford algebras. We choose the canonical orthonormal basis of $L_2 (\mathcal M)$ with index set $\Omega$  according to the Fock representation $\bigoplus_n H^{\otimes n}$, and denote by $U $ the corresponding isomorphism. We view  $m:\mathbb N \to \mathbb C$ naturally as a function on $\Omega$ by setting the value $m(n)$ on indexes of basis in $H^{\otimes n}$. Then for any such $m$, the operator $T_m$ coincides with the radial Fourier multiplier studied in \cite[Section 9]{jungelemerdyxu06Hinftycalut}.
	
	(4) Quantum Euclidean spaces \cite{gonzalezjungeparcet19singular}: Let $\mathcal M =\mathcal R_{\mathrm{\Theta}}$ be the quantum
	Euclidean space associated with an antisymmetric $n\times n$-matrix $\mathrm \Theta$. Take $\Omega=\mathbb R ^n$ and let $U:L_2(\mathcal R_{\mathrm{\Theta}})\to L_2 (\mathbb R ^n)$ be the canonical isomorphism. Then the operator $T_m$ coincides with a usual quantum Fourier multiplier on $\mathcal R_{\mathrm{\Theta}}$.
	
	(5) The framework also applies to non-abelian compact groups, compact quantum groups and von Neumann algebras of locally compact groups, where we may take $U$ to be the usual Fourier transform. We will discuss some of them in more details in the next chapters.
\end{example}

Our criterion is based on comparisons with symbols of a symmetric Markov semigroup. To state our results, we fix a semigroup $(S_t)_{t \in \bR_+}$ of unital completely positive  trace preserving and symmetric maps on $\mathcal M$ of the form 
$$S_t =T_{e^{-t\ell}}: L_2(\cM)\to L_2(\cM),\quad U(S_t x) = e^{-t\ell(\cdot)  }
(Ux),\quad x\in L_2(\cM),$$
for a distinguished continuous function $\ell:\Omega\to [0, \infty)$. We will also consider the subordinate Poisson semigroup $(P_t)_t$ of $(S_t)_{t}$, that is, 
$$P_t =T_{e^{-t\sqrt \ell}}: L_2(\cM)\to L_2(\cM),\quad U(P_t x) = e^{-t\sqrt{\ell(\cdot)}
}	(Ux),\quad x\in L_2(\cM).$$
We will consider the family of operators $(T_{m_N})_{N\in \mathbb N}$ (resp. $(T_{m_t})_{t\in \bR_+}$) induced by a sequence of  measurable functions  $(m_N)_{N\in \bN}$ (resp. $(m_t)_{t\in \bR_+}$)  on $\Omega$. Recall that we are interested in the following types of conditions for the symbols $(m_N)_{N\in \bN}$ (resp. $(m_t)_{t\in \bR_+}$) in Criterion 1:

\textbf{(A1)} There exist $\alpha >0$ and $\beta >0$ such that for all $N\in \bN$ and almost all $\omega\in \Omega$,  we have
\begin{equation}\label{eq:criterion a1} 
	|1-m_N (\omega)|\leq \beta \frac{\ell(\omega)^\alpha}{2^{ N}},\quad |m_N (\omega)|\leq \beta \frac{2^{ N}}{\ell(\omega)^\alpha}. 
\end{equation}


\textbf{(A2)}
There exist $\alpha >0$,  $\beta >0$  and $\eta\in \mathbb N _+$ such that $t\mapsto m_t(\omega)$ is piecewise $\eta$-differentiable for almost all $\omega\in \Omega$,  and for  all $1\leq k\leq \eta$, all $t\in \bR_+$ and almost all $\omega\in \Omega$ we have
\begin{equation}\label{eq:criterion a3} 
	|1-m_t (\omega)|\leq \beta \frac{\ell(\omega)^\alpha}{t},\quad |m_t (\omega)|\leq \beta \frac{ t}{\ell(\omega)^\alpha},\quad \left|\frac{d^k m_t (\omega)}{dt^k} \right|\leq \beta \frac{1}{t^k}.
\end{equation} 

Intuitively, \textbf{(A1)} is motivated by considering the subsequence $(m_{2^N})_{N\in \mathbb N}$ of $(m_t)_{t\in \bR_+}$ in \textbf{(A2)}, but the present form in \textbf{(A1)} is slightly more general. Indeed we will  see in Chapter~\ref{sec:multipliers on CQG} other abstract and important constructions of symbols satisfying \textbf{(A1)} but without being of the aforementioned form $(m_{2^N})_{N\in \mathbb N}$.

\medskip
We split our study into two parts. The first part mainly deals with the $L_2$-theory of the above multipliers.  Note that \textbf{(A1)} (resp. \textbf{(A2)}) implies that $(m_N)_{N\in \bN}$ (resp. $(m_t)_{\bR_+}$)  is uniformly bounded with respect to the $\|\ \|_\infty$-norm: for any  $N\in \bN $ and  almost all $\omega\in \Omega$,
\begin{equation*}\label{eq:m N uniformly bounded}
	|m_N(\omega)|\leq \min\{|m_N(\omega)|+1,|1-m_N(\omega)|+1\}\leq \beta+1.  
\end{equation*}
Similar arguments hold for $(m_t)_{t\in\bR_+}$. This implies that the operators $(T_{m_N})_{N\in \mathbb N}$  and $(T_{m_t})_{t\in \bR_+}$ are uniformly bounded on $L_2 (\mathcal M)$. In this Chapter we will always assume that  \emph{the operators $(T_{m_N})_{N\in \mathbb N}$ and $(T_{m_t})_{t\in \bR_+}$ extend to uniformly bounded maps on $ \mathcal M $} and for notational convenience we set
$$ \gamma= \|T_{m_N}:\mathcal M  \to  \mathcal M \|,\ (\text{resp. } \gamma= \|T_{m_t}:  \mathcal M  \to  \mathcal M \|).$$ 
Then by complex interpolation they also extend to uniformly bounded maps on $ L_p ( \mathcal M ) $ for all $2\leq p \leq \infty$.	In this setting we have the following result. A more precise estimate on the endpoint case $p=2$ can be found in Section \ref{subsect:ltwo}.
\begin{theorem}
	\label{theorem:criterion1} Let $(T_{m_N})_{N\in \mathbb N}$ and $(T_{m_t})_{t\in \bR_+}$ be the uniformly bounded maps on $ \mathcal M $ given above.
	
	\emph{(1)} If $(m_N)_{N }$ satisfies  \emph{\textbf{(A1)}}, then for any $2\leq p < \infty$ there exists a constant $c>0$ depending only on $p,\alpha,\beta$ and $\gamma$, such that for all $x\in L_p (\mathcal M)$, 
	$$\|(T_{m_N} x)_N\|_{L_p(\cM;\ell_\infty)} \leq c \|x\|_p, \quad\text{and}\quad  T_{m_N} x\to x \text{ a.u. as }N\to\infty.$$ 
	
	\emph{(2)} If $(m_t)_{t }$ satisfies \emph{\textbf{(A2)}}, then for any $2\leq p < \infty$ there exists a constant $c>0$ depending only on $p,\alpha,\beta$ and $\gamma$, such that for all $x\in L_p (\mathcal M)$, 
	$$\|(T_{m_t} x)_t\|_{L_p(\cM;\ell_\infty)} \leq c \|x\|_p, \quad\text{and}\quad  T_{m_t} x\to x \text{ a.u. as }t\to\infty.$$
\end{theorem}
In order to obtain similar results for general $p>1$, we need to assume the positivity of the maps $(T_{m_N})_{N\in \mathbb N}$ and $(T_{m_t})_{t\in \bR_+}$. Note that if the maps extend to positive and symmetric contractions on $ \cM$, then by the argument before Lemma \ref{lemma: A(p,2)<A(p/p-2) sqrt}, they also extend to contractions on $L_p(\cM )$ for all $1\leq p\leq\infty$. In this framework we have the following results.
\begin{theorem}
	\label{theorem:criterion2}
	Assume that the operators $(T_{m_N})_{N\in \mathbb N}$ and $ (T_{m_t})_{t\in\bR_+}$ extend to  positive  and symmetric contractions on $ \cM$. Assume additionally that for all $t \in \bR_+$, the operator $S_t$ satisfies Rota's dilation property. 
	
	\emph{(1)} If $(m_N)_{N }$ satisfies    \emph{\textbf{(A1)}}, then for any $1< p<\infty$ there is a constant $c>0$ depending only on $p,\alpha,\beta$ such that for  all $x\in L_p (\cM)$,
	$$\|(T_{m_N} x)_N\|_{L_p(\cM;\ell_\infty)} \leq c \|x\|_p  \quad\text{and}\quad  T_{m_N} x\to x \text{ b.a.u. as }N\to\infty.$$ 
	Moreover, for $2 \leq p < \infty $ the b.a.u. convergence can be strengthened to a.u. convergence.
	
	\emph{(2)} If $(m_t)_{t }$ satisfies   \emph{\textbf{(A2)}}, then for any $1+\frac{1}{2\eta}<p<\infty $ there is a constant $c>0$ depending only on $p,\alpha,\beta,\eta$ such that for  all $x\in L_p (\cM)$
	$$\|(T_{m_t} x)_t\|_{L_p(\cM;\ell_\infty)} \leq c \|x\|_p  \quad\text{and}\quad  T_{m_t} x\to x \text{ b.a.u. as }t\to\infty.$$
	Moreover, for $2 \leq p < \infty $ the b.a.u. convergence can be strengthened to a.u. convergence.
\end{theorem}

The above theorems recover Criteria 1 and 2. Indeed, if $\cM $ is a finite von Neumann algebra, the additional assumption on Rota's dilation property is fulfilled by Lemma \ref{lem:eg rota} (1). Note that for any positive definite function $m$ on $\Gamma$, the associated map $T_m $ on $\vN$ is completely positive (see e.g. \cite[Theorem  2.5.11]{brownozawa08bookC*}). Also, by the Schoenberg theorem, for any conditionally negative definite function $\ell:\Gamma \to [0, \infty) $, the associated map $\lambda (g) \mapsto e^{-t\ell(g)} \lambda(g)$ forms a semigroup of unital completely positive trace preserving and {symmetric} maps on $VN(\Gamma)$. On the other hand,  for any function $m:\Gamma\to \mathbb \bC$ with $m(e)=1$, the map $T_m$ is $\tau$-preserving; if  $T_m$ is unital positive on $VN(\Gamma)$, then it extends to positive contractions to $L_p (VN(\Gamma))$ for all $1\leq p\leq \infty$ (see e.g. \cite[Lemma 1.1]{jungexu07ergodic}). Moreover, if $m$ is real-valued, then one may easily check that $T_m$ is a symmetric map. So the assumptions of Criterion 1 coincide with those of the above theorems. 

\medskip
Before starting the proof, we give several remarks on the statement of the above theorems.
\begin{remark}\label{rmk:B2}
	Instead of continuous families $(m_t)_{t\in \bR_+}$ in \textbf{(A2)}, we may also consider maximal inequalities of families $(m_N)_{N\in \bN}$ with suitable conditions on their differences, which we will frequently use in further discussions.
	Let  $(m_N)_{N\in \bN}$ be a family of  measurable functions  on $\Omega$ satisfying the following  assumption: 
	there exist $\alpha >0$ and $\beta >0$ such that for almost all $\omega\in \Omega$ we have
	\begin{equation}\label{eq:criterion a2} 
		|1-m_N (\omega)|\leq\beta \frac{\ell(\omega)^\alpha}{ N},\quad |m_N (\omega)|\leq \beta  \frac{ N}{\ell(\omega)^\alpha},\quad |m_{N+1} (\omega)- m_N (\omega)|\leq \beta  \frac{1}{N}. 
	\end{equation} 
	Then for any $2\leq p<\infty$, there exists a constant $c>0$ depending only on $p,\alpha,\beta$ and $\gamma$, such that for all $x\in L_p (\mathcal M)$, we have
	$$\|(T_{m_N} x)_N\|_{L_p(\cM;\ell_\infty)} \leq c \|x\|_p, \quad\text{and}\quad  T_{m_N} x\to x \text{ a.u. as }N\to\infty.$$ 
	If moreover the operators $(T_{m_N})_{N\in \mathbb N}$  extend  to  positive and symmetric contractions on $\cM$, then the assertion holds for all $3/2<p<\infty$ as well.
	
	This follows immediately from the previous theorems  since   \eqref{eq:criterion a2} leads to a special case of  \textbf{(A2)}. Indeed, for $0\leq t<1$ , 	set $m_t=m_0=0$. 
	For $t \geq 1$, we write $t=N_t+r_t$ with $N_t\in \bN$ and $0\leq r_t< 1$, and we define
	$${m}_t=(1-r_t)m_{N_t}+r_tm_{N_{t}+1}.$$
	It is obvious that  $(m_t)_{t\in \bR_+}$ satisfies  \textbf{(A2)} with $\eta=1$.
	
	One may also study more general conditions associated with higher order differences, which might be parallel to the case $\eta>1$ in  \textbf{(A2)}. However the computation seems to be much more intricate and we would like to leave it to the reader.
\end{remark}

\begin{remark}\label{rmk: A1 prime and change index set N to Z}
	The statement in \textbf{(A1)} and \textbf{(A2)} can be flexibly adjusted, which we will frequently use in further discussions: 
	
	(1) For $\alpha\geq 1$ and for the maps $(T_{m_N})_{N\in \mathbb N}$ and $(T_{m_t})_{t\in \bR_+}$ given in
	Theorem~\ref{theorem:criterion1} or Theorem~\ref{theorem:criterion2}, we will indeed establish the corresponding maximal inequalities under the following weaker conditions \eqref{eq:criterion a1prime} or \eqref{eq:criterion a3prime}. Indeed, for $\alpha\geq 1$, \textbf{(A1)} implies that for almost all $\omega\in \Omega$ we have 
	\begin{equation}\label{eq:criterion a1prime} 
		|1-m_N (\omega)|\lesssim_{\beta} \frac{\ell(\omega)}{2^{ N/\alpha}},\quad |m_N (\omega)|\lesssim_{\beta}  \frac{2^{ N/\alpha}}{\ell(\omega)}. 
	\end{equation}
	To see this, recall that  we have $|m_N(\omega)|\leq \beta+1$, so we see that 
	$$(\beta+1)^{-1}|m_N(\omega)|\leq ({(\beta+1)}^{-1}|m_N(\omega)|)^{\frac{1}{\alpha}}\leq (\beta+1)^{-\frac{1}{\alpha}}  \beta^{\frac{1}{\alpha}} \frac{2^{N/\alpha}}{\ell(\omega)}.$$
	Similarly, using $|1-m_N(\omega)|\leq \beta+1$ and repeating the above argument, we see that
	$$|1-m_N(\omega)|\lesssim_{\beta} \frac{\ell(\omega)}{2^{N/\alpha}}. $$
	In the same way, \textbf{(A2)} implies
	\begin{equation}\label{eq:criterion a3prime} 
		|1-m_t (\omega)|\lesssim_{\beta,\gamma} \frac{\ell(\omega)}{t^{1/\alpha}},\quad |m_t (\omega)|\lesssim_{\beta, \gamma} \frac{ t^{1/\alpha}}{\ell(\omega)},\quad \left|\frac{d^k m_t (\omega)}{dt^k} \right|\leq \beta \frac{1}{t^k}.
	\end{equation}
	On the other hand, the proof of the above theorems for the case of $0<\alpha<1$ can be always reduced to that of $\alpha\geq 1$. To see this it suffices to take $\tilde{\ell}=\ell^\alpha$ for $0<\alpha<1$ and consider the new semigroup of unital completely positive trace preserving and symmetric maps given by $ \tilde{S}_t \coloneqq T_{e^{-t\tilde{\ell}}}$ (see \cite{yosida95fa}); if the multipliers satisfy \textbf{(A1)} or \textbf{(A2)} with respect to $\ell$ for $0<\alpha<1$, then they also satisfy the same condition with respect to $\tilde{\ell}$ for $\alpha=1$. 

	(2) Theorem~\ref{theorem:criterion1}~(1) and Theorem~\ref{theorem:criterion2}~(1) still hold with the index set $\mathbb N$ replaced by $\mathbb Z$ in \textbf{(A1)}. 
	This can be seen from their proofs; alternatively, we may deduce this easily from a standard re-indexation argument. Indeed, let $(m_N)_{N\in \bZ}$ be  a sequence  of measurable functions on $\Omega$ satisfying \eqref{eq:criterion a1} for all $N\in \bZ$.
	Take $N_0\in \bZ$ and write  $\widetilde{m}_j=m_{N_0 +j}$ for $j\in \bN$. Then \eqref{eq:criterion a1} implies that for all $j\in \bN$ and almost all $\omega \in \Omega$,
	$$|1-\widetilde m_j(\omega)|\leq \beta\frac{2^{N_0}\ell(\omega)^\alpha}{2^j},\quad |\widetilde m_j(\omega)|\leq \beta\frac{2^j}{2^{N_0}\ell(\omega)^\alpha}.$$
	Note that $\tilde \ell  =2^{N_0/\alpha}\ell$ yields again a semigroup of unital completely positive trace preserving and symmetric maps $\tilde{S}^{(N_0)}_t:=T_{e^{-t\tilde \ell}}$. 
	Then applying Theorem~\ref{theorem:criterion1}~(1) or Theorem~\ref{theorem:criterion2}~(1) to $\widetilde{m}$ and $\tilde \ell$, we see that $(\widetilde{m}_j)_{j\in \bN}=(m_N)_{N\geq N_0}$ satisfies the corresponding maximal inequality with constant independent of $\widetilde{ \ell}$ and $N_0$. Thus the similar maximal inequalities and a.u. (b.a.u.) convergence still hold for $(m_N)_{N\in \bZ}$.
\end{remark}

\begin{remark}\label{rk:mtheorem for cb}
	The completely bounded version of the above two theorems holds true as well. In other words, if $\mathcal N$ is another semifinite von Neumann algebra and if we replace $T_{m_N}$ by $T_{m_N} \otimes\mathrm{Id}_{\mathcal N}$, $T_{m_t}$ by $T_{m_t} \otimes\mathrm{Id}_{\mathcal N}$  and $\mathcal M$ by $\tilde{\mathcal M} =\mathcal M \overline{\otimes} \mathcal N$, then the above two theorems still hold true. Indeed, it suffices to consider a larger Hilbert space $\tilde{H} =H \otimes L_2 (\mathcal N)$ and apply the above theorems to $\tilde{M}$ and $\tilde{H}$.
\end{remark}

The following result on mean convergences is an easy consequence of our assumptions.
\begin{prop}\label{lemma:p continuou of m_n}
	Let $1<p<\infty $ and $(T_{m_N})_N$, $(T_{m_t})_t$  be given as in Theorem \ref{theorem:criterion1} or Theorem \ref{theorem:criterion2} which satisfy \emph{\textbf{(A1)}} or \emph{\textbf{(A2)}} correspondingly. 
	
	\emph{(1)}	The family $(T_{m_t})_t$ is strongly continuous on $L_p (\mathcal M)$, i.e., for any ${x\in L_p (\mathcal M)}$ the function $t\mapsto T_{m_t} x $ is continuous from $\mathbb R _+$ to  $L_p (\mathcal M)$. 
	
	\emph{(2)}	We have 
	$$\lim_{N\to \infty}\|T_{m_N} x -x \|_p =0,\quad \lim_{t\to\infty}\|T_{m_t} x -x \|_p =0,\quad x\in L_p (\mathcal M).$$
\end{prop}
\begin{proof}
	Let $x\in \cU ^{-1}(C_c (\Omega;H)) $ and $E=\supp(\cU  (x))\subset \Omega$. By the H\"{o}lder inequality, for any $t_0\geq 0$ and $2\leq p <\infty$, 
	\begin{align*}
		\|T_{m_t}x-T_{m_{t_0}}x\|_p &
		\leq \|T_{m_t}x-T_{m_{t_0}}x\|_2^{2/p} \|T_{m_t}x-T_{m_{t_0}}x\|_\infty ^{1-2/p} \\
		&\lesssim \|(m_t -m_{t_0}) \dsone_{E} \|_{\infty}^{2/p}\|x\|_2^{2/p} \| x\|_\infty^{1-2/p}.
	\end{align*}
	By the continuity of $m_t$ and the compactness of $E$, the above quantity tends to $0$ as $t\to t_0$. Similar arguments work for $p<2$ by using the H\"{o}lder inequality with endpoints $p=1,2$.
	Similarly, by the continuity of $\ell$ and the compactness of $E$, we have
	$$\lim_{N\to\infty}\|T_{m_N}x-x\|_p = 0,\quad \lim_{t\to \infty}\|T_{m_t}x-x\|_p = 0.$$ 
	
	For general elements $x\in L_p (\mathcal M)$, it suffices to note that  the operators $(T_{m_N})_N$ and $(T_{m_t})_t$ extend to uniformly bounded operators on $L_p (\mathcal M)$. Thus the desired results follow from a standard density argument.
\end{proof}

Now we are ready to proceed with the proof of the previous theorems.

\section{$L_2$-estimates under lacunary conditions}
\label{subsect:ltwo}

\begin{proposition}\label{prop:general case for p=2}
	Let $(m_N)_{N\in \bZ}\subset L_{\infty}(\Omega)$. Assume that there exist a  function $f:\Omega \to [0, \infty)$ and a positive number $a>1$ such that for almost all $\omega\in \Omega$,
	\begin{equation}\label{eq:1-psi N < |g|/2 N}
		|m_N(\omega)|\leq \beta  \frac{a^Nf(\omega)}{(a^N+f(\omega))^2}.
	\end{equation}
	Then
	$$\| (T_{m_N}x)_{N\in\bZ}\|_{L_2(\cM; \ell_2^{cr})}\lesssim \beta \sqrt{\frac{a^2}{a^2-1}}\|x\|_2, \qquad x\in L_2(\cM).$$	
\end{proposition}
\begin{proof}
	We have 
	\begin{align*}
		\|(T_{m_N} x)_N\|^2_{L_2(\cM; \ell_2^{cr})}&=  \tau \left(\sum_{N\in \bZ}|T_{m_N} x|^2 \right)\\
		&=\sum_{N\in \bZ}\|T_{m_N} x\|^2_{L_2 (\mathcal M)}=\sum_{N\in \bZ} \| m_N U(x)\|^2_{L_2 (\Omega;H)}\\
		&=  \int_\Omega \sum_{N\in \bZ}\| m_N (\omega) (Ux)(\omega) \|_H^2 d\mu(\omega) \\
		&\leq  \|\sum_{N\in \bZ}|{m_N} |^2 \|_{L_\infty (\Omega)} \| Ux \|^2_{L_2 (\Omega;H)}
		\\
		&=  \|\sum_{N\in \bZ}|{m_N} |^2 \|_{L_\infty (\Omega)}  \|x\|_{L_2 (\cM)}^2.
	\end{align*}
	However, by (\ref{eq:1-psi N < |g|/2 N}) we see that for almost all $\omega \in \Omega$ with $f(\omega)>0$,
	$$ \sum_{N\in \bZ}|{m_N}(\omega)|^2\leq \sum_{N<\log_a f(\omega)} \beta^2 \frac{a^{2N}}{f(\omega)^2}+\sum_{N\geq \log_a f(\omega)} \beta^2\frac{f(\omega)^2}{a^{2N}}\lesssim \beta^2\frac{a^2}{a^2-1},$$
	while  $m_N(\omega)=0$ if $f(\omega)=0$ by \eqref{eq:1-psi N < |g|/2 N}.
	Thus we obtain the desired inequality.	
\end{proof}

Below we  show a  more precise $L_2$-estimate.
\begin{lemma}\label{lemma:the condition for whole sequence, p=2}
	Assume that $t\mapsto m_t(\omega)$ is differentiable for almost all $\omega \in \Omega$. 
Choose  an arbitrary measurable function $f:\Omega\to [0, \infty)$. For $j\in \bZ$, define  
$$a_{j}= \sup_t \left\{\sup_{2^{j-2}<\frac{f(\omega)}{t}\leq 2^{j}} |m_t(\omega)| \right\}, \quad b_{j}=\sup_t \left\{\sup_{2^{j-2}<\frac{f(\omega)}{t}\leq 2^{j}} t\cdot	\left|\frac{\partial m_t(\omega)}{\partial t}\right| \right\}.$$
Assume 
$$K=\sum_{j\in \bZ} a_j^{1/2}(a_j^{1/2}+b_j^{1/2})<\infty.$$
Then for $x\in L_2(\cM)$, we have the following maximal inequalities
$$\|(T_{m_t} x)_t\|_{L_2(\cM;\ell_{\infty})}\lesssim K \|x\|_2 \add{and} \|(T_{m_t} x)_t\|_{L_2(\cM;\ell_{\infty}^c)}\lesssim K \|x\|_2 .$$
\end{lemma}
\begin{proof}
	
	We prove the second assertion first.
	Let $\{\eta_j\}_{j\in \bZ}$ be a partition of unity of $\bR_+$ satisfying 
	$$\sum_j \eta_j=1,\qquad\supp \eta_j \subset [2^{j-2}, 2^{j}], \qquad 0\leq \eta_j\leq 1 \quad \text{and} \quad |\eta_j^\prime|<C2^{-j}.$$
	Define $m_{t,j}(\omega)=m_t(\omega)\eta_j(\frac{f(\omega)}{t})\in L_{\infty}(\Omega)$. For notational simplicity, denote by $T_{t,j}$ the operators with symbols $m_{t,j}$; that is,
	$$\cU(T_{t, j}x)= m_{t,j}\cU  (x), \quad x\in L_2(\cM).$$ 
	Then we have 
	\begin{equation}\label{eq:sup+ Litllewood-paley decomposition}
		\| (T_{m_t} x)_t\|_{L_2(\cM;\ell_{\infty}^c)}=\left\|\left( \sum_{j\in \bZ}  T_{t, j}x      \right)_t \right\|_{L_2(\cM;\ell_{\infty}^c)}\leq \sum_{j\in \bZ} \| (T_{t, j}x)_t\|_{L_2(\cM;\ell_{\infty}^c)}.
	\end{equation}
	From now on we fix an arbitrary $j\in \bZ$. In the sequel of this proof,  for  $x\in L_2 (\mathcal M)$, we denote 
	$$U_{k}(x)(\omega)=\cU  (x)(\omega)\cdot \dsone_{[2^{k-2}, 2^{k+1}]}(f(\omega)),\quad \omega \in \Omega \quad \text{ and } \quad
	x_{k}=\cU  ^*(U_{k}(x)).$$ 
	Since $\supp \eta_j\subset [2^{j-2},2^{j}]$, for $v\in \bZ$ and $t \in [2^v, 2^{v+1})$ we have
	$$m_{t,j}(\omega) \cU  (x)(\omega)=m_t(\omega)\eta_j(\frac{f(\omega)}{t})\cU  (x)(\omega)=m_t(\omega)\eta_j(\frac{f(\omega)}{t})U_{v+j}(x)(\omega). $$	
	We may rewrite the above equality as 
	$$T_{t, j}x=T_{t, j}x_{v+j},\qquad t \in [2^v, 2^{v+1}).$$
	Choose an integer $A_j$  such that $\frac{a_j+b_j}{a_j}\leq A_j\leq \frac{2(a_j+b_j)}{a_j}$ and	we divide the interval $[2^v, 2^{v+1}]$ into $A_j$ parts:
	$$2^v=\gamma_0<\gamma_1<\gamma_2\cdots<\gamma_{A_j}=2^{v+1}\quad \text{with }\quad \gamma_{k+1}-\gamma_k={2^v}\cdot{A_j^{-1}}.$$	
	For any  $t\in [2^v, 2^{v+1})$, there exists $0\leq k(t)\leq A_j-1$ such that $t\in [\gamma_{k(t)}, \gamma_{k(t)+1})$.	By the convexity of the operator square function, we have
	\begin{align*}
		|T_{t, j} x_{v+j}|^2&=|\cU^*(m_{t,j}U_{v+j}(x))|^2\\
		&=\left|\cU  ^*\left( \int_{\gamma_{k(t)}}^{t }\left( \frac{\partial m_{s,j}}{\partial s}\right)  U_{v+j}(x)ds+m_{\gamma_{k(t)}, j}\cdot U_{v+j}(x)\right) \right|^2\\
		&\leq 2(t-\gamma_{k(t)})\int_{\gamma_{k(t)}}^{t }\left|\frac{\partial T_{s,j}(x_{v+j})}{\partial s}\right|^2ds+2|T_{\gamma_{k(t)},j}(x_{v+j})|^2\\
		&\leq 2\left(\frac{2^v}{A_j} \int_{\gamma_{k(t)}}^{\gamma_{k(t)+1} }\left|\frac{\partial T_{s,j}(x_{v+j})}{\partial s}\right|^2ds+|T_{\gamma_{k(t)},j}(x_{v+j})|^2 \right)\\
		&\leq 2\sum_{k=0}^{A_j-1}\left(\frac{2^v}{A_j} \int_{\gamma_k}^{\gamma_{k+1} }\left|\frac{\partial T_{s,j}(x_{v+j})}{\partial s}\right|^2ds+|T_{\gamma_k, j}(x_{v+j})|^2 \right) \\
		&=\frac{2^{v+1}}{A_j} \int_{2^v}^{2^{v+1} }\left|\frac{\partial T_{s,j}(x_{v+j})}{\partial s}\right|^2ds+2\sum_{k=0}^{A_j-1}|T_{\gamma_k, j}(x_{v+j})|^2  .
	\end{align*}
	We denote  
	$$y_v=\frac{2^{v+1}}{A_j} \int_{2^v}^{2^{v+1}}\left|\frac{\partial T_{s,j}(x_{v+j})}{\partial s}\right|^2ds+2\sum_{k=0}^{A_j-1}|T_{\gamma_k,j}(x_{v+j})|^2 .$$ Then
	\begin{equation}\label{eq: T t j < y v}
		|T_{t,j}x_{v+j}|^2\leq y_v.
	\end{equation}
	Similarly, we have 
	$$|(T_{t,j}x_{v+j})^*|^2\leq y_v^\prime$$
	where we denote
	$$y_v^\prime=\frac{2^{v+1}}{A_j} \int_{2^v}^{2^{v+1}}\left|\left( \frac{\partial T_{s,j}(x_{v+j})}{\partial s}\right) ^*\right|^2ds+2\sum_{k=0}^{A_j-1}|T_{\gamma_k,j}(x_{v+j})^*|^2 .$$
	Let us estimate the quantities $\|y_v \|_1$ and $\|y_v^\prime \|_1$. 
	We have
	\begin{align*}
		\|y_v\|_1=\tau(y_v)=\frac{2^{v+1}}{A_j} \int_{2^v}^{2^{v+1} }\left\|\frac{\partial T_{s,j}(x_{v+j})}{\partial s}\right\|_2^2ds+2\sum_{k=0}^{A_j-1}\|T_{{\gamma_k}, j}(x_{v+j})\|_2^2  
	\end{align*}
	and
	\begin{align*}
		\|y^\prime_v\|_1=\tau(y^\prime_v)=\frac{2^{v+1}}{A_j} \int_{2^v}^{2^{v+1} }\left\|\left( \frac{\partial T_{s,j}(x_{v+j})}{\partial s}\right) ^*\right\|_2^2ds+2\sum_{k=0}^{A_j-1}\|T_{{\gamma_k}, j}(x_{v+j})^*\|_2^2.
	\end{align*}
	Hence, $\|y_v\|_1=\|y_v^\prime\|_1$.
	Note that 
	\begin{align*}
		 &\quad \ \int_{2^v}^{2^{v+1} } \left|\left( \frac{\partial m_{s,j}(\omega)}{\partial s}\right) \right|^2ds \\
		 &  = \int_{2^v}^{2^{v+1} } \left| \frac{\partial }{\partial s} \left(  m_s(\omega)\eta_j(\frac{f(\omega)}{s})\right) \right|^2ds\\
		&= \int_{2^v}^{2^{v+1} }\left|\frac{\partial m_{s}(\omega)}{\partial s} \cdot \eta_j\left( \frac{f(\omega)}{s}\right) -\eta_j^\prime\left( \frac{f(\omega)}{s}\right) \frac{f(\omega)}{s^2}\cdot m_{s}(\omega)\right|^2ds\\
		&	\lesssim  \int_{2^v}^{2^{v+1} }  \left(\left|\frac{b_j}{s}\right|+  \left|\frac{a_j}{s}\right|\right)^2ds \quad (\text{since }\supp \eta_j\subset [2^{j-2}, 2^j]) \\
		&	\lesssim \frac{1}{2^{v+1} }(b_j+a_j)^2.
	\end{align*}
	By the  Fubini theorem,	we have
	\begin{align*}
		&\quad\ \frac{2^{v+1}}{A_j} \int_{2^v}^{2^{v+1} }\left\|\frac{\partial T_{s,j}(x_{v+j})}{\partial s}\right\|_2^2ds
		\\&= \frac{2^{v+1}}{A_j}\int_{2^v}^{2^{v+1} }\left\|  \frac{\partial m_{s,j}}{\partial s} U_{v+j}(x)\right\|_{L_2(\Omega;H)}^2ds\\
		&=\frac{2^{v+1}}{A_j} \int_{\Omega} \left( \int_{2^v}^{2^{v+1} } \left|\left( \frac{\partial m_{s,j}(\omega)}{\partial s}\right) \right|^2ds\right) |U_{v+j}(x)(\omega)|^2d\mu(\omega)\\
		&\lesssim  \frac{(b_j+a_j)^2}{A_j} \|x_{v+j}\|_2^2,
	\end{align*}
	and
	\begin{align*}
		2\sum_{k=0}^{A_j-1}\|T_{{\gamma_k},j}(x_{v+j})\|_2^2&=\sum_{k=0}^{A_j-1}\int_{ \Omega} |m_{{\gamma_k}}(\omega)|^2\left|\eta_j\left(\frac{f(\omega)}{{\gamma_k}}\right)\right|^2|U_{v+j}(x)(\omega)|^2d\mu(\omega)\\
		&\leq 2A_ja_j^2\|x_{v+j}\|_2^2.
	\end{align*}
	Therefore,
	\begin{equation*}\label{eq:1-norm of yv}
		\|y_v^\prime\|_1=\|y_v\|_1\lesssim \left( \frac{(b_j+a_j)^2}{A_j}+ A_ja_j^2\right) \|x_{v+j}\|_2^2.
	\end{equation*}	
	Recall that (\ref{eq: T t j < y v}) asserts that	$|T_{t, j} x_{v+j}|^2\leq y_v\leq  \sum_{u\in \bZ} y_u $.
	So
	\begin{align}\label{eq:sup n Psi t j  < sum_v y_v}
		\|(T_{t, j}x)_t\|_{L_2(\cM;\ell_{\infty}^c)}&= \|(|T_{t,j}x_{v+j}|^2)_t \|_{L_1(\cM;\ell_{\infty})}^{1/2}
		\leq \left\|\sum_{u\in \bZ} y_u\right\|_1^{1/2}\leq\left(  \sum_{u\in \bZ}\|y_u\|_1\right) ^{1/2} \\
		& \lesssim \left( \sum_{u\in\mathbb Z}
		\left( \frac{(b_j+a_j)^2}{A_j}+A_ja_j^2\right) \|x_{u+j}\|_2^2 \right) ^{\frac{1}{2}}. 
		\nonumber
	\end{align}
	Note that 
	$$[2^{j+u-2}, 2^{j+u+1}]=\cup_{l=0}^{2} [2^{j+u+l-2}, 2^{j+u+l-1}].$$
	We have
	\begin{align*}
		\sum_{u\in \bZ}\|x_{u+j}\|_2^2&=\sum_{u\in \bZ}\int_{\Omega} \left| \cU  (x)(\omega)\cdot \dsone_{[2^{j+u-2}, 2^{j+u-2}]}(f(\omega))\right|^2 d\mu(\omega) \\
		&\leq \sum_{l=0}^{2} \int_{\Omega} \left| \cU  (x)(\omega)\cdot \sum_{u\in \bZ} \dsone_{[2^{j+u+l-2}, 2^{j+u+l-1}]}(f(\omega))\right|^2 d\mu(\omega)\\
		&=\sum_{l=0}^{2}\|\cU(x)\|_2^2
	\end{align*}
	where the last equality holds since
	$$\sum_{u\in \bZ} \dsone_{[2^{j+u+l-2}, 2^{j+u+l-1}]}(f(\omega))=1.$$
	Thus,
	$$\sum_{u\in \bZ}\|x_{u+j}\|_2^2\leq  3\|x\|_2^2.$$
	Recall that $\frac{a_j+b_j}{a_j}\leq A_j\leq \frac{2(a_j+b_j)}{a_j}$. 
	Together with  (\ref{eq:sup n Psi t j  < sum_v y_v}),
	we have 
	\begin{align*}
		\| (T_{t,j}x)_t\|_{L_2(\cM;\ell_{\infty}^c)}&\lesssim a_j^{1/2}(a_j+b_j)^{1/2}\left( \sum_{u\in \bZ} \|x_{u+j}\|_2^2\right)^{1/2}\lesssim a_j^{1/2}(a_j+b_j)^{1/2}\|x\|_2.
	\end{align*}
	By (\ref{eq:sup+ Litllewood-paley decomposition}), the proof is complete for the second maximal inequality.
	
	Similarly, we have 	$$ \|(T_{m_t} x)_t\|_{L_2(\cM;\ell_{\infty}^r)} \lesssim K \|x\|_2 .$$
	Let us recall Lemma~\ref{theorem:interpolation l infty} which shows that the space $L_2(\cM; \ell_\infty)$ is the complex interpolation space of $ L_2(\cM;\ell_{\infty}^c)$ and $ L_2(\cM;\ell_{\infty}^r)$.
	Therefore we have
	\begin{align*}
		&\quad\  \|(T_{m_t}): L_2(\cM)\to L_2(\cM;\ell_\infty) \|\\& \leq\|(T_{m_t}): L_2(\cM)\to L_2(\cM;\ell_\infty^c) \|^{1/2}\|(T_{m_t}): L_2(\cM)\to L_2(\cM;\ell_\infty^r) \|^{1/2}\\
		& \lesssim K. \qedhere
	\end{align*} 
\end{proof}

\section[Proof of Theorem~3.2]{Proof of Theorem~\ref{theorem:criterion1}}\label{sec:proof ot criterion1}	
Now we are ready to conclude  Theorem~\ref{theorem:criterion1}. 

First assume that $(m_N)_{N\in \bN}$ satisfies \textbf{(A1)}.
	Set $T_{\phi_N}=T_{m_N}-P_{2^{-N/\alpha}}$ with the symbol $\phi_N=m_N-e^{-\frac{\sqrt{\ell}}{2^{ N/2\alpha}}}$.By Remark~\ref{rmk: A1 prime and change index set N to Z}~(1) and \textbf{(A1)}, we can easily see that	
\begin{align}
	\label{eq:phi satisfies A1 A2 1}	|\phi_N(\omega)|&\leq |1-m_N(\omega)|+|1-e^{-\frac{\sqrt{\ell(\omega)}}{2^{ N/2\alpha}}}|\lesssim_{ \beta}  \frac{\sqrt{\ell(\omega)}}{2^{ N/2\alpha}} ,\\ 	\label{eq:phi satisfies A1 A2 2}
	|\phi_N(\omega)|&\leq |m_N(\omega)|+|e^{-\frac{\sqrt{\ell(\omega)}}{2^{ N/2\alpha}}}|\lesssim_{ \beta} \frac{2^{ N/2\alpha}}{\sqrt{\ell(\omega)}},
\end{align}
Therefore, $|\phi_N(\omega)|\leq \frac{2^{N/2\alpha}\sqrt{\ell(\omega)}}{\left(2^{N/2\alpha}+\sqrt{\ell(\omega)}\right)^2}$.
By Proposition~\ref{claim: sup+ < CRp norm}, Proposition~\ref{prop:general case for p=2} and Proposition \ref{prop: prop of LMinfty and LMl1} (3),  we get for any $2\leq p<\infty$ and $x\in L_p(\cM)$,
$$ \|(T_{\phi_N}(x))_{N}\|_{\Lpinfty{\cM}}\lesssim_{\alpha,\beta,\gamma, p} \|x\|_p,\quad \|(T_{\phi_N}(x))_{N}\|_{\Lpinftyc{\cM}}\lesssim_{\alpha,\beta,\gamma, p} \|x\|_p.$$
Applying Proposition~\ref{prop:maximal inequality of semigroup}, we also get the strong type  $(p, p)$ estimate for $T_{m_N}$ with a constant $c$ depending only on $\alpha,\beta,\gamma,p$.

	Assume that	 $(m_t)_{t\in \bR _+}$ satisfies \textbf{(A2)}. Again set 
$T_{\phi_t}=T_{m_t}-P_{t^{-1/\alpha}}$ with $\phi_t=m_t-e^{-\frac{\sqrt{\ell}}{t^{1/2\alpha}}}$. 
We have the following estimates similar to \eqref{eq:phi satisfies A1 A2 1} and \eqref{eq:phi satisfies A1 A2 2}: for almost all $\omega\in\Omega$
\begin{align*}
&|\phi_t(\omega)|\lesssim_\beta \min \left\{ \left(\frac{\ell(\omega)^\alpha}{t}\right)^{1/2\alpha}, \left(\frac{t}{\ell(\omega)^\alpha}\right)^{1/2\alpha}  \right\},\\
&|\frac{\partial \phi_t(\omega)}{\partial t}|\leq \beta \frac{1}{t} +\frac{1}{t}\cdot ( \frac{\sqrt{\ell(\omega)}}{\alpha  t^{1/2\alpha}}e^{-\frac{\sqrt{\ell(\omega)}}{t^{1/2\alpha}}})\lesssim_{ \alpha, \beta}\frac{1}{t}.
\end{align*}
Applying Lemma~\ref{lemma:the condition for whole sequence, p=2}, we get
$$\|{\sup_t}^+ T_{\phi_t}x\|_2\leq K\|x\|_2 \qquad x\in L_2(\cM),$$
where 
$$K\lesssim_{\alpha, \beta} \sum_{j\in \bZ}(2^{-|j|/2\alpha}(2^{-|j|/2\alpha}+1))^{1/2}\lesssim_{\alpha, \beta} 1.$$
By Proposition \ref{prop: prop of LMinfty and LMl1} (3), for any $2\leq p<\infty$,    $(T_{\phi_t})_{t\in \bR _+}$ is of strong type $(p, p)$ with  constant $c$ depending only on $\alpha, \beta, \gamma,  p$.	
Similarly, for any $2\leq p<\infty$, we have   
\begin{equation}\label{eq:phi max c}
\|(T_{\phi_t}(x))_{t\in \bR_+}\|_{\Lpinftyc{\cM}}\lesssim_{\alpha,\beta,\gamma,p} \|x\|_p \qquad x\in L_p(\cM).
\end{equation}
Therefore, we conclude the strong type $(p, p)$ estimate for $(T_{m_t})_t$ thanks to Proposition~\ref{prop:maximal inequality of semigroup}.  

	Now the desired a.u. convergence follows immediately from the above maximal inequalities by an argument in \cite{hongwangliao2017noncommutative}. For instance, we consider  the symbols $(m_t)_{t\in \bR_+}$ satisfying \textbf{(A2)}. 
Let  $x \in \cU ^{-1}(C_c (\Omega;H))$ and  set $E=\supp \cU(x)$. Note that $E$ is a compact set. We consider the maps $T_{\psi_t}=T_{m_t}-\id$ with $\psi_t=m_t-1$.
As the proof of Proposition~\ref{lemma:p continuou of m_n}, by \textbf{(A2)}, we have
\begin{align*}
\|T_{\psi_t}x\|_p &\leq 2\gamma^{1-(2/p)}\|(m_t-1)\dsone_E\|_{L_\infty(E)}^{2/p} \|x\|_2^{2/p} \|  x\|_\infty^{1-2/p} \\
&\lesssim_{\alpha, \gamma} \frac{\|\ell^\alpha\dsone_E\|^{2/p}_{L_\infty(E)}}{t^{2/p}}\|x\|_2^{2/p} \|  x\|_\infty^{1-2/p}.
\end{align*}
By the continuity of $\ell$ and the compactness of $E$, we have that for any $a>1$, and any large integer $M$
\begin{align}\label{eq:limit M int T psi }
\sum_{j\in \bN, a^j\geq M} \|T_{\psi_{a^j}}x\|_p^p dt 
&\lesssim_{\alpha, \gamma} \sum_{j\in \bN, a^j\geq M} \frac{\|\ell^\alpha\dsone_E\|_{L_\infty(E)}^2}{a^{2j}}\|x\|_2^{2 } \|  x\|_\infty^{p-2} dt\\
&\lesssim_{\alpha, \gamma}  \frac{\|\ell^\alpha\dsone_E\|_{L_\infty(E)}^2 \|x\|_2^{2 } \|  x\|_\infty^{p-2}}{M^2} . \nonumber
\end{align}
Thus, by Remark~\ref{prop:subsequence approxiamte to whole sequence}, as $M$ tends to $\infty$,
\begin{align*}
&\quad\|(T_{\psi_t}x)_{t\geq M}\|_{L_p(\cM;\ell_\infty^c)}^p\\
&\leq 	\|(|T_{\psi_t}x|^2)_{t\geq M}\|_{L_{p/2}(\cM;\ell_\infty)}^{p/2}=\lim_{a\to 1^+} \|(|T_{\psi_{a^j}}x|^2)_{a^j\geq M}\|_{L_{p/2}(\cM;\ell_\infty)}^{p/2}\\
& \leq \lim_{a\to 1^+} \|(\sum_{j\in \bN, a^j\geq M} |T_{\psi_{a^j}}x|^p)^{\frac{2}{p}}\|_{p/2}^{p/2}
\leq \lim_{a\to 1^+} \sum_{j\in \bN, a^j\geq M} \|T_{\psi_{a^j}}x\|_p^p \to 0. 
\end{align*}
As a result, for any $x \in \cU ^{-1}(C_c (\Omega;H))$,  $T_{\psi_t}(x)$ converges a.u. to $0$  as $t\to \infty$ according to Lemma \ref{lemma:Lp M c0 imply b.a.u}~(2). 
Moreover, \eqref{eq:phi max c} obviously yields that $(T_{\phi_t})_{t\in \bR_+}$ satisfies the {one-sided} weak type $(p,p)$ maximal inequality for $p\geq 2$ as in Proposition~\ref{prop:Phix-x in Lp M C0}. Note that the subordinate Poisson semigroup $(P_t)_t$ is of weak type $(p,p)$ for all $p\geq 1$ by \cite[Remark 4.7]{jungexu07ergodic}.
Hence for any $p\geq 2$, $x\in L_p(\cM)$, we find a projection $e\in \cM$ such that
$$\sup_t\|eP_t(x^*x)e\|_\infty \leq \alpha^2,\quad \tau(e^\perp)\lesssim_p \alpha^{-p}\|x^*x\|_{p/2}^{p/2}=\alpha^{-p}\|x\|_{p}^{p}.  $$
However, since $P_t$ is completely positive, we can use the Cauchy-Schwarz inequality to get  
$$\sup_t\|P_t(x)e\|_\infty=\|eP_t(x)^*P_t(x)e\|_\infty^{1/2}\leq \sup_t\|eP_t(x^*x)e\|_\infty^{1/2} \leq \alpha.$$ 
 So $(P_t)_t$ also satisfies the {one-sided} weak type $(p,p)$ maximal inequality for $p\geq 2$. Thus $(T_{m_t})_t$ also satisfies the same inequality and by Proposition~\ref{prop:Phix-x in Lp M C0}, we see that  $T_{\psi_t}(x)=T_{m_t}(x)-x$ converges a.u. to $0$ as $t\to \infty$ for all $x\in L_p(\cM)$.	
The case of \textbf{(A1)} can be dealt with similarly. Thus the proof of Theorem \ref{theorem:criterion1} is complete.
\section[Proof of Theorem~3.3~(1)]{Proof of Theorem~\ref{theorem:criterion2}~(1)}
In this section we study the maximal inequalities for $1<p<2$ in Theorem~\ref{theorem:criterion2}~(1). To approach this we need to develop several interpolation methods.

\begin{lemma}\label{lemma: strong pp strong 22 imply restricted  qq}
	Let $( \Phi_j)_{j\in \bZ}$ be a sequence of uniformly bounded linear maps on $\cM$ and write $\gamma=\sup_j\| \Phi_j:\cM\to\cM\|<\infty$. Let $1<p<q<2$ and $\theta\in(0, 1)$ be determined by $\frac{1}{q}=\frac{1-\theta}{p}+\frac{\theta}{2}$. Assume that there exist $c_1, c_2>0$ such that for any $s\in \bN_+$, there is a decomposition of maps $ \Phi_j= \Phi_j^{(s,1)}+ \Phi_j^{(s,2)}$ where $(\Phi_j^{(s,1)})_j$ extends to a family of maps of weak type $(p, p)$ with constant $C_p\leq s^{\theta} c_1$ and 
	$(\Phi_j^{(s,2)})_j$ extends to a family of maps of weak type $(2, 2)$ with constant $C_2\leq s^{-(1-\theta)}c_2$.
	Then for any $x\in \cS_\cM$ and $\lambda>0$, there exists a projection $e\in \cM$ such that
	$$ \sup_j \|e\Phi_j(x)e\|_\infty<\lambda \add{and} \tau(e^{\perp})\lesssim_\gamma (c_1^p+c_2^2)\left(  \frac{\|x\|_p^{1-\theta}\|x\|_2^\theta}{\lambda}\right) ^q.
	$$
	In particular, $(\Phi_{j})_{j\in \bZ}$ is of restricted weak type $(q,q)$  with constant
	\begin{equation*}
		C_q\lesssim_{\gamma}  (c_1^p+c_2^2)^{1/q}.
	\end{equation*}	
\end{lemma}
\begin{proof}

	Let  $x\in {\cS_\cM}_+$. Consider a  positive integer $s\in \bN_+$. By the weak type estimates of  $(\Phi_j^{(s,1)})_j$ and $(\Phi_j^{(s,2)})_j$, for any $\lambda>0$,
	we take two projections $e_1, e_2\in \cM$, such that
	$$\sup_j \|e_1 \Phi_j^{(s,1)}(x)e_1\|_\infty \leq \lambda\add{and} \tau(e_1^\perp)\leq\left(  c_1 s^\theta\frac{\|x\|_p}{\lambda}\right)^p ,$$
	$$\sup_j \|e_2 \Phi_j^{(s,2)}(x)e_2\|_\infty \leq \lambda\add{ and } \tau(e_2^\perp)\leq \left( c_2 s^{\theta-1}\frac{\|x\|_2}{\lambda}\right)^2.$$
	Set $e=e_1\wedge e_2$. Since $ \Phi_j= \Phi_j^{(s,1)}+ \Phi_j^{(s,2)}$, we have
	$$\|e \Phi_j(x)e\|_\infty 
	\leq 2\lambda, \quad\   j\in \bZ,$$
	and
	\begin{equation*}
		\tau(e^\perp)\leq \tau(e_1^\perp+e_2^\perp)
		\leq \left( c_1 s^{\theta}\frac{\|x\|_p}{\lambda}\right)^p + \left(  c_2s^{\theta-1}\frac{\|x\|_2}{\lambda}\right)^2.
	\end{equation*}
	
	We consider $x_\lambda^\perp=x\dsone_{(\lambda, \infty)}(x)$ and $x_\lambda=x\dsone_{[0, \lambda]}(x)$. Applying the above arguments to $x_\lambda^\perp$, we can find a projection 
	$e\in \cM$ such that
	$$\|e \Phi_j(x_\lambda^\perp)e\|_\infty \leq 2\lambda, \quad\   j\in \bZ,$$
	and
	\begin{equation}\label{eq: weak type of T-n-j}
		\tau(e^\perp)\leq \left(  c_1 s^{\theta}\frac{\|x_\lambda^\perp\|_p}{\lambda}\right)^p + \left( c_2s^{\theta-1}\frac{\|x_\lambda^\perp\|_2}{\lambda}\right)^2.
	\end{equation}
	Since $x=x_\lambda+x_\lambda^\perp$ and $\|x_\lambda \|_\infty \leq \lambda $, for any $j\in \bZ$ we have
	$$\|e( \Phi_jx)e\|_\infty\leq\|e( \Phi_j(x_\lambda))e\|_\infty+\|e ( \Phi_j(x_\lambda^\perp)) e\|_\infty \leq  (\gamma+2)\lambda.$$
	Note that for $1<p\leq 2$,
	$$\left(t\dsone_{(\lambda, \infty)}(t)\right)^p=\frac{\left(t\dsone_{(\lambda, \infty)}(t)\right)^2}{t^{2-p}}\leq \frac{\left(t\dsone_{(\lambda, \infty)}(t)\right)^2}{\lambda^{2-p}}, \qquad t> 0.$$
	Therefore we have $(x_\lambda^\perp)^p\leq \frac{(x_\lambda^\perp)^2}{\lambda^{2-p}}$
	and
	$$\frac{\|x_\lambda^\perp\|_p^p}{\lambda^p}\leq \frac{\|x_\lambda^\perp\|_2^2}{\lambda^2}.$$
	Hence, we can choose  $s$ to be an integer satisfying
	$$s\asymp  \left(  \frac{\lambda^p \|x_\lambda^\perp\|_2^2}{\lambda^2\|x_\lambda^\perp\|_p^p} \right)^{\frac{q}{2p}},$$
	and	by	(\ref{eq: weak type of T-n-j}) we have
	\begin{equation}\label{eq: weak type of T-n-j for x alpha}
		\begin{split}
			\tau(e^\perp)&\lesssim\left(  c_1 \left( \frac{\lambda^p \|x_\lambda^\perp\|_2^2}{\lambda^2\|x_\lambda^\perp\|_p^p}\right)^{\frac{q-p}{p(2-p)}}  \frac{\|x_\lambda^\perp\|_p}{\lambda}\right)^p + \left( c_2 \left(  \frac{\lambda^p \|x_\lambda^\perp\|_2^2}{\lambda^2\|x_\lambda^\perp\|_p^p} \right)^{\frac{q-2}{2(2-p)}}  \left( \frac{\|x_\lambda^\perp\|_2}{\lambda}\right)\right)^2\\
			&\lesssim   ( c_1^{p}+c_2^2)\lambda^{-q}\|x_\lambda^{\perp}\|_p^{(1-\theta) q}\|x_\lambda^\perp\|_2^{\theta q} .
		\end{split}
	\end{equation}
	Since $0\leq x_\lambda^\perp \leq x$, the above inequality yields
	\begin{equation*}\label{eq: weak x norm  2}
		{\tau(e^\perp)}\lesssim (c_1^{p}+c_2^2)\left( \frac{\|x\|_p^{1-\theta}\|x\|_2^{\theta }}{\lambda}\right)^q. 
	\end{equation*}
	
	In order to obtain the  restricted weak type $(q, q)$ estimate, it  suffices to take $x=f$ in the above inequality for an arbitrary projection $f\in {\cS_\cM}_+$. 
	Then
	$$\tau(e^\perp)\lesssim_\gamma (c_1^{p}+c_2^2)\left( \frac{\tau(f^p)^{(1-\theta)/p}\tau(f^2)^{\theta/2} }{\lambda}\right)^q= (c_1^{p}+c_2^2)\lambda^{-q}\tau(f) ,$$
	which implies that $(\Phi_{j})_{j\in \bZ}$ is of restricted weak type $(q,q)$ with constant
	\begin{equation*}
		C_q\lesssim_{ \gamma}  (c_1^p+c_2^2)^{1/q} . \qedhere
	\end{equation*}
\end{proof}


%

\begin{lemma}\label{lemma:Strong (p,p) of Delta_j}
	Assume that for all $t \in \bR_+$, the operator $S_t$ satisfies Rota's dilation property.
	Let $s\in \bN$,  $\alpha>0, j\in \bZ$ and define $${\Delta^{(s)}_{\alpha,j}=P_{2^{-(j+2s)/\alpha}}-P_{{2}^{-(j-2s)/\alpha}}}.$$
	Then for any $1<p<2$,
	$$\|(\Delta^{(s)}_{\alpha,j} x)_j\|_{\Lplcr{\cM}}\lesssim_\alpha s (p-1)^{-6} \|x\|_p,\qquad  x\in L_p(\cM).$$
\end{lemma}
\begin{proof}
	We may write
	$$\Delta^{(s)}_{\alpha,j}x=\int_{2^{-(j+2s)/\alpha}}^{2^{-(j-2s)/\alpha}}\left( -\frac{\partial}{\partial t}  P_t(x)\right) dt, \qquad x\in L_p(\cM) . $$
	Let $x=x_1+x_2$ for $x_1, x_2\in L_p(\cM)$. 
	By the convexity of the operator square function,
	\begin{align*}
		|\Delta^{(s)}_{\alpha,j}x_1|^2 &= \left\vert \int_{2^{-(j+2s)/\alpha}}^{2^{-(j-2s)/\alpha}} \frac{1}{ \sqrt{t} } \left(-\sqrt{t} \frac{\partial}{\partial t}  P_t(x_1)\right) dt  \right\vert^2 \\
		&\leq \left(  \int_{2^{-(j+2s)/\alpha}}^{2^{-(j-2s)/\alpha}}   t \left\vert   \left( \frac{\partial}{\partial t}  P_t(x_1) \right)\right\vert^2 dt  \right) \left( \int_{2^{-(j-2s)/\alpha}}^{2^{-(j+2s)/\alpha}} \frac{dt}{t}\right) \\
		&\lesssim_\alpha s \left(  \int_{\tilde{\alpha}^{-j-2s}}^{\tilde{\alpha}^{-j+2s}}   t \left\vert   \left( \frac{\partial}{\partial t}  P_t(x_1) \right)\right\vert^2 dt  \right) ,
	\end{align*}
	where $\tilde{\alpha}=2^{1/\alpha}$.
	Therefore,
	\begin{align*}
		\|\left(  \Delta^{(s)}_{\alpha,j}x_1 \right) _j\|_{L_p(\cM;\ell_2^c)}&=\left\|\left( \sum_{j=-\infty}^\infty|\Delta^{(s)}_{\alpha,j}x_1|^2\right)^{1/2}  \right\|\\
		&\lesssim_\alpha  \sqrt{s} \left\Vert  \left(  \sum_{j=-\infty}^{\infty}  \int_{\tilde{\alpha}^{-j-2s}}^{\tilde{\alpha}^{-j+2s}}   t \left\vert   \left( \frac{\partial}{\partial t}  P_t(x_1) \right)\right\vert^2 dt     \right)^{1/2}      \right\Vert_p\\
		&\lesssim_\alpha   \sqrt{s}  \left\Vert  \left(  \sum_{j=-\infty}^{\infty}  \sum_{k=-2s}^{2s-1}\int_{\tilde{\alpha}^{-j+k}}^{\tilde{\alpha}^{-j+k+1}}   t \left\vert   \left( \frac{\partial}{\partial t}  P_t(x_1) \right)\right\vert^2 dt     \right) ^{1/2}      \right\Vert_p\\
		&\lesssim_\alpha    s \left\Vert  \left(  \int_{0}^{\infty}   t \left\vert   \left( \frac{\partial}{\partial t}  P_t(x_1) \right)\right\vert^2 dt     \right) ^{1/2}      \right\Vert_p\\
	\end{align*}
	Similarly, 	
	$$\|(\Delta^{(s)}_{\alpha,j}x_2)_j\|_{L_p(\cM;\ell_2^r)}\lesssim_\alpha s \left\Vert  \left(  \int_{0}^{\infty}   t \left\vert   \left( \frac{\partial}{\partial t}  P_t(x_2) \right)^*\right\vert^2 dt     \right) ^{1/2}      \right\Vert_p.$$
	On the other hand, 
	$$\|(\Delta^{(s)}_{\alpha,j}x)_j\|_{L_p(\cM;\ell_2^{cr})}\leq \inf\{\| (\Delta^{(s)}_{\alpha,j}x_1) _j\|_{L_p(\cM;\ell_2^c)}+\|(\Delta^{(s)}_{\alpha,j}x_2)_j\|_{L_p(\cM;\ell_2^r)}\}$$
	where  the infimum runs over all $x_1, x_2\in L_p(\cM)$ such that $x=x_1+x_2$.
	Then the conclusion follows from Proposition~\ref{thm:order of square funct}.
\end{proof}

Now, let us prove  Theorem~\ref{theorem:criterion2}~(1).
\begin{proof}
	The case $p\geq 2$ has been already treated	by Theorem~\ref{theorem:criterion1}. In this proof, we focus on the case $1<p<2$. 	
	
	Fix a finite index set $J\subset \bZ$.
	Denote
	\begin{align*}
		A(p,\infty)&=\|(T_{m_j})_{j\in J}:\Lpinfty{\cM}\to\Lpinfty{\cM}\|,\\
		A(p, 1)&=\|(T_{m_j})_{j\in J}:\Lpli{\cM}\to\Lpli{\cM}\|,\\
		A(p,2)&=\|(T_{m_j})_{j\in J}: \Lplcr{\cM} \to L_p(\cM; \ell_2^{cr})\|.
	\end{align*}
	Since $J$ is finite, all these quantities are well-defined and finite.
	Because the operators $({T_{m_j}})_{j\in J}$ are positive maps, by Proposition~\ref{prop:positive maps Cp=A(p infty)} we have
	\begin{equation}\label{eq:A(p, infty)= C p}
		\|({T_{m_j}})_{j\in J}: L_p(\cM)\to L_p(\cM;\ell_\infty)\|\asymp A(p, \infty).
	\end{equation}
	Let $1<p<2$. It is sufficient to show that $A(p,\infty)$ is dominated by a positive constant independent of $J$. 
	
	Consider $1<q_1<q_2<2$ and let $\theta\in (0, 1)$ be the number satisfying ${\frac{1}{q_2}=\frac{1-\theta}{q_1}+\frac{\theta}{2}}$.
	For $s\in \bN_+$ we write $ s_0= \left[(1-\theta)\alpha\log_{2} s \right]+1$. Denote by
	$$ \Delta_j^{(s)} = P_{2^{-(j+2s_0)/2\alpha}}-P_{{2}^{-(j-2s_0)/2\alpha}}$$
	the  difference introduced in Lemma~\ref{lemma:Strong (p,p) of Delta_j} associated with $s_0$ and $2\alpha$. 
	By Proposition~\ref{claim: sup+ < CRp norm} and Lemma~\ref{lemma:Strong (p,p) of Delta_j},  we have
	\begin{align*}
		\|{\sup_{j\in J}}^+ \  T_{m_j} ( \Delta_j^{(s)} x)\|_{q_1}&\leq A(q_1, 2)\| (   \Delta_j^{(s)}  x  )_j \|_{L_{q_1}(\cM; \ell_2^{cr})}\\
		&\lesssim_\alpha A(q_1, 2) s_0 (q_1-1)^{-6} \|x\|_{q_1}.
	\end{align*}
	By Proposition~\ref{prop:maximal inequality of semigroup} and  Lemma~\ref{lemma:Strong (p,p) of Delta_j} , we have 
	\begin{align*}
		\|{\sup_{j\in J}}^+ \  S_{2^{-j/\alpha}} ( \Delta_j^{(s)} x)\|_{q_1}&\lesssim  ({q_1}-1)^{-2}\| (   \Delta_j^{(s)}  x  )_{j} \|_{L_{q_1}(\cM; \ell_2^{cr})}\\
		&	\lesssim_\alpha   s_0({q_1}-1)^{-8}\|x\|_{q_1}.
	\end{align*} 
	We set $T_{\phi_j}=T_{m_j}-S_{2^{-j/\alpha}}$ with $\phi_j=m_j-e^{-\frac{\ell(\cdot)}{2^{j/\alpha}}}$.
	Hence,
	\begin{equation*}
		\|{\sup_{j\in J}}^+ \  T_{\phi_j} ( \Delta_j^{(s)} x)\|_{q_1}\lesssim_\alpha A({q_1}, 2)  ({q_1}-1)^{-8} s_0\|x\|_{q_1}.
	\end{equation*}
	Let us assume $(1-\theta)\alpha\log_{2} s\geq 1$ first. Note that for any $s>0$ and $\delta>0$, we have 
	$\log_{2} s\lesssim \frac{s^\delta}{\delta}$. Therefore we get that
	$$s_0\leq 2 (1-\theta)\alpha\log_{2} s\lesssim_\alpha \frac{s^{\theta}}{\theta}\lesssim_{ \alpha}   (q_2-{q_1})^{-1}s^\theta. $$
	If $(1-\theta)\alpha\log_{2} s<1$, then $s_0=1\leq s^\theta\leq (q_2-{q_1})^{-1}s^\theta $.
	Hence, 
	\begin{equation}\label{eq:Delta j strong p, p}
		\|{\sup_{j\in J}}^+ \  T_{\phi_j} ( \Delta_j^{(s)} x)\|_{q_1}\lesssim_\alpha A({q_1}, 2)  ({q_1}-1)^{-8} (q_2-{q_1})^{-1}s^\theta\|x\|_{q_1}.
	\end{equation}
	Let $\omega \in \Omega$ and let 
	$$ \delta_j^{(s)} (\omega)=\exp \left( -\frac{\sqrt{\ell(\omega)}}{2^{(j+2s_0)/2\alpha}}\right)-\exp \left( -\frac{\sqrt{\ell(\omega)}}{{2}^{(j-2s_0)/2\alpha}}\right)$$
	be the symbol of $ \Delta_j^{(s)} $. Note that
	\begin{align*}
		|1- \delta_j^{(s)} (\omega)|&\leq \left|1-\exp \left( -\frac{\sqrt{\ell(\omega)}}{2^{(j+2s_0)/2\alpha}}\right) \right|+\left|\exp \left( -\frac{\sqrt{\ell(\omega)}}{{2}^{(j-2s_0)/2\alpha}}\right)\right|\\
		&\lesssim_\alpha   \frac{\sqrt{\ell(\omega)}}{2^{(j+2s_0)/2\alpha}}+ \frac{{2}^{(j-2s_0)/2\alpha}}{\sqrt{\ell(\omega)}}.
	\end{align*}
	 When $2^j\geq \ell(\omega)^\alpha$, by the above inequality we have  $|1- \delta_j^{(s)} (\omega)|\lesssim_\alpha   2^{-s_0/\alpha}( \frac{2^{j/\alpha}}{\ell(\omega)}) ^{1/2}$, and as the computation in \eqref{eq:phi satisfies A1 A2 1} we have $|\phi_j(\omega)|\lesssim_{\beta}  \frac{\ell(\omega)}{2^{j/\alpha}}$.
	In particular,
	$$|\phi_j(\omega)(1- \delta_j^{(s)} (\omega))|\lesssim_{\alpha, \beta}  2^{-s_0/\alpha}\left( \frac{\ell(\omega)}{2^{j/\alpha}}\right) ^{1/2}\lesssim_{\alpha, \beta}  2^{-s_0/\alpha}\left( \frac{2^{j/\alpha}\ell(\omega)}{(2^{j/\alpha}+\ell(\omega))^2}\right)^{1/2}  .$$
When  $2^j< \ell(\omega)^\alpha$, similarly we have $|1- \delta_j^{(s)} (\omega)|\lesssim_\alpha   2^{-s_0/\alpha}( \frac{\ell(\omega)}{2^{j/\alpha}}) ^{1/2}$, and as the computation in \eqref{eq:phi satisfies A1 A2 2} we have $|\phi_j(\omega)|\lesssim_{\beta}  \frac{2^{j/\alpha}}{\ell(\omega)}$. Therefore
	$$|\phi_j(\omega)(1- \delta_j^{(s)} (\omega))|\lesssim_{\alpha, \beta}  2^{-s_0/\alpha}\left( \frac{2^{j/\alpha}}{\ell(\omega)}\right) ^{1/2}\lesssim_{\alpha, \beta} 2^{-s_0/\alpha} \left( \frac{2^{j/\alpha}\ell(\omega)}{(2^{j/\alpha}+\ell(\omega))^2}\right)^{1/2}.$$
	By Proposition~\ref{prop:general case for p=2}, we have
	\begin{equation}\label{eq:1-Delta j strong 2 2}
		\|{\sup_{j\in J}}^+ T_{\phi_j}(1- \Delta_j^{(s)} )x\|_2\lesssim_{\alpha, \beta} 2^{-s_0/\alpha}\|x\|_2\lesssim_{\alpha, \beta} s^{\theta-1}\|x\|_2 .
	\end{equation}	
	Thus by (\ref{eq:Delta j strong p, p}), (\ref{eq:1-Delta j strong 2 2}) and Lemma~\ref{lemma: strong pp strong 22 imply restricted  qq},  we see that $(T_{\phi_j})_{j\in J}$ is of restricted weak type $(q_2, q_2)$ with constant
\begin{align}\label{eq:restrict qq phi}
	C_{q_2}^\prime
	&\lesssim_{\alpha, \beta} \left(\left( A({q_1},2) ({q_1}-1)^{-8} (q_2-{q_1})^{-1}\right)^{q_1}+1 \right)^{1/q_2}\\
	&\lesssim_{\alpha, \beta}  A({q_1},2) ({q_1}-1)^{-8} (q_2-{q_1})^{-1} . \nonumber
\end{align} 
	
	Set $D=\sup_{1<u\leq2} (u-1)^{22}A (u, \infty)<\infty$.  
	Choose an index $1<r\leq2$ such that $$(r-1)^{22}A(r, \infty)>\frac{D}{2}.$$ 
	We apply the restricted weak type estimate of $(T_{\phi_j} )_j$ in \eqref{eq:restrict qq phi} to the particular case ${q_1}=\frac{1}{2}(r+1)$ and $q_2={q_1}+(r-{q_1})/2$.  Note that by Proposition~\ref{prop:maximal inequality of semigroup}, the semigroup  $(S_t)_t$ is of strong type $(q_2, q_2)$ with constant $c(q_2-1)^{-2}$. 
	Recall that $T_{m_j}=T_{\phi_j}+S_{2^{-j/\alpha}}$, thus
	$(T_{m_j})_{j\in J}$ is also of restricted weak type $(q_2, q_2)$ with constant
	\begin{align*}
		C_{q_2}&\lesssim_{\alpha, \beta} A({q_1},2) ({q_1}-1)^{-8} (q_2-{q_1})^{-1}+(q_2-1)^{-2}\\
		&\lesssim_{\alpha, \beta} A(\frac{{q_1}}{2-{q_1}},\infty)^{1/2} (r-1)^{-9}.
	\end{align*}
	The last inequality above follows from Lemma~\ref{lemma: A(p,2)<A(p/p-2) sqrt} and the values of ${q_1}$ and $q_2$.
	Because $(T_{m_j})_j$ is of strong type $(\infty, \infty)$ and of restricted weak type of $(q_2, q_2)$,	applying  Theorem~\ref{theorem:interpolation of dirksen}  we have 
	\begin{equation}\label{eq:strong qq by interpolation}
		\|{\sup_{j\in J}}^+ T_{m_j}x\|_r\leq \max\{C_{q_2}, 1\}(\frac{rq_2}{r-q_2}+r)^2\|x\|_r, \qquad x\in L_r(\cM).
	\end{equation} 
	By \eqref{eq:A(p, infty)= C p}, this means that $A(r, \infty)\lesssim C_{q_2}  (r-1)^{-2}$.
	Therefore,
	\begin{equation}\label{eq:induction of A(p)}
		(r-1)^{-22}\frac{D}{2}<A(r, \infty) \lesssim_{\alpha, \beta} A(\frac{{q_1}}{2-{q_1}},\infty)^{1/2} (r-1)^{-11}.
	\end{equation} 
	Recall that  $A(\frac{{q_1}}{2-{q_1}}, \infty)\lesssim_{\alpha, \beta} 1$ if $\frac{{q_1}}{2-{q_1}}\geq 2$ by Theorem~\ref{theorem:criterion1}.  
	Without loss of generality we assume that  $1<\frac{{q_1}}{2-q_1}<2$. Then   (\ref{eq:induction of A(p)}) yields
	\begin{align*}
		(r-1)^{-22}\frac{D}{2}&\lesssim_{\alpha, \beta}    \left( \left( \frac{{q_1}}{2-{q_1}}-1\right) ^{-22}D \right) ^{1/2}(r-1)^{-11}
	\end{align*}
	Recall that ${q_1}=\frac{1}{2}(r+1)$. We have $$D\lesssim_{\alpha, \beta} 1.$$
	In other words,  $(p-1)^{22}A(p, \infty)\lesssim_{\alpha, \beta}    1$ for any $1<p\leq2$. 
	In particular, this estimate is independent of  the finite index set $J$. So we obtain the desired maximal inequality according to Remark \ref{prop:subsequence approxiamte to whole sequence}.
	
	By an argument similar to the proof of Theorem~\ref{theorem:criterion1}, we get that $T_{m_j}x$ converges a.u. to $x$ as $j\to \infty$ for $x\in L_p(\cM)$ with $p\geq 2$.		 Note that a.u. converges implies b.a.u. convergence and that $L_2(\cM)\cap L_p(\cM)$ is dense in $L_p(\cM)$ for any $1<p<2$.	
	Then by applying  Proposition~\ref{prop:Phix-x in Lp M C0} (1), we get the b.a.u. convergence of $(T_{m_j}x )_j$ for $x\in L_p (\mathcal M)$ with  $1<p<2$.
\end{proof}

\section[Proof of Theorem~3.3~(2)]{Proof of Theorem~\ref{theorem:criterion2}~(2)}	
Our idea is reducing the desired maximal inequalities to those for lacunary subsequences already studied in the preceding section.

\begin{lemma}\label{prop:from 2j to whole t}
	Assume that the family $(m_t)_{t\in \bR_+}$ satisfies  \emph{\textbf{(A2)}}. Then for any $1\leq q< 2$ and $q+\frac{q(2-q)}{q-1+2\eta}<  p< 2$, we have
\begin{align*}
	 &\quad\ \|(T_{m_{t}})_{t\in \bR_+}:L_p(\cM)\to L_p(\cM;\ell_\infty)\|\\ &\lesssim_{\beta,\eta,p,q} \sup_{1\leq \delta \leq 2}\|(T_{m_{  \delta 2^j}})_{j\in \bZ}:L_q(\cM)\to L_q(\cM;\ell_\infty)\|^{1-\theta}  
\end{align*}
	provided that the right hand side is finite, where $\theta$ is determined by $\frac{1}{p}=\frac{1-\theta}{q}+\frac{\theta}{2}$.
\end{lemma}
\begin{proof}
	Our proof is based on the estimate of  multi-order differences of $(m_t)_t$. For notational simplicity we denote these differences as follows: we start with setting the first order differences of the following form 
	$$\xpsii{s}{t}=m_{2^{2^{-s-1}}  t}-m_{t},\quad s\in\bN, t\in \mathbb R _+ ,$$
	and define the higher order ones inductively by
	$$\psi^{[s_1,s_2,\dots, s_v]}_{t}=\psi^{[s_1,s_2, \cdots,s_{v-1}]}_{2^{2^{-s_v -1} } t}-\psi^{[s_1,s_2, \cdots,s_{v-1}]}_{t},\quad   s_1,\ldots,s_v \in\bN, t\in\mathbb R _+ , 2\leq v\leq \eta.$$
	We denote by $\xPsi{s_1,s_2,\dots, s_v}{t}=T_{\psi^{[s_1,s_2,\dots, s_v]}_{t}}$ the associated multipliers for $1\leq v\leq \eta$. 
	
	We will estimate the maximal norms of $(T_{m_t})_t$ by using those of $(\xPsi{s_1,s_2,\dots, s_v}{t})_t$. To see this, note that
	by Proposition~\ref{lemma:p continuou of m_n}, $(T_{m_t}(x))_t$ is  strongly continuous on $L_p(\cM)$ for all $1< p< \infty$. We consider the dyadic approximations with increasing index sets  $I_s=\{2^{j/2^{s}}: j\in \bZ\}$ for $s\in\bN$. By Remark~\ref{prop:subsequence approxiamte to whole sequence},  we have for $x\in L_p (\mathcal M)$,
	\begin{align*}
		\|{\sup_{t \in \bR_+}}^+ T_{m_t} x \|_p&=\lim _{s\to \infty}\|{\sup_{t\in I_s}}^+ T_{m_{t}}x\|_p\\
		&\leq \|{\sup_{t\in I_0}}^+ T_{m_{t}}x\|_p+\sum_{s=0}^\infty\left(\|\Sup_{t\in I_{s+1}} T_{m_{t}}x\|_p-\|{\sup_{t\in I_s}}^+ T_{m_{t}}x\|_p\right) .
	\end{align*}
	Note that we have the bijection $J:I_s\to I_{s+1}\backslash I_{s},\ 2^{j/2^s}\mapsto 2^{(2j+1)/ 2^{s+1}}=2^{\frac{1}{2^{s+1}}} 2^{j/2^s}$ and $I_{s+1}=I_s\cup J(I_s)$. Hence for 
	$$y_t=
	\begin{cases}
	T_{m_t}x \quad &t\in I_s\\
	T_{m_{J^{-1}(t)}}x \quad &t\in J(I_s)
	\end{cases}  ,
	$$	
	we have $\|\sup^+_{t\in I_{s+1}} y_t\|=\|\sup^+_{t\in I_s} T_{m_t}x\|$.		
	Then applying the triangle inequality, we get
	$$\|\Sup_{t\in I_{s+1}} T_{m_t} x\|_p-\|{\sup_{t\in I_{s}}}^+ T_{m_t} x\|_p\leq  \|{\sup_{t\in I_s}}^+ \xPsi{s}{t} (x)\|_p.$$
In other words, we obtain
$$\|{\sup_{t \in \bR_+}}^+ T_{m_t} x \|_p\leq \|{\sup_{t\in I_0}}^+ T_{m_{t}}x\|_p+\sum_{s=0}^\infty\|{\sup_{t\in I_s}}^+ \xPsi{s}{t} (x)\|_p.$$
Applying the above arguments to maps of the form $\xPsi{s}{t}$ in place of $T_{m_t}$, we see that for each $s_1\geq 1$,
\begin{align*}
\|{\sup_{t\in I_{s_1}}}^+ \xPsi{s_1}{t} (x)\|_p  \leq \|{\sup_{t\in I_{0}}}^+ \xPsi{s_1}{t} (x)\|_p + \sum_{s_2=0}^{s_1-1}  \|{\sup_{t\in I_{s_2}}}^+ \xPsi{s_1,s_2}{t} (x)\|_p .
\end{align*}
Hence, 
\begin{align*}
&\quad\ \|{\sup_{t \in \bR_+}}^+ T_{m_t} x \|_p\\
&\leq \|{\sup_{j\in \bZ}}^+ T_{m_{2^j}}x\|_p+\sum_{s_1=0}^\infty\|{\sup_{t\in I_{0}}}^+ \xPsi{s_1}{t} (x)\|_p+\sum_{s_1=1}^\infty \sum_{s_2=0}^{s_1-1} \|\Sup_{t\in I_{s_2}} \xPsi{s_1,s_2}{t} (x)\|_p.
\end{align*}
Repeating this process  $\eta$ times, we get 
\begin{align}\label{eq:Phi t decomposition to s1 s2...}
&\quad\ \|{\sup_{t \in \bR_+}}^+ T_{m_t} x \|_p\\
&\leq \|{\sup_{j\in \bZ}}^+ T_{m_{2^j}}x\|_p+\sum_{s_1=0}^\infty\|{\sup_{t\in I_{0}}}^+ \xPsi{s_1}{t} (x)\|_p+\sum_{s_1=1}^\infty \sum_{s_2=0}^{s_1-1} \|{\sup_{t\in I_{0}}}^+ \xPsi{s_1,s_2}{t} (x)\|_p\nonumber\\
&\qquad +\cdots 	+\sum_{ s_1>s_2>\cdots>s_{\eta-1}} \|{\sup_{t\in I_{0}}}^+ \xPsi{s_1,s_2,\dots s_{\eta-1}}{t} (x)\|_p \nonumber\\
&\qquad +\sum_{ s_1>s_2>\cdots>s_{\eta}}  \|{\sup_{t\in I_{s_\eta}}}^+ \xPsi{s_1,s_2,\dots, s_\eta}{t} (x)\|_p.\nonumber
\end{align}	
	
	It remains to study the maximal norms of $(\xPsi{s_1,s_2,\dots, s_v}{t})_t$ on the right hand side of the above inequality. To this end we need to estimate the derivative $\dt{k}  \xpsi{s_1,s_2,\dots, s_v}{t}$ by virtue of Lemma \ref{lemma:the condition for whole sequence, p=2}. 			More precisely, we will show that for any  $1\leq v\leq \eta$,
	\begin{equation}\label{eq:k-th partical of psi}
		\left|\dt{k}\left( \xpsii{s_1,s_2,\dots, s_v}{t}\right)\right|\lesssim_\beta  2^{-(s_1+s_2\cdots+ s_v)} \frac{{(2k+2v)}^v}{t^k} , \qquad   0\leq k\leq \eta-v. 
	\end{equation}	
	Let us prove this inequality by induction. For notational simplicity we write ${\rho_v=2^{2^{-s_v-1}}}$. Note that $1<\rho_v <2$ and that applying the mean value theorem to the function $x\mapsto 2^x$, we get that for any $1\leq v\leq \eta$ and $k\geq 0$, 
	\begin{equation}\label{eq:delta v -1}
		\rho_v^k-1\lesssim k2^{-s_v}.
	\end{equation} 
	Consider first $v=1$. 
	For any $0\leq k\leq \eta-1$,
	$$
	|\dt{k}\xpsi{s_1}{t}|=|\dt{k}m_{\rho_{1}t}(\omega)-\dt{k}m_{t}(\omega)|
	= \left|\rho_{1}^k  \partial_{\gamma}^{k} m_{\gamma}(\omega)|_{\gamma=\rho_{1}t}-\dt{k}m_{t}(\omega)\right|$$
	By \textbf{(A2)},  $\left|\partial_{\gamma}^{k} m_{\gamma}(\omega)|_{\gamma=\rho_{1}t}\right|\lesssim_{ \beta} \frac{1}{(\rho_{1}t)^k}$. Hence
	$$
	|\dt{k}\xpsi{s_1}{t}|	\lesssim_{\beta}  (\rho_{1}^k-1)  \frac{1}{(\rho_{1}t)^k}+\left| \partial_{\gamma}^{k} m_{\gamma}(\omega)|_{\gamma=\rho_{1}t}-\dt{k}m_{t}(\omega) \right|.$$
	By the mean value theorem and \textbf{(A2)}, $$\left| \partial_{\gamma}^{k} m_{\gamma}(\omega)|_{\gamma=\rho_{1}t}-\dt{k}m_{t}(\omega) \right|\leq (\rho_{1}t-t) \sup_{t\leq \gamma \leq \rho_{1}t} |\partial_\gamma^{k+1} m_\gamma(\omega)|\lesssim_\beta \frac{\rho_{1}t-t}{t^{k+1}}. $$
	Therefore, according to \eqref{eq:delta v -1}, we have
	\begin{align*}
		|\dt{k}\xpsi{s_1}{t}|&\lesssim_{ \beta}(\rho_{1}^k-1)\frac{1}{(\rho_{1}t)^k} +\frac{\rho_{1}-1}{t^{k}}\\
		&\lesssim_{\beta}    \frac{ k 2^{-s_1}}{t^k}+\frac{2^{-s_1}}{t^k}\\
		&\lesssim_{\beta} 2^{-s_1} \frac{k+1}{t^k}. 
	\end{align*}
	So (\ref{eq:k-th partical of psi}) holds for  $v=1$.
	Assume that (\ref{eq:k-th partical of psi}) holds for some $1\leq v\leq\eta-1$ and consider the case of $v+1$. For any $0\leq k\leq \eta-(v+1)$,  arguing as  above, we have
	\begin{align*}
		|\dt{k}\xpsi{s_1,s_2, \cdots,s_{v+1}}{t}|
		&= |\rho_{v+1}^k  \partial_{\gamma}^{k} (\xpsi{s_1,s_2, \cdots,s_v}{\gamma})|_{\gamma=\rho_{v+1}t}-\dt{k}(\xpsi{s_1,s_2, \cdots,s_v}{t})|\\
		&\lesssim_{\beta}  (\rho_{v+1}^k-1)  2^{-(s_1+s_2\cdots+ s_v)} \frac{{(2k+2v)}^v}{(\rho_{v+1}t)^k}\\
		&\quad\qquad+(\rho_{v+1}t-t) \sup_{t\leq \gamma\leq \rho_{v+1}t} |\partial_\gamma^{k+1} (\xpsi{s_1,s_2, \cdots,s_v}{\gamma})|\\
		&\lesssim_{\beta}   k 2^{-s_{v+1}}  2^{-(s_1+s_2\cdots+ s_v)} \frac{(2k+2v)^v}{t^k}\\
		&\quad\qquad+(\rho_{v+1}t-t)  2^{-(s_1+s_2\cdots+ s_v)} \frac{(2k+2+2v)^v}{t^{k+1}}\\
		&\lesssim_{\beta} 2^{-(s_1+s_2\cdots+ s_{v+1})} \frac{(2(k+1+v))^{v+1}}{t^k}. 
	\end{align*}
	So (\ref{eq:k-th partical of psi}) is proved.
	In particular, setting $k=0$ and $k=1$ respectively, we get for any $1\leq v\leq \eta$ 
	\begin{equation}\label{eq:psi t s1 s2 ...}
		\left|\xpsi{s_1,s_2, \cdots,s_{v}}{t}\right|\lesssim_{\beta, \eta}  2^{-(s_1+s_2\cdots+ s_{v})},
	\end{equation}
	and  for any $1\leq v\leq\eta-1$,
	\begin{equation}\label{eq:dt d psi t s1 s2 ...}
		\left|\partial_t \xpsi{s_1,s_2,\cdots,s_{v}}{t}\right| \lesssim_{\beta, \eta} 2^{-(s_1+s_2\cdots+ s_{v})} \frac{1}{t}.
	\end{equation}
	This also yields
	\begin{align}
		\left|\partial_t \xpsi{s_1,s_2,\cdots,s_{\eta}}{t}\right| &\leq  \left|\partial_t \xpsi{s_1,s_2, \cdots,s_{\eta-1}}{\rho_\eta t}\right|+\left|\partial_t \xpsi{s_1,s_2, \cdots,s_{\eta-1}}{ t}\right| \nonumber \\
		&\lesssim_{\beta, \eta}  2^{-(s_1+s_2+ \cdots+ s_{\eta-1})} \frac{1}{t}. \label{eq:partical derivative for eta}
	\end{align}
	On the other hand,  by definition 
	\begin{equation}\label{eq:sum epsilon}
		\xpsii{s_1, s_2, \cdots,s_v}{t}= \sum_{\bm \varepsilon\in \{0, 1\}^v} (-1)^{\left( v+\sum_{i=1}^v\varepsilon_i\right) }m_{\bm{\rho^\varepsilon}   t} 
	\end{equation}
	where $\bm \varepsilon=(\varepsilon_1, \cdots, \varepsilon_v)\in \{0, 1\}^v$ and $\bm{\rho^\varepsilon}=\rho_{1}^{\varepsilon_1}\rho_{2}^{\varepsilon_2} \cdots\rho_{v}^{\varepsilon_v}$. Recall that ${s_1<s_2<\cdots< s_v}$ and $s_v\geq v$, we have
	$$1<\bm{\rho^\varepsilon}\leq 2^{2^{-s_v-1}+2^{-s_v}+\cdots+2^{-s_v+v-2}}< 2^{2^{-s_v-1+v}}<2.$$
	By \textbf{(A2)}, 
	$$ \left|\xpsi{s_1,s_2, \cdots,s_{v}}{t}\right|\leq \sum_{\bm \varepsilon\in \{0, 1\}^v}|m_{\bm{\rho^\varepsilon}   t}(\omega)|\leq \sum_{\bm \varepsilon\in \{0, 1\}^v}\beta   \frac{\bm{\rho^\varepsilon}  t}{\ell(\omega)^\alpha} \leq \beta  2^{2v} \frac{t}{\ell(\omega)^\alpha},$$
	$$ \left|\xpsi{s_1,s_2, \cdots,s_{v}}{t}\right|\leq \sum_{\bm \varepsilon\in \{0, 1\}^v}|1-m_{\bm{\rho^\varepsilon}   t}(\omega)|\leq \sum_{\bm \varepsilon\in \{0, 1\}^v} \beta   \frac{\ell(\omega)^\alpha}{\bm{\rho^\varepsilon}  t} \leq 2^v\beta   \frac{\ell(\omega)^\alpha}{t}.$$
	Thus, setting 
	$$a^{[s_1,s_2, \cdots, s_v]}_{j}:=\sup_t\left(  \sup_{2^{j-2}<\frac{\ell(\omega)^\alpha}{t}\leq 2^{j}}\left|\xpsi{s_1,s_2, \cdots,s_{v}}{t}\right| \right),$$
	and
	$$b^{[s_1,s_2, \cdots, s_v]}_{j}:=\sup_t\left( \sup_{2^{j-2}<\frac{\ell(\omega)^\alpha}{t}\leq 2^{j}} t 	\left|\partial_t \xpsi{s_1,s_2, \cdots,s_{v}}{t}\right| \right),$$
	together with \eqref{eq:psi t s1 s2 ...} and \eqref{eq:dt d psi t s1 s2 ...}, 	we have  for $1\leq v\leq \eta-1$,
	\begin{equation*}
		a^{[s_1,s_2, \cdots, s_v]}_{j} \lesssim_{\beta, \eta}  \min\{2^{-(s_1+s_2+ \cdots +s_{v})}, 2^{-|j|} \},   
	\end{equation*}
	\begin{equation*}
	b^{[s_1,s_2, \cdots, s_v]}_{j} \lesssim_{\beta, \eta} 2^{-(s_1+s_2+ \cdots +s_{v})}.  
	\end{equation*}
	Then by Lemma~\ref{lemma:the condition for whole sequence, p=2}, for $1\leq v\leq \eta-1$, we have
	\begin{equation}\label{eq:p=2 for v}
		\|\Sup_{t\in \bR_+} \xPsi{s_1,s_2, \cdots,s_v}{t} x\|_2\lesssim K^{[s_1,s_2, \cdots s_v]}\|x\|_2,
	\end{equation}
	with
	\begin{align*}
		&\quad\ K^{[s_1,s_2, \cdots, s_v]}\\
		&=\sum_{j\in \bZ} (a^{[s_1,s_2, \cdots, s_v]}_j)^{1/2}(a^{[s_1,s_2, \cdots, s_v]}_j+b^{[s_1,s_2, \cdots, s_v]}_j)^{1/2}\\
		&\lesssim_{\beta, \eta} \left( \sum_{|j|\leq s_1+s_2+\cdots+s_v} 2^{-\frac{s_1+s_2+\cdots+s_v}{2}}+\sum_{ |j|> s_1+s_2\cdots+s_v} 2^{-\frac{|j|}{2}}\right)  \cdot 2^{-\frac{s_1+s_2+\cdots+s_v}{2}}  \\
		&\lesssim_{\beta, \eta}  \frac{s_1+s_2+\cdots+s_v}{2^{(s_1+s_2+\cdots+s_v)}}.
	\end{align*}	
	Similarly, for $v=\eta$, by \eqref{eq:psi t s1 s2 ...} and \eqref{eq:partical derivative for eta},
	\begin{equation}\label{eq:p=2 for eta}
		\|\Sup_{t\in \bR_+} \xPsi{s_1,s_2, \cdots,s_\eta}{t} x\|_2\lesssim K^{[s_1,s_2, \cdots, s_\eta]}\|x\|_2,
	\end{equation}
	with 
	$$K^{[s_1,s_2, \cdots s_\eta]}\lesssim_{\beta, \eta}   \frac{s_1+s_2+\cdots+s_\eta}{2^{(s_1+s_2\cdots+s_{\eta-1})+\frac{s_\eta}{2}}}.$$

	
	In the following we consider the case  $1\leq q< 2$.
	Denote $$A_q=\sup_{1\leq \delta \leq 2}\|(T_{m_{\delta 2^j}})_{j\in \bZ}:L_q(\cM)\to L_q(\cM;\ell_\infty)\|.$$
	For  $1\leq v\leq \eta-1$, by \eqref{eq:sum epsilon} we have
	\begin{equation}\label{eq:q<2 for v}
		\|{\sup_{t \in I_0}}^+ \xPsi{s_1, s_2, \cdots,s_v}{t} x\|_q \leq \sum_{\bm \varepsilon\in \{0, 1\}^v}  \|{\sup_{t \in I_0}}^+ T_{m_{{\bm \rho}^{\bm \varepsilon} t} }x\|_q\leq 2^v A_q \|x\|_q.
	\end{equation}
	For $v=\eta$, we decompose  $$I_{s_\eta}=\{2^{j/2^{s_\eta}}: j\in \bZ\}=\bigcup_{l=0}^{2^{s_\eta}-1}\{2^{\frac{(2^{s_\eta})j+l}{2^{s_\eta}}}: j\in \bZ\}=\bigcup_{l=0}^{2^{s_\eta}-1} 2^{\frac{l}{2^{s_\eta}}}I_0.$$  
	By \eqref{eq:sum epsilon} and the triangle inequality, we have 
	\begin{align}\label{eq:q<2 for eta}
		\|\Sup_{t\in I_{s_\eta}} \xPsi{s_1,s_2,\dots, s_\eta}{t} (x)\|_q&\leq  \sum_{\bm \varepsilon\in \{0, 1\}^\eta} \|\Sup_{t \in I_{s_\eta}} T_{m_{\bm{\rho^\varepsilon} t}}x\|_q\\
		&\leq  \sum_{\bm \varepsilon\in \{0, 1\}^\eta}\sum_{l=0}^{2^{s_\eta}-1}\|\Sup_{t\in 2^{\frac{l}{2^{s_\eta}}}I_0} T_{m_{\bm{\rho^\varepsilon} t}}x\|_q \nonumber \\
		&\leq  \sum_{\bm \varepsilon\in \{0, 1\}^\eta}\sum_{l=0}^{2^{s_\eta}-1}\|{\sup_{t \in I_0}}^+ T_{m_{2^{\frac{l}{2^{s_\eta}}}   \bm{\rho^\varepsilon}   t}}x\|_q \nonumber\\
		&\lesssim_{\eta} A_q 2^{s_\eta} \|x\|_q . \nonumber
	\end{align}
	
	Now the conclusion follows easily from the complex interpolation. Let $1<p<2$ and $0<\theta<1$ with $\frac{1}{p}=\frac{1-\theta}{q}+\frac{\theta}{2}$. By (\ref{eq:p=2 for v}), (\ref{eq:q<2 for v}) and interpolation, 
	we see that for $v\leq\eta-1$,
	$$\|{\sup_{t \in I_0}}^+ \xPsi{s_1, s_2, \cdots,s_v}{t} x\|_p \lesssim_{\beta, \eta}  A_q^{1-\theta}\left( s_1+s_2\cdots+s_v\right)^{\theta}2^{-\theta(s_1+s_2\cdots+s_v)} \|x\|_p .$$
	By (\ref{eq:p=2 for eta}), (\ref{eq:q<2 for eta}) and interpolation, for $v=\eta$,
	$$\|\Sup_{t\in I_{s_\eta}} \xPsi{s_1,s_2,\dots, s_\eta}{t} (x)\|_p\lesssim_{\beta, \eta}A_q^{1-\theta} 2^{(1-\theta)s_\eta}\left( s_1+s_2\cdots+s_\eta\right)^{\theta}2^{-\theta(s_1+s_2\cdots+\frac{s_\eta}{2})}\|x\|_p.$$
	We 	apply  the above estimate to  (\ref{eq:Phi t decomposition to s1 s2...}).
	Note that	when $v\leq \eta-1$, 
	\begin{align*}
		&	\sum_{s_1>s_2>\cdots>s_v} \left( s_1+s_2\cdots+s_v\right)^{\theta}2^{-\theta(s_1+s_2\cdots+s_v)}\\
		\lesssim_{\theta} &\sum_{s_2>\cdots>s_v} \sum_{s_1=s_2+1}^\infty 2^{-\theta(s_1+s_2\cdots+s_v)/2}
		\lesssim_{\theta}  \sum_{s_2> \cdots>s_v} 2^{-\theta s_2/2}2^{-\theta(s_2+ \cdots+s_v)/2}\\
		\lesssim_{ \theta} &\sum_{s_3> \cdots>s_v} 2^{-\theta s_3}2^{-\theta(s_3+ \cdots+s_v)/2}
		\lesssim_{ \theta}	\cdots		\lesssim_{\theta}   \sum_{s_v\geq 0} 2^{-(v-1)\theta s_v/2}
		\lesssim_{\theta}  1.
	\end{align*}
	Applying the	similar computation to $v=\eta$, we have
	\begin{align*}
		&\quad\ \sum_{s_1>s_2>\cdots>s_\eta}2^{(1-\theta)s_\eta}\left( s_1+s_2\cdots+s_\eta\right)^{\theta}2^{-\theta(s_1+s_2\cdots+\frac{s_\eta}{2})}\\
		&\lesssim_{\theta}\sum_{s_\eta\geq 0} 2^{(1-\theta)s_\eta}(s_\eta+1)^\theta 2^{-(\eta-\frac{1}{2}) \theta s_\eta}.
	\end{align*}
	Thus the above quantity is finite if $(1-\theta)<(\eta-\frac{1}{2}) \theta$,  i.e. $\theta>\frac{2}{2\eta+1}$, which requires that $q+\frac{q(2-q)}{q-1+2\eta}< p\leq 2$. 
	Therefore, together with (\ref{eq:Phi t decomposition to s1 s2...}), if  $q+\frac{q(2-q)}{q-1+2\eta}< p\leq 2$,
	$$\|\Sup_{t \in \bR_+} T_{m_t} x \|_p\lesssim_{\beta,\eta, \theta} A_q^{1-\theta} \|x\|_p.$$
	The proof is complete.
\end{proof}

\begin{proof}[Proof of Theorem~\ref{theorem:criterion2}~(2)]

	For any $1\leq \delta \leq 2$ and $j\in \bZ$, set  $ t_j= 2^j\delta $. 
	\textbf{(A2)} implies that
	$$|1-m_{t_j}(\omega)|\leq \beta\frac{\ell(\omega)^\alpha}{ 2^j\delta}\leq \beta \frac{\ell(\omega)^\alpha}{2^j},\quad |m_{t_j}(\omega)|\leq 2\beta \frac{2^j}{\ell(\omega)^\alpha}.$$ 
	By Theorem~\ref{theorem:criterion2}~(1) and Remark~\ref{rmk: A1 prime and change index set N to Z}~(2), for any  $1<q\leq 2$, we have 
	$$\sup_{1\leq \delta\leq2} \|{\sup_{j\in \bZ}}^+ T_{ m_{t_j}}x\|_{q}\lesssim_{\alpha, \beta, q, \eta} \|x\|_{q}.$$
	Note that $q+\frac{q(2-q)}{q-1+2\eta}$ tends to $1+\frac{1}{2\eta}$ as $q\to 1$.  Hence, by Lemma~\ref{prop:from 2j to whole t}, for any $1+\frac{1}{2\eta}<p\leq 2$, 
	we get 
	$$\|\Sup_{t\in \bR_+} T_{m_t}x\|_p\lesssim_{\alpha, \beta, p, \eta} \|x\|_p.$$
	The a.u. (b.a.u.) convergence is proved similarly as in  (1).  		
\end{proof}

\section{The case of operator-valued multipliers and nontracial states}	
Based on the previous arguments, we may extend our results to the setting of operator-valued multipliers and Haagerup's nontracial $L_p$-spaces. This will be   particularly essential for our further study of multipliers on quantum groups in the next chapter. All the previous arguments for $p\geq 2$ can be transferred without difficulty into this new setting, and we will leave the details to interesting readers. However, the previous proof for the case  $p< 2$ does not continue to hold for Haagerup's $L_p$-spaces. Based on Haagerup's reduction method, we will rather use our previous results for the tracial setting   to deduce the desired properties for the nontracial ones.

\subsection{Operator-valued multipliers in the tracial setting}	
Let us first begin with the operator-valued multipliers on tracial von Neumann algebras. Let $\cR$ be a von Neumann algebra equipped with a semifinite normal trace. Assume that there is   an isometric isomorphism $$U:L_2(\cR)\to   \bigoplus_{i\in I} H_i $$ where $I$ is an index set and $H_i$ is a Hilbert space for each $i\in I$. 
Additionally, we will always assume that $\{U^{-1}(X_i)_{i\in I}: X_i = 0 \text{ except  finitely many }i\in I\}$ is dense in $L_p(\cR)$ for all $1\leq p\leq \infty$.
For any bounded sequence $m:=(m(i))_{i\in I}$    with $m(i)\in B(H_i)$,  we can define an \textit{operator-valued multiplier} on $\cR$:
\begin{align}
T_m: L_2(\cR)&\to L_2(\cR) \nonumber \\
x&\mapsto U^{-1}\left(  m(i)  (Ux)(i) \right)_{i\in I} . \label{eq:def of multiplier+} 
\end{align}
Note that if $H_i=\mathbb C$ for all $i\in I$, then this goes back to our first setting in \eqref{eq:definition of multiplier operoter} with $\Omega=I$ equipped the counting measure.
Proposition~\ref{prop:general case for p=2} can be adapted to this new setting.
\begin{prop}\label{prop:general case for p=2+}
	Let $(T_{m_N})_{N\in \bZ}$ be a sequence of operator-valued multipliers as above. Assume that there exist a  function $f:I \to [0, \infty)$ and a positive number $a>1$ such that
	\begin{equation}
	\|m_N(i)\|_{B(H_i)}\leq \beta  \frac{a^Nf(\omega)}{(a^N+f(\omega))^2}
	\end{equation}
	Then, 
	$$\| (T_{m_N}x)_{N\in\bZ}\|_{L_2(\cR; \ell_2^{cr})}\leq \beta \sqrt{\frac{a^2}{a^2-1}}\|x\|_2.$$
\end{prop}
\begin{proof}
	Repeat  the proof of Proposition~\ref{prop:general case for p=2}.
\end{proof}

Using Proposition~\ref{prop:general case for p=2+}, Lemma~\ref{lemma: strong pp strong 22 imply restricted  qq},  Lemma~\ref{lemma: A(p,2)<A(p/p-2) sqrt} and  Lemma~\ref{lemma:Strong (p,p) of Delta_j},  we  may deduce the following result. The proof is the same as that of  Theorem~\ref{theorem:criterion1} and  Theorem~\ref{theorem:criterion2} (for \textbf{(A1)}).

\begin{theorem}\label{theorem:criterion2+: tracial noncentral case}
	Let $\ell:=(\ell(i))_{i\in I}$ be a sequence  with   $\ell(i)\in B(H_i)$. Assume that $(T_{e^{-t\ell}})_{t\in \bR_+}$ is a semigroup of unital completely positive trace-preserving symmetric maps  on $\cR$.	For any ${N\in \bN}$, let $m_N:=(m_N(i))_{i\in I}$ be a sequence  with $m_N(i)\in~B(H_i)$. Assume that $(T_{m_N})_{N\in \mathbb N}$ extends to a family of bounded maps on $ \cR $ with $\gamma:=\sup_N \|T_{m_N}: \cR\to \cR\|<\infty$. Assume that there exist $\alpha >0$ and $\beta >0$ such that for  all $i\in I$ we have
	\begin{equation}\label{eq:criterion a1+} 
	\|\id_{H_i}-m_N (i)\|_{B(H_i)} \leq \beta \frac{\|\ell(i)\|_{B(H_i)}^\alpha}{2^N},\quad \|m_N (i)\|_{B(H_i)}\leq \beta \frac{2^N}{\|\ell(i)\|_{B(H_i)}^\alpha}. 
	\end{equation}
	
	\emph{(1)} For all $2\leq p<\infty$ there is a constant $c>0$ depending only on $p,\alpha,\beta, \gamma$ such that for all $x\in L_p (\mathcal R)$, we have
	$$\|(T_{m_N} x)_N\|_{L_p(\cR;\ell_\infty)} \leq c \|x\|_p, \qquad\text{and}\qquad  T_{m_N} x\to x \text{ a.u. as }N\to\infty.$$

	\emph{(2)}	Assume  additionally that the operators $(T_{m_N})_{N\in \mathbb N}$ extend to  positive symmetric contractions on $  \cR$  and that $S_t$ satisfies Rota's dialtion property  for all $t \in \bR_+$.  Then for any $1< p<\infty$, there is a constant $c$ depending only on $p,\alpha,\beta$  such that  for all $x\in L_p (\cR)$,
	$$\|(T_{m_N} x)_N\|_{L_p(\cR;\ell_\infty)} \leq c \|x\|_p  \quad\text{and}\quad  T_{m_N} x\to x \text{ b.a.u. (a.u. if $p\geq 2$), as }N\to\infty.$$ 
\end{theorem}

\subsection{Operator-valued multipliers on Haagerup noncommutative $L_p$-spaces}\label{sect:nontracial matrix}
Let $\cM$ be a von Neumann algebra acting on a Hilbert space $H$. Let $\varphi$ be a fixed normal semifinite faithful state on $\cM$. Let $\sigma=(\sigma_t)_t=(\sigma^\varphi_t)_t$ be the modular automorphism group with respect to $\varphi$.

Let $\hlp{p}$ be the Haagerup noncommutative $L_p$-spaces associated with  $(\mathcal M ,\varphi)$. In the following discussions we will not need the detailed information of these spaces, and we refer to \cite{terp81lpspace,pisierxu03nclp,haagerupjungexu10reduction} for a detailed presentation. We merely remind the reader that
the elements in $\hlp{p}$ can be realized as  densely defined closed operators on $L_2(\bR ; H)$ and that
$\hlp{\infty}$ coincides with $\cM$ for a certain suitable representation $\cM$ on $L_2(\bR ; H)$.  If $\mathcal N $ is a von Neumann subalgebra of $\mathcal M$, then the associated Haagerup $L_p$-space $L_p(\mathcal N ,\varphi|_{\mathcal N})$ can be naturally embedded as a subspace of $\hlp{p}$ which preserves positivity.
There is a distinguished positive element $D_\varphi \in L_1 (\mathcal M ,\varphi)_+$, usually called the \emph{density} operator associated with $\varphi$, such that $D_\varphi^{1/2p} \cM D_\varphi^{1/2p}$ is dense in $\hlp{p}$  (see e.g. \cite[Lemma 1.1]{jungexu03burkholder}).
The space  $\hlp{p}$   isometrically coincides with the usual tracial noncommutative $L_p$-space used previously if $\mathcal M$ is tracial. Indeed, if $\mathcal M$ is equipped with a normal faithful tracial state $\tau$ with $\varphi=\tau(\cdot \rho )$ for some $\rho\in L_0 (\mathcal M, \tau)$, then for any $x\in \mathcal M$, 
\begin{equation}\label{eq:haggeruptracial}
\|D_\varphi^{1/2p}x D_\varphi^{1/2p}\|_{\hlp{p}}  = [\tau ( (\rho^{1/2p}x \rho^{1/2p})^p  ) ]^{1/p}  ,
\end{equation}
which coincides with the norm of $\rho^{1/2p}x \rho^{1/2p}$ in the tracial $L_p$-space $L_p (\mathcal M ,\tau)$ in the sense of Chapter \ref{sec:preliminaries}. In the sequel we will not distinguish the Haagerup $L_p$-spaces and the tracial $L_p$-spaces introduced in Chapter \ref{sec:preliminaries} if $\varphi$ is tracial.

In this section we set $\Gamma=\bigcup_{n\geq 1}2^{-n}\bZ$, which is regarded as a discrete subgroup of $\bR$. We consider the crossed product ${\cR}=\cM\rtimes_\sigma \Gamma$. 
Recall that ${\cR}$ is the von Neumann subalgebra    generated by $\pi(\cM)$ and $\id_H \otimes \lambda(\Gamma)$ in $B(\ell_2(\bR; H))$, where $\pi:\cM\to B(\ell_2(\Gamma; H))$ is the $*$-representation given by $\pi(x)=\sum_{t \in \Gamma} \sigma_{-t} (x)\otimes e_{t,t}$ and $\lambda$ is the left regular representation of $\Gamma$ on $\ell_2(\Gamma)$. 
We will identify $\cM$ with $\pi(\cM)$ and denote $x\rtimes \lambda(t) =\pi(x) (\id_H \otimes \lambda(t))$ for $ x\in \cM$ and $ t \in \Gamma$.  We have 
$$(x\rtimes \lambda(t))\cdot (y\rtimes \lambda(s))=(x\sigma_t(y))\rtimes \lambda(t+s)$$ for any $x, y\in \cM$, $t, s\in \Gamma$.
Let $\tau_\Gamma$ be the usual trace on $VN(\Gamma)$ given by ${\tau_\Gamma (\lambda(t)) = \delta_{t=0}}$. The dual state $\hatphi$ on ${\cR}$ is defined by 
\begin{equation}\label{eq:def of hatphi}
\hatphi(x\rtimes \lambda(t))=\varphi (x) \tau_\Gamma (\lambda(t))
\qquad x\in \cM, \quad t \in \Gamma.
\end{equation}
We set
$$a_k = - \mathrm i 2^n \mathrm{Log} (\lambda (2^{-k}))\quad \text{and} \quad
\tau_k = \widehat{\varphi} (e^{-a_k}\, \cdot)$$
with $ \mathrm{Log} $ the principal branch of the logarithm so that $ 0\leq  \mathrm{Im}(\mathrm{Log}(z)) < 2\pi $. We denote by $\mathcal R _k$ the  centralizer of $\tau_k$ in $\mathcal R$. 
We will use the following two  theorems to reduce our problem to the tracial case studied previously. 
\begin{theorem}[{\cite[Theorem 2.1, Example 5.8, Remark 6.1]{haagerupjungexu10reduction}}]\label{theorem:reduction theorem HJX}
	\emph{(1)} For each $k\geq 1$, the subalgebra $\mathcal R _k$  is finite and $\tau_k$ is a normal faithful tracial state on $\mathcal R _k$;
	
	\emph{(2)}   $\{   \cR _k  \}_{k\geq 1}$ is an increasing sequence of von Neumann subalgebras such that $\cup_{k\geq 1}  \cR _k   $ is w*-dense  in  $\cR$; 
	
	\emph{(3)} for every $k\in \bN$, there exists a normal   conditional expectation $\bE_k$ from $\cR$ onto
	$ \cR_k$ such that 
	$$\hatphi \circ \bE_k=\hatphi \add{and} \sigma_t^\hatphi\circ \bE_k=\bE_k\circ \sigma_t^{\widehat{ \varphi}},  \quad t\in \bR.$$
	For each $1\leq p<\infty $ and any $k\in \bN$, the map 
	$$\bE_k ^{(p)} (D_{\widehat{\varphi}}^{1/2p}x D_{\widehat{\varphi}}^{1/2p}) = D_{\widehat{\varphi}}^{1/2p}\bE_k(x) D_{\widehat{\varphi}}^{1/2p},\quad
	x\in \cR$$ extends to a conditional expectation from $L_p (\cR, \widehat \varphi)$ onto  $L_p (\cR _k, \widehat \varphi |_{\cR _k})$, and  $$\lim_k \| \bE_k ^{(p)} x -x \|_{L_p (\cR, \widehat \varphi)} =0,\quad x\in L_p (\cR, \widehat \varphi).$$ 
\end{theorem}
\begin{theorem}[{\cite[Theorem 4.1, Proposition 4.3 and Theorem 5.1]{haagerupjungexu10reduction}}]\label{theorem:reduction theorem HJX-etension}
	Assume that $T:\cM\to \cM$ is a completely bounded normal map such that \begin{equation}\label{eq: Tm commutes with sigma}
	T\circ \sigma_t=\sigma_t\circ T, \qquad t\in \mathbb R.
	\end{equation}
	Then $T$ admits a unique completely bounded normal extension $\widehat{T}$ on $\cR$ such that 
	$$
	\|\widehat{T}\|_{cb}=\|T\|_{cb}\add{and} \widehat{T}(x\rtimes \lambda(g))=T(x)\rtimes \lambda(g), \qquad x\in \cM, g\in \Gamma.
	$$
	Moreover, $\widehat{T}$ satisfies the following properties:
	
	\emph{(1)} $\sigma_t^\hatphi\circ \widehat{T}=\widehat{T}\circ \sigma_t^{\hat{ \varphi}}, t\in \bR$;
	
	\emph{(2)} $\widehat{T}\circ \bE_k (x)= \bE_k \circ \widehat{T}(x)$ for all $x\in \cR$, where $(\bE_k)_k$ are  conditional expectations given in Theorem~\ref{theorem:reduction theorem HJX}. 
	
	Assume in addition that $T$ is completely positive and   $\varphi $-preserving. Then $\widehat{T}$ is also positive and $\tau_k\circ \widehat{T} =\tau_k$,
	$\hatphi\circ \widehat{T} =\hatphi$ where $\tau_k$ is the trace given in Theorem~\ref{theorem:reduction theorem HJX}. Moreover, the map 
	$$\widehat{T}^{(p)} : D_\hatphi ^{1/2p} x D_\hatphi ^{1/2p}\mapsto  D_\hatphi ^{1/2p}\widehat{T}(x)  D_\hatphi ^{1/2p},\quad
	x\in \mathcal R$$ 
	extends to a positive bounded maps on $\hlpr{p}$ for all $1\leq p\leq \infty$.
\end{theorem}	
\textbf{Convention.} In the sequel, for a given map $T:\mathcal M \to \mathcal M$, we will denote, by the same symbol $ T $, all the maps
$ \widehat{T}^{(p)} $ and   their extensions to the $L_p$-spaces in the above setting, whenever no confusion can occur.

\medskip

Let $H_\varphi$ be the GNS completion of $\mathcal M$ with respect to  $\varphi$ (we make the convention that $\|x\|_{H_\varphi}^2 = \varphi(x^* x)$ for $x\in \mathcal M$). Note that $x\mapsto xD_\varphi ^{1/2}$ yields an isometric isomorphism from $H_\varphi$ to $L_2 (\cM ,\varphi)$. Assume that there is   an isometric isomorphism 
\begin{equation}\label{eq:isom u}
U:H_\varphi\to   \bigoplus_{i\in I} H_i
\end{equation}    where $I$ is an index set and $H_i$ is a Hilbert space for each $i\in I$. 
Let $m:=(m(i))_i$ be a bounded sequence   with $m(i)\in B(H_i)$. As in  (\ref{eq:def of multiplier+}),  we may define the multiplier 
\begin{equation} \label{eq:nontracialoperatormultiplier}
T_m : H_\varphi \to H_\varphi, \quad
U(T_m x ) = (m(i)(Ux)(i))_{i\in I},\quad x\in H_\varphi.
\end{equation}

Applying the reduction theorems quoted above, we obtain the maximal inequalities for the nontracial setting. The a.u. convergence can be adapted in the setting of Haagerup's $L_p$-spaces, usually called Jajte's \emph{(bilaterally) almost sure} (b.a.s. and a.s. for short) convergence \cite{jajte91aus},
for which we also refer to \cite[Section 7.4]{jungexu07ergodic}.
\begin{definition}
	(1)	Let $x_n, x\in L_p(\cM,\varphi)$ with $1\leq p<\infty$. The sequence  $(x_n)$ is said to converge almost surely (a.s. in short) to $x$ if for every $\varepsilon>0$ there is a projection $e\in \cM$ and a family $(a_{n,k})\subset \cM$ such that
	$$ \varphi(e^{\perp})<\varepsilon\add{and} x_n-x=\sum_{k\geq 1}a_{n,k}D^{\frac{1}{p}},  \qquad \lim_{n\to \infty} \|\sum_{k\geq 1}\left(a_{n,k}e\right)\|_\infty=0, $$
	where the two series converge in norm in $L_p(\cM,\varphi)$ and $\cM$, respectively.
	
	(2)	Let $x_n, x\in L_p(\cM,\varphi)$ with $1\leq p<\infty$. The sequence  $(x_n)$ is said to converge bilateral almost surely (b.a.s. in short) to $x$ if for every $\varepsilon>0$ there is a projection $e\in \cM$ and a family $(a_{n,k})\subset \cM$ such that
	$$ \varphi(e^{\perp})<\varepsilon\add{and} x_n-x=D^{\frac{1}{2p}}\sum_{k\geq 1}a_{n,k}D^{\frac{1}{2p}},  \qquad \lim_{n\to \infty} \|\sum_{k\geq 1}\left(ea_{n,k}e\right)\|_\infty=0, $$
	where the two series converge in norm in $L_p(\cM,\varphi)$ and $\cM$, respectively.
\end{definition}
As we mentioned at the beginning of this section,  the  space $L_p(\cM, \varphi)$ isometrically coincides with the  tracial noncommutative $L_p$-space if the state $\varphi$ is tracial. In this case, one can easily verify that Jajte's a.s. (resp. b.a.s) convergence recovers Lance's a.u. (resp. b.a.u.) convergence defined in Definition~\ref{def:auconvergence}.
\medskip

We keep the notation introduced previously in this section. The following is our main result in this section, which generalizes the results for \textbf{(A1)} in Theorem \ref{theorem:criterion1} and Theorem \ref{theorem:criterion2}. Those for \textbf{(A2)} can be dealt with in a similar manner, and we leave the details to interesting readers.
\begin{theorem}\label{theorem:criterion2+: Haagerup case}
	Let  $\ell:=(\ell(i))_{i\in I}$ be a  sequence   with   $\ell(i)\in B(H_i)$ and write  $S_t=T_{e^{-t\ell}}$ for $t\in \bR_+$.. For any ${N\in \bN}$, let $m_N:=(m_N(i))_{i\in I}$ be a bounded sequence with $m_N(i)\in B(H_i)$. Assume that the following conditions hold:
	
	\emph{(i)} $(S_t)_{t\in \bR_+}$ extends to a semigroup of unital completely positive $\varphi$-preserving  maps on $\cM$ and for any $t\in \bR_+$, $r\in \bR$,
	$$	S_{t}\circ \sigma_{r}=\sigma_{r}\circ S_{t},  \add{ } \varphi(S_t(x)^*y)=\varphi(x^*S_t(y))\qquad x, y\in \cM.$$ 
	
	\emph{(ii)} $(T_{m_N})_{N\in \mathbb N}$ extends to a family of selfadjoint maps on  $ \cM$ with $$\gamma:=\sup_N\|T_{m_N}: \cM\to \cM\|<\infty.$$
	
	\emph{(iii)} There exist $\alpha >0$ and $\beta >0$ such that for  all $i\in I$ we have
	\begin{equation}\label{eq:A1 condtion+}
	\|\id_{H_i}-m_N (i)\|_{B(H_i)} \leq \beta \frac{\|\ell(i)\|_{B(H_i)}^\alpha}{2^N},\quad \|m_N (i)\|_{B(H_i)}\leq \beta \frac{2^N}{\|\ell(i)\|_{B(H_i)}^\alpha}. 
	\end{equation}
	Then there is a constant $c>0$ depending only on $p,\alpha,\beta, \gamma$ such that  for all ${2\leq  p<\infty}$,
	\begin{equation}\label{eq:max non tracial}
	\|(T_{m_N} x)_N\|_{L_p(\cM;\ell_\infty)} \leq c \|x\|_p  \qquad  x\in \hlp{p}.
	\end{equation} 
	and $T_{m_N} x\to x$ a.s. (resp. b.a.s.) as $N\to\infty$ for all $x\in L_p(\cM)$ with $2< p< \infty$ (resp. $p=2$).
	
	If in addition the maps $(T_{m_N})_{N\in \mathbb N}$ are unital completely positive, symmetric and $\varphi$-preserving on $\cM$ and commute with the modular automorphism group $\sigma$, i.e.
	\begin{equation*}
	T_{m_N}\circ \sigma_r=\sigma_r\circ T_{m_N} \qquad r\in \bR,  N\in \bN,
	\end{equation*} 
	then the above maximal inequality \eqref{eq:max non tracial} also holds for all  $1<  p<\infty$ and ${T_{m_N} x\to x}$ b.a.s. as $N\to\infty$ for all $x\in L_p(\cM)$. 
\end{theorem}	
The proof of \eqref{eq:max non tracial} for the case of $p\geq 2$ is  a \emph{mutatis mutandis} copy of the arguments in previous sections. It suffices to note that the proof of Proposition \ref{prop:general case for p=2} and Proposition \ref{prop:general case for p=2+} remains valid in the setting of Haagerup's $L_p$-spaces if $T_{m_N}$ is selfadjoint on $\cM$. We leave the details to interesting readers. The reason why the 
previous arguments do not adapt to the case of $p< 2$ is that the weak interpolation (Theorem \ref{theorem:interpolation of dirksen}) fails for Haagerup's $L_p$-spaces. So we will provide a proof for this case using the reduction theorems. On the other hand, we will only prove the maximal inequalities.  The implication from maximal inequalities to a.s. (resp. b.a.s.) convergences, in particular the analogue of Proposition \ref{prop:Phix-x in Lp M C0} (2), remains valid on $\hlp{p}$ if we replace the one sided weak type inequality \eqref{eq:one sided weak} by the strong type one on $L_p(\cM;\ell_\infty ^c)$ for $p>2$ (resp. on $L_p(\cM;\ell_\infty)$ for $1<p\leq 2$) by using  \cite[Lemma 7.10]{jungexu07ergodic}.
\begin{proof}[Proof of \eqref{eq:max non tracial} for $1<p<2$ and completely positive $T_{m_N}$]
	The operator $U$ in \eqref{eq:isom u} induces an isometry on $L_2 (\mathcal R _k , \tau_k)$ given by 
	$$
	\begin{array}{rccc}
	\widehat{U} _k:& L_2 (\mathcal R _k , \tau_k)&\to& \bigoplus_{i\in I} H_i \otimes L_2( VN(\Gamma), \tau_\Gamma(e^{-a_k} \,\cdot) )\\
	&x\rtimes \lambda(g)&\mapsto& U(x)\otimes \lambda(g) .
	\end{array}
	$$
	Indeed, $\widehat{U} _k$  is an isometry since for any finite sum $\sum_g x_g\rtimes \lambda(g)\in \cR$, we have
	\begin{align*}
	&\quad \|\widehat{U} _k(\sum_g x_g\rtimes \lambda(g))\|_{\bigoplus_{i\in I} H_i \otimes L_2( VN(\Gamma), \tau_\Gamma(e^{-a_k} \,\cdot)  )}^2\\
	&=\sum_{g,h}\varphi(x_g^*x_h)\tau_\Gamma(e^{-a_k}\lambda(h-g))
	=\sum_{g,h}\varphi(\sigma_{-g}(x_g^*x_h))\tau_\Gamma(e^{-a_k}\lambda(h-g)) \\
	&=\sum_{g,h}\hatphi \left( \sigma_{-g}(x_g^*x_h)\rtimes \lambda((h-g)e^{-a_k} \right)
	=\|\sum_g x_g\rtimes \lambda(g)\|_{L_2 (\mathcal R _k , \tau_k)}^2.
	\end{align*}
	Take the Hilbert subspaces $H_i ' \subset  H_i \otimes L_2( VN(\Gamma), \tau_\Gamma(e^{-a_k} \,\cdot)  )$ so that we have ${\mathrm{ran} (\widehat{U} _k) = \bigoplus_i H_i '}$. Then $\widehat{U} _k : L_2 (\mathcal R _k , \tau_k) \to \bigoplus_i H_i '$ becomes an isometric isomorphism.  
	For any $x\rtimes \lambda(g)\in \cR_k$, the element $(T_{m_N} x)\rtimes \lambda(g)=\widehat{T}_{m_N} (x\rtimes \lambda(g))$ also belongs to $ \cR_k$ since $\widehat{T}_{m_N}\circ \bE_k=\bE_k \circ \widehat{T}_{m_N}$ by Theorem \ref{theorem:reduction theorem HJX-etension}.	
	So $\widehat{U}_k((T_{m_N} x)\rtimes \lambda(g))$ is well-defined and moreover, by (iii), we have
	\begin{align*}
	\widehat{U}_k((T_{m_N} x)\rtimes \lambda(g)) 
	&= U(T_{m_N}(x))\otimes \lambda(g)
	=\left(  m U(x) \right) \otimes \lambda(g)
	\\
	&=(m\otimes \mathrm{id}) \widehat{U}_k(  x \rtimes \lambda(g)) .
	\end{align*}
	In particular $m_N(i)\otimes\mathrm{id}$ sends $H_i '$ into $H_i '$, and $\widehat{T}_{m_N} |_{\mathcal R _k}$ is an operator-valued multiplier in the sense of \eqref{eq:def of multiplier+} (recall that $(\cR_k, \tau_k)$ is tracial). It is straightforward to verify that $\widehat{T}_{m_N} |_{\mathcal R _k}$ is unital completely positive $\tau_k$-preserving on $\mathcal R _k$, and therefore it extends to a positive contraction on $L_p(\cR_k, \tau_k)$. Similarly, the extension  $\widehat{S_t}:=\widehat{T}_{e^{-t\ell}}$ also gives rise to a semigroup of unital completely positive maps on $\cR$. It is easy to check that $\widehat{S}_t$  is symmetric relative to $\widehat{\varphi}$. The restriction $\widehat{S_t}|_{\cR_k}$ is $\tau_k$-preserving and symmetric relative to $\tau_k$ since  $$\widehat{S}_{t} ((x\rtimes \lambda(g)) e^{-a_k}) = (S_t x) \rtimes \lambda(g) e^{-a_k} =(\widehat{S}_t (x\rtimes \lambda(g)) )e^{-a_k} $$ for all $x\in \mathcal M$ and $g\in \Gamma$. Thus applying Theorem~\ref{theorem:criterion2+: tracial noncentral case} to $\widehat{T}_{m_N} |_{\mathcal R _k}$, we obtain
	\begin{equation}\label{eq:maximal for tracial case Rk}
	\|{\sup_N}^+ \ \widehat{T}_{m_N}(x)\|_{L_p(\cR_k,\tau_k)}\leq c\|x\|_{L_p(\cR_k,\tau_k)},\quad 
	x\in L_p(\cR_k,\tau_k),
	\end{equation}
	where $c$ is a constant only depending on $\alpha, \beta, p$.
	
	%
	%
	
	In the following we consider  $x\in L_p(\cM,\varphi)_+$. Since $L_p(\cM, \varphi)$ can be naturally embedded into $L_p(\cR, \hatphi)$,   we regard $x$ as an element in $L_p(\cR, \hatphi)_+$.  
	By Theorem~\ref{theorem:reduction theorem HJX-etension}, we see that $ \widehat{T}_{m_N}\circ \bE_k = \bE_k \circ \widehat{T}_{m_N}$ and hence  $\widehat{T}_{m_N}^{(p)}\circ \bE_k^{(p)} = \bE_k^{(p)} \circ \widehat{T}_{m_N}^{(p)} $.  
	By Theorem \ref{theorem:reduction theorem HJX}, we have
	$$\lim _{k\to \infty}  \widehat{T}^{(p)}_{m_N}(\bE_k^{(p)}(x))=\lim _{k\to \infty}  \bE_k^{(p)}(\widehat{T}^{(p)}_{m_N}(x))={T}_{m_N}^{(p)}(x)   \quad \text{in} \quad L_p(\cR, \hatphi).$$
	Thus for any $M>0$,
	\begin{equation}\label{eq:sup martingale}
	\lim_{k\to \infty}\|{\Sup_{1\leq N\leq M}}\ \widehat{T}^{(p)}_{m_N}(\bE_k^{(p)}(x))\|_{L_p (\cR _k, \widehat \varphi |_{\cR _k})}=\|{\Sup_{1\leq N\leq M}}\ \widehat{T}_{m_N}^{(p)}(x)\|_{L_p(\cR, \hatphi)}.
	\end{equation} 
	Without loss of generality, we assume that $x=\ddp{y}$ with some $y\in \mathcal R_+$. By the correspondence in \eqref{eq:haggeruptracial}, we get 
	$$\|{\Sup_{1\leq N\leq M}}\ \widehat{T}^{(p)}_{m_N}(\bE_k^{(p)}(x))\|_{L_p (\cR _k, \widehat \varphi |_{\cR _k})}=\|{\Sup_{1\leq N\leq M}}\ \aap{\widehat{T}_{m_N}(\bE_k(y))}\|_{L_p(\cR_k, \tau_k)}.$$
	Note that $e^{\frac{a_k}{2p}}$ belongs to the subalgebra generated by $1\rtimes \lambda(\Gamma)$. 
	For any element of the form $z=\sum_{g\in \Gamma} z(g)\rtimes \lambda(g)$ in $ \cR_k$, we have
	\begin{equation*}
	e^{\frac{a_k}{2p}}\widehat{T}_{m_N}(z) e^{\frac{a_k}{2p}}=\sum_{g\in \Gamma} T_{m_N}(z(g))\rtimes (e^{\frac{a_k}{2p}}\lambda(g)e^{\frac{a_k}{2p}})=\widehat T_{m_N}({e^{\frac{a_k}{2p}}ze^{\frac{a_k}{2p}}}).
	\end{equation*}
	So the previous equality reads 
	$$\|{\Sup_{1\leq N\leq M}}\ \widehat{T}^{(p)}_{m_N}(\bE_k^{(p)}(x))\|_{L_p (\cR _k, \widehat \varphi |_{\cR _k})}=\|{\Sup_{1\leq N\leq M}}\ \widehat{T}_{m_N}(\aap{ \bE_k(y) })\|_{L_p(\cR_k, \tau_k)}.$$
	Together with \eqref{eq:maximal for tracial case Rk}, \eqref{eq:sup martingale} and \eqref{eq:haggeruptracial} we obtain
	\begin{align*}
	\|{\Sup_{1\leq N\leq M}}\ \widehat T_{m_N}^{(p)}(x)\|_{L_p(\cR, \hatphi)}&\leq \lim_{k\to\infty }\|\aap{ \bE_k(y) }\|_{L_p(\cR_k, \tau_k)}\\
	&= \lim_{k\to\infty }\|D_{\widehat{\varphi}}^{1/2p} \mathbb E _k (y) D_{\widehat{\varphi}}^{1/2p} \|_{L_p (\cR _k, \widehat \varphi |_{\cR _k})}
	\\
	&= \lim_{k\to\infty }\|\bE_k^{(p)}(x)\|_{L_p (\cR _k, \widehat \varphi |_{\cR _k})}\\ 
	&= \|x\|_{L_p(\cR, \hatphi)}.
	\end{align*} 
	The proof is complete.
\end{proof}	
\begin{remark}\label{rem:au on m}
	Lance's notion of a.u. convergence still makes sense for $p=\infty$ in the nontracial setting. The above theorem   also implies that $T_{m_N} x\to x$ { a.u. as }$N\to\infty$ for all $x\in  \cM $, according to \cite[Lemma 7.13]{jungexu07ergodic}.
\end{remark}

\chapter{Convergences of Fourier series on  quantum groups}\label{sec:multipliers on CQG}
In this chapter we will construct some summation methods satisfying the pointwise convergence for groups with suitable approximation properties. As in the previous chapter, only the framework of the form \eqref{eq:definition of multiplier operoter} (or the form \eqref{eq:def of multiplier+}) is involved in the essential part of our arguments.
Since the approximation properties of discrete quantum groups have drawn wide interest in recent years, we would like to present the work in a more general setting, that is, Woronowicz's \emph{compact quantum groups}. 

\section{Fourier multipliers on compact quantum groups}
Before proving the main results, let us collect some preliminary facts on Fourier multipliers on compact quantum groups. We refer to \cite{woronowicz87matrix,woronowicz98compact,timmermann08qgbook} for a complete description of compact quantum groups. In this paper, it suffices to recall that each compact quantum group $\mathbb G$ is an object corresponding to a distinguished von Neumann algebra denoted by $L_\infty (\mathbb G )$, a unital normal $*$-homomorphism $\Delta:L_\infty (\mathbb G ) \to L_\infty (\mathbb G ) \overline{\otimes} L_\infty (\mathbb G )$ (usually called the \emph{comultiplication}), and a normal faithful state $h:L_\infty (\mathbb G ) \to \mathbb C$ (usually called the \emph{Haar state}) with the following properties. First, the Haar state $h$ is invariant in the sense that
$$(h\otimes\id)\circ\Delta(x)=h(x)1=(\id\otimes h)\circ\Delta(x), \qquad   x\in L_\infty (\mathbb G ).$$
Second, a unitary $n\times n$ matrix $u=(u_{ij})_{i,j=1}^n$ with coefficients $u_{ij} \in  L_\infty (\mathbb G )$ is called an \emph{$n$-dimensional unitary representation} of $\bG$ if for any $1\leq i,j\leq n$ we have
$$\Delta(u_{ij})=\sum_{k=1}^nu_{ik}\otimes u_{kj}.$$
We denote by $\mathrm{Irr}(\bG)$ the collection of unitary equivalence classes of  irreducible  representations of $\bG$, and we fix a  representative $u^{(\pi)}$ on a Hilbert space $H_\pi$ for each  class $\pi\in \Irr(\bG)$ and denote by $d_\pi$ its dimension. In particular, we denote by $\mathbf{1}\in \Irr(\bG)$ the trivial representation, i.e. $u^{\mathbf{1}}=1_{\bG}$ with dimension 1.
Then the space $$\Pol(\bG)=\text{span}\{\uu{ij}{\pi}: \uu{ }{\pi}=(\uu{ij}{\pi})_{i,j=1}^{d_\pi} , \pi\in \mathrm{Irr}(\bG) \}$$
is a w*-dense $*$-subalgebra of $L_\infty (\mathbb G )$.  
We denote by  $L_p(\bG)$ the Haagerup noncommutative $L_p$-spaces associated with    $ (L_\infty(\bG),h)$. 
Last, there is a linear antihomomorphism $S$ on $\Pol(\bG)$, called the \textit{antipode} of $\bG$, determined by
$$S(\uu{ij}{\pi})=(\uu{ji}{\pi})^*\qquad \pi \in \Irr(\bG), 1\leq i, j \leq d_\pi.$$
The antipode $S$ has the polar decomposition
$S=R\circ \tau_{-i/2}= \tau_{-i/2}\circ R$
where $R$ is a $*$-antiautomorphism on  $L_\infty(\bG)$ and $(\tau_t)_{t\in \bR}$ is a one-parameter group  of $*$-automorphisms on $\Pol(\bG)$ (called the \emph{scaling group}). There exists a distinguished sequence of strictly positive matrices  $Q_\pi\in B(H_\pi)$ with $\pi\in \Irr(\bG)$ implementing the scaling group $(\tau_t)_{t\in \bR}$ and the modular automorphism group $(\sigma_t)_{t\in \bR}$ with respect to $h$ on  $L_\infty(\bG)$, and indeed for all $z\in\mathbb C$ we have (see for instance the computations in \cite[Section 2.1.2 and Section 3]{wang17sidon}),
\begin{equation}\label{eq:matrix q and mod gp}
	(\tau_z\otimes \mathrm{id})(u^{(\pi)})=Q_\pi ^{\mathrm{i}z}u^{(\pi)}Q_\pi ^{-\mathrm{i}z},\quad
	(\sigma_z\otimes \mathrm{id})(u^{(\pi)})=Q_\pi ^{\mathrm{i}z}u^{(\pi)}Q_\pi ^{\mathrm{i}z}. 
\end{equation} 
We say that $\bG$ is of \emph{Kac type} if $Q_\pi=\id_{H_\pi}$ for all $\pi\in \Irr(\bG)$. 
In other  words, $\bG$ is of Kac type if and only if the Haar state $h$ is tracial.

Denote by $\ell_{\infty}(\widehat{\bG})=\bigoplus_{\pi \in \Irr(\bG)} B(H_\pi)$ the direct sum of von Neumann algebras $B(H_\pi)$ and $c_c(\widehat{\bG})$ be the finite direct sum in $\bigoplus_{\pi \in \Irr(\bG)} B(H_\pi)$, i.e. $m\in c_c(\widehat{\bG}) $ if there are only finite many $\pi$ such that $m(\pi)\neq 0$. The notation $\widehat \bG$ used above in fact corresponds to the dual discrete quantum group of $\bG$ (see e.g.  \cite{van98alge}). We will not involve the detailed quantum group structure of $\widehat{\bG}$. 
For a linear functional $\varphi$ on $\Pol(\bG)$, we define the \emph{Fourier transform} by 
$$\cF(\varphi)(\pi)=(\varphi\otimes \id)(( u^{(\pi)})^*), \qquad \pi \in \Irr(\bG).$$
This induces the \textit{Fourier transform}  $\cF:   \Pol(\bG)  \to c_c(\widehat{\bG}) $, given by $$\cF(x)(\pi)=(h(\cdot x)\otimes \id)\left( (\uu{}{\pi})^*\right)\in B(H_\pi) , \qquad \pi \in \Irr(\bG) .$$ 
Note that $\cF:   \Pol(\bG)  \to c_c(\widehat{\bG}) $ is bijective. Obviously there exists a Hilbert space completion of $c_c(\widehat{\bG}) $, denoted by $\ell_2 (\widehat{\bG}) $, such that $\cF$ extends to an isometric isomorphism $\cF:   H_h \to \ell_2(\widehat{\bG}) $, where $H_h$ denotes the GNS completion of $L_\infty (\bG) $ with respect to $h$ as in Section \ref{sect:nontracial matrix} (see e.g. \cite{podlesworonowicz90quantum} and \cite[Proposition 3.2]{wang17sidon}), more precisely,
 \begin{equation}\label{eq:fourier isom}
 h(x^* x) = \sum_{\pi\in \Irr(\bG)} \mathrm{Tr} (Q_\pi)\mathrm{Tr} (Q_\pi (\cF (x) (\pi))^* \cF (x) (\pi)),\quad x\in H_h,
 \end{equation}	
where $\mathrm{Tr}$ denotes the usual (unnormalized) trace on matrices.
Then  $\cF$ is consistent with our framework in Section \ref{sect:nontracial matrix}.
For a symbol $m=(m(\pi))_{\pi}\in \ell_\infty(\widehat \bG)$, we can define a multiplier $T_m$ by \eqref{eq:nontracialoperatormultiplier}, i.e.
$$T_m(x)=\cF^{-1}(m\cdot \cF(x)),\quad x\in \Pol(\bG),$$
which extends to a bounded map on $H_h$.
This coincides with the multipliers considered in \cite{jungeneufangruan09replcqg,daws12cpmultiplier,wang17sidon}. By \eqref{eq:matrix q and mod gp} and the definition of $T_m$, we have (see e.g. \cite[Lemma 3.6]{wang17sidon}) 
\begin{equation}\label{eq:sigma m}
	\sigma_{ r}\circ T_{m} \circ \sigma_{-r} = T_{Q^{ir}mQ^{-ir}} , \quad r\in \mathbb R,\quad \text{where }Q=\oplus_\pi Q_\pi. 
\end{equation}

\begin{remark}\label{example: VN(G) as example of bG}	
	Let $\Gamma$ be a discrete group. 
	We may define a  comultiplication $\Delta$ on the group von Neumann algebra $\vN$ by
	$$\Delta(\lambda(g))=\lambda(g)\otimes\lambda(g), \qquad g\in \Gamma.$$
	The triple $(\vN, \Delta, \tau)$ carries a compact quantum group structure $\bG$ satisfying the aforementioned properties, where we take $L_\infty(\bG) =\vN $ and $h=\tau$. In this case we usually denote $\bG=\widehat{\Gamma}$. We remark that in our language of quantum groups, $\Gamma$ coincides with the dual discrete quantum group   $\widehat{\bG}$.
	The set of unitary equivalence classes of irreducible representations $\Irr(\bG)$ can be indexed by $\Gamma$, so that for every $g\in \Gamma$ the associated representation is of dimension $1$ and is given by  $u^{(g)}=\lambda(g)\in \vN$. Therefore, the set $I$ defined above becomes $\Gamma$, and the Fourier transform sends $\lambda(g)$ to $\delta_g$. 
	Hence, for any $m\in \ell_\infty(\Gamma)$, the associated multiplier is 
	$$T_m: \lambda(g) \mapsto   m(g) \lambda(g), \quad g\in  \Gamma .$$
	This notion coincides with the usual one on groups mentioned in the introduction.
\end{remark}

A straightforward computation shows the following proposition  (see for instance \cite[Section 3]{wang17sidon}).
\begin{prop}\label{prop: prop of Fourier transfrorm} Let $S$ be the antipode and $\Phi$  be a functional  on $\Pol(\bG)$.
	
	\emph{(1)} $\cF({\Phi\circ S^{-1}})(\pi)=(\Phi\otimes \id )(u^{(\pi)}).$
	
	\emph{(2)} $\cF({\Phi^*\circ S^{-1}})(\pi)=\left(\cF(\Phi)(\pi)\right)^*$ where $\Phi^*(x)=\overline{\Phi(x^*)}$ for any $x\in \Pol(\bG)$.
	
	\emph{(3)} $\cF((\id\otimes \Phi)\Delta(x))= \cF(\Phi\circ S^{-1}) \cdot \cF(x)$.
\end{prop}

\begin{remark}\label{remark:multiplier and state}
	Let $\Phi$ be a functional on $\Pol(\bG)$ such that $\cF(\Phi\circ S^{-1}) =m $.
	In other words, by Proposition~\ref{prop: prop of Fourier transfrorm}~(1),  $$m(\pi)=[\Phi(\uu{ij}{\pi})]_{ij}.$$  By Proposition~\ref{prop: prop of Fourier transfrorm}~(3), we have 
	\begin{equation}
		T_{m}(x)=\cF^{-1}\left(\cF( \Phi\circ S^{-1}) \cdot \cF(x)\right)=(\id\otimes \Phi)\Delta(x).
	\end{equation}
	Therefore, we have the following properties.
	
	(1) $T_m$ is a unital completely positive map on $\Pol(\bG)$ if and only if $\Phi$ is a state on $\Pol(\bG)$. 
	
	(2) $T_m$ is selfadjoint if and only if $\Phi^*=\Phi$. Moreover $m(\pi)^*=m(\pi)$ if and only if $\Phi^*\circ S=\Phi$.
	
	(3) We have
	\begin{align*}
		\|m(\pi)\|_{B(H_\pi)}&= \|\uu{ }{\pi}m(\pi)\|_{L_\infty(\bG)\otimes B(H_\pi)}
		= \left\|\left[ \sum_k  \uu{ik}{\pi} \Phi(\uu{kj}{\pi}) \right]_{ij}\right\|_{L_\infty(\bG)\otimes B(H_\pi)}\\&=\left\|\left[T_m(\uu{ij}{\pi})\right]_{ij}\right\|_{L_\infty(\bG)\otimes B(H_\pi)}\leq \|T_m\|_{cb}.
	\end{align*}
\end{remark}


\section[Maximal inequalities and convergence of Fourier series]{Maximal inequalities and pointwise convergence of Fourier series}
In view of our study in Section \ref{sect:nontracial matrix}, to study Fourier multipliers on compact quantum groups, it is essential to consider the case where $T_m$ commutes with the modular automorphism group $\sigma$.
\begin{proposition}\label{prop:exist multipliers commute with Q}
	Let $\Phi$ be a functional on $\Pol(\bG)$ with $\Phi^*=\Phi$ and write ${\varphi(\pi)=[\Phi(u_{ij}^{(\pi)})]_{i,j}}$ for $\pi\in\mathrm{Irr}(\bG)$.  Then the element  
	$$ m(\pi) =\lim_{a\to \infty}\frac{1}{2a}\int_{-a}^{a} Q_\pi^{ir}\tilde{\varphi}(\pi)Q_\pi^{-ir} dr,\quad
	\text{where } \tilde{\varphi}(\pi) = [\frac{1}{2}(\Phi+\Phi\circ R )(u_{ij}^{(\pi)})]_{i,j}$$ is well-defined in $B(H_\pi )$, and  satisfies:
	
	\emph{(i)} $m(\pi)$ is a selfadjoint matrix for any $\pi\in \Irr(\bG)$; 
	
	\emph{(ii)}  for any  $t\in \bR$, $\sigma_{t}\circ T_{m}=T_{m}\circ \sigma_{t}$; 
	
	\emph{(iii)} for any $\pi\in\Irr(\bG)$, we have 
	$$\|m(\pi)\|_{B(H_\pi)}\leq \|\varphi(\pi)\|_{B(H_\pi)},$$
	$$ \|\id_{H_\pi}-m(\pi)\|_{B(H_\pi)}\leq \|\id_{H_\pi}-\varphi(\pi)\|_{B(H_\pi)} ; $$
	
	\emph{(iv)} If $T_{\varphi}$ is unital completely positive on $\Pol(\bG)$,  so is $T_m$.
\end{proposition}
\begin{proof}
	The construction is implicitly given in the proof of \cite[Lemma 5.2]{casperaskalski15Haagerup} and \cite[Proposition 7.17]{dawsfimaskalskiwhit16haagerup}.

	The element $m$ is well-defined by the ergodic theorem since $B(H_\pi)$ is finite-dimensional.
	Note that there is a $*$-antiautomorphism $\widehat R$ of $\ell_\infty (\widehat{\mathbb G})$ with $\widehat{R}^2=\id$ such that $(R\otimes \widehat{R})\mathrm{U}=\mathrm{U}$  where 
	$\mathrm{U}=\oplus_{\pi\in \Irr(\bG)} u^{(\pi)}$ is regarded as an element in $L_\infty(\bG)\overline \otimes \ell_\infty(\widehat \bG)$ (see \cite[Proposition 7.2]{kustermans01LCQGuniversal}), and therefore for any functional $\Upsilon$ on $L_\infty(\bG)$, the following inequality holds
	\begin{align*}
		\|(\Upsilon \circ R\otimes \id)(\mathrm{U})-\mathbf{1}_{\ell_{\infty}(\widehat \bG)}\|_{\ell_{\infty}(\widehat \bG)} &\leq 
		\|(\Upsilon \otimes \widehat{R})(\mathrm{U})-\widehat R (\mathbf{1}_{\ell_{\infty}(\widehat \bG)})\|_{\ell_{\infty}(\widehat \bG)} \\
		&\leq \|(\Upsilon\otimes \id)(\mathrm{U})-\mathbf{1}_{\ell_{\infty}(\widehat \bG)}\|_{\ell_{\infty}(\widehat \bG)},  
	\end{align*} 
	and 
	$$\|(\Upsilon \circ R\otimes \id)(\mathrm{U})\|_{\ell_{\infty}(\widehat \bG)} \leq \|(\Upsilon\otimes \id)(\mathrm{U})\|_{\ell_{\infty}(\widehat \bG)}. $$
	In particular, for  $\pi \in \Irr(\bG)$, taking $\Upsilon(\uu{ij}{\alpha})=\delta_{\alpha\pi}\Phi(\uu{ij}{\pi})$ for all $\alpha\in \Irr(\bG)$, we get
	$$\|\widetilde{\varphi}(\pi)\|_{B(H_\pi)}\leq \|\varphi(\pi)\|_{B(H_\pi)},$$
	and taking $\Upsilon(\uu{ij}{\alpha})=\begin{cases}
	\Phi(\uu{ij}{\pi}),  &\text{ if } \alpha=\pi\\
	\delta_{ij},  &\text{ if } \alpha\neq\pi
	\end{cases}$,
	we get
	$$ \|\id_{H_\pi}-\widetilde \varphi(\pi)\|_{B(H_\pi)}\leq \|\id_{H_\pi}-\varphi(\pi)\|_{B(H_\pi)}, \quad \pi\in\Irr(\bG).$$
	Then (iii) follows from the definition of $m$. Also, by Remark \ref{remark:multiplier and state} we see that $T_{\tilde{\varphi}}$ is unital completely positive if   $T_\varphi$ is. So in the following part we will assume without loss of generality that $\varphi=\tilde \varphi$ and $\Phi = \Phi \circ R$. By \eqref{eq:sigma m} we have 
	
	\begin{equation}\label{eq:def of Tm' commut with Q}
		T_m =\lim_{a\to \infty}\frac{1}{2a}\int_{-a}^{a} \sigma_{ r}\circ T_{\varphi} \circ \sigma_{-r} dr ,
	\end{equation}	
	so we established the assertions (ii) and (iv). We consider 
	$$ \Psi (x)=\lim_{a\to\infty}\frac{1}{2a}\int_{-a}^{a }\Phi(\tau_r(x))dr ,\quad x\in \Pol(\bG).$$
	Then by \eqref{eq:matrix q and mod gp},
	\begin{align*}
		m(\pi) & = \lim_{a\to \infty}\frac{1}{2a}\int_{-a}^{a} Q_\pi^{ir}\varphi(\pi)Q_\pi^{-ir} dr
		=\lim_{a\to \infty}\frac{1}{2a}\int_{-a}^{a} Q_\pi^{ir}[(\Phi  \otimes\mathrm{id}) (u^{(\pi)})]Q_\pi^{-ir} dr\\
		&=\lim_{a\to \infty}\frac{1}{2a}\int_{-a}^{a} (\Phi\circ \tau_r \otimes\mathrm{id}) (u^{(\pi)}) dr = (\Psi \otimes \mathrm{id})(u^{(\pi)}).
	\end{align*} 
	Note that $\Psi$ is invariant under $\tau$ according to the ergodic theorem. Recall that $\Phi\circ R = \Phi$. So we have $\Psi\circ S = \Psi$.  By Proposition~\ref{prop: prop of Fourier transfrorm} and $\Phi=\Phi^*$, we get 
	$$m(\pi)=\cF(\Psi\circ S^{-1})(\pi)=\left(\cF(\Psi)(\pi)\right)^*=\left(\cF(\Psi\circ S^{-1})(\pi)\right)^*=m(\pi)^*.$$
	So we obtain (i).	
\end{proof}

We recall the following approximation properties of quantum groups introduced by De Commer, Freslon and Yamashita in \cite{decommerfreslon14CCAP}. For simplicity of exposition, \emph{we always assume in this paper that $\Irr(\bG)$ is countable}. The general cases can be dealt with by considering the collection of all finitely generated quantum subgroups of 
$\widehat{\bG}$.
\begin{definition}\label{def:ACPAP}
	Let $\bG$ be a compact quantum group. $\widehat{\bG}$ is said to have the \emph{almost completely positive approximation property} (ACPAP for short) if there are two sequences  $(\varphi_s)_{s\in \bN}\subset \ell_{\infty}(\widehat \bG)$ and $(\psi_k)_{k\in \bN}\subset c_c(\widehat \bG)$  such that
	
	{(1)} for any $s, k\in \bN$, $T_{\varphi_s}$ is a unital completely positive map on {$L_\infty(\bG)$}; 
	
	{(2)} for any $\pi\in \Irr(\bG)$, we have
	$$\lim _{s\to \infty}\|\id_{H_\pi}-\varphi_s(\pi)\|_{B(H_\pi)}=0\ \text{ and }\lim _{k\to \infty}\|\id_{H_\pi}-\psi_k(\pi)\|_{B(H_\pi)}=0 ; $$
	
	{(3)} for any $s\in \bN$ and $\varepsilon>0$, there is a $k=k(s, \varepsilon)$  such that $\|T_{\psi_k}-T_{\varphi_s}\|_{cb}\leq \varepsilon$.
	
	Moreover,  if for any $k\in \bN$ we may directly choose $T_{\psi_k}$ to be unital completely positive, then $\widehat{\bG}$ is said to be \emph{amenable} (or to have the \emph{completely positive approximation property}).
\end{definition}
For notational convenience and without loss of generality, in the sequel we will always set $\varphi_0(\pi)=\psi_0(\pi)=\delta_{\mathbf{1}}(\pi)$. Note that $\varphi_s(\mathbf{1})=1$ for all $s$, therefore we will always assume that $\psi_k(\mathbf{1})=1$ for all $k$. 
\begin{remark}\label{def: ACPAP+}
	Assume that  $\widehat{\bG}$ has the ACPAP. Then we may indeed find two sequences  $(\varphi_s)_{s\in \bN}\subset \ell_{\infty}(\widehat \bG)$, $(\psi_k)_{k\in \bN}\subset c_c(\widehat \bG)$  satisfying {(1)}-{(3)} in Definition~\ref{def:ACPAP} such that the following assertions also hold:
	
	{(4)} $T_{\psi_k}$ is selfadjoint for all $ k\in \bN$;
	
	{(5)} $\varphi_s$ and $\psi_k$ are selfadjoint matrices for all $s, k\in \bN$;
	
	{(6)} for any $t\in \bR$, $s\in \bN$, $k\in \bN$, $\sigma_t\circ T_{\phi_s}=T_{\phi_s} \circ \sigma_t$ and $\sigma_t\circ T_{\psi_k}=T_{\psi_k} \circ \sigma_t$.	
	
	Indeed,	for any  sequences  $(\varphi_s)_{s\in \bN}\subset \ell_{\infty}(\widehat \bG)$, $(\psi_k)_{k\in \bN}\subset c_c(\widehat \bG)$  satisfying {(1)}-{(3)} in Definition~\ref{def:ACPAP}, by Remark \ref{remark:multiplier and state} we see that the map $ x\mapsto (T_{\psi_k}(x^*))^*$ is a multiplier associated with the matrices $\widetilde{\psi}_k(\pi)=[{\Psi_k^*(u_{ij}^{(\pi)})}]_{ij}$, where $\Psi_k$ is the functional on $\Pol(\bG)$ so that $\Psi_k(u_{ij}^{(\pi)})=\psi_k(\pi)_{ij}$. Then $T_{(\psi_{k}+\widetilde{\psi}_k)/2}$ is selfadjoint. Note that $\lim _{k\to \infty}\|\id_{H_\pi}-\psi_k(\pi)\|_\infty= 0$ is equivalent to  $\lim_{k\to \infty} \Psi_k(\uu{ij}{\pi})\to \delta_{ij}$ pointwise. Thus $ (\psi_{k}+\widetilde{\psi}_k)/2$  satisfies  (2). Since the operators $T_{\varphi_s}$ are positive, ${T_{\varphi_s}(x)=\left( T_{\varphi_s}(x^*)\right)^* }$. Thus $( (\psi_{k}+\widetilde{\psi}_k)/2)_k$  satisfies (3).
	Therefore we may always assume that $T_{\psi_k}$ is selfadjoint, which means that $\Psi_k=\Psi_k^*$ by Remark~\ref{remark:multiplier and state}~(2). 		
	We  construct two new sequences $(\varphi_s^\prime)_{s\in \bN}\subset \ell_{\infty}(\widehat \bG) , (\psi_k^\prime)_{k\in \bN}\subset c_c(\widehat \bG)$ by the formulas given in Proposition~\ref{prop:exist multipliers commute with Q}, then they immediately satisfy (3) by \eqref{eq:def of Tm' commut with Q}. It is easy to see that $(\varphi_s^\prime)_{s\in \bN}$ and $(\psi_k^\prime)_{k\in \bN}$ satisfy (1) (2) (4)  (5) (6).
\end{remark}
\begin{remark}
	If $\mathbb G$ is of Kac type, the multipliers $T_{\varphi_s}$ and $T_{\psi_k}$ can be taken central by a simple averaging argument, that is,  $ \varphi_s(\pi) $ and $\psi_k(\pi)$ belong to $\bC\id_{H_\pi}$ for all $\pi\in \Pol(\bG)$. We refer to  \cite{krausruan99approximation,brannan17approximation} for details.
\end{remark}

\begin{lemma}\label{lemma:def of t j subsequnece for amenable group}
	For any $s\in \bN_+$,	$\varphi_s=(\varphi_s(\rho))_{\rho\in I}$ with $\varphi_s(\rho)\in B(H_\rho)$, where $(H_\rho)_\rho$ are Hilbert spaces and $I$ is an infinite countable set.  Let $(E_s)_{s\geq 1}$ be an  increasing sequence of finite  subsets of $I$ with $\cup_{s\geq 1} E_s=I$. If 
	$$\lim _{s\to \infty}\|\varphi_s(\rho)-\id_{H_\rho}\|_{B(H_\rho)}=0$$
	for any $\rho \in I$, then  there is a subsequence $(s_{_{N}})_{N\geq 1}$ of $\bN$  such that 
	\begin{equation}\label{eq:condition of s N subsequence}
		\|\id_{H_\rho}-\varphi_{s_{_{N+1}}}(\rho)\|_{B(H_\rho)}\leq 2^{-N}, \qquad   \rho\in E_{s_{_N}}.
	\end{equation}
\end{lemma}	
\begin{proof}
	We will construct a  sequence $(s_{_{N}})_{N\in \bN}$ by induction.
	First we let ${s_1=1}$.  
	Assume that $(s_{{j}})_{j=0}^N$ has already been defined. For any $\rho\in E_{s_N}$,  we can find an $s_{_{N+1}}(\rho)>s_{_{N}}$ large enough, such that for any $s\geq s_{_{N+1}}(\rho)$,
	$$\|\id_{H_\rho}-\varphi_{s}(\rho)\|_{B(H_\rho)}\leq 2^{-N}.$$
	Since $E_{s_{_{N}}}$ is a finite set, we can set $s_{_{N+1}}=\max\{s_{_{N+1}}(\rho): \rho\in   E_{s_{_{N}}}\}.$ 
	Therefore  the proof is complete. 
\end{proof}
Then we may construct the semigroups and multipliers satisfying the assumptions of Theorem \ref{theorem:criterion2+: Haagerup case}.
\begin{theorem}\label{lemma: def of conditionally negative definite function on G}
	Assume that $\widehat{\bG}$ has the  ACPAP and let  $(\varphi_s)_{s\in \bN}$ and $(\psi_k)_{k\in \bN}$ be the corresponding sequences satisfying \emph{(1)-(6)} in Definition~\ref{def:ACPAP} and Remark~\ref{def: ACPAP+}. 
	Let  $(k_s)_{s\in \bN_+}$  be an increasing   subsequence of $\bN$ such that
	$\|T_{\varphi_s}-T_{\psi_{k_s}}\|_{cb}\leq \frac{1}{2^{s+1}}.$   Let $k_0=s_0=0$ and  $(s_{_N})_{N\in \bN_+}$ be a subsequence of $\bN$  such that \eqref{eq:condition of s N subsequence} holds with $E_{s}=\cup_{i=0}^{s}\supp \psi_{k_i}$. Define
	$$\ell(\pi)=\sum_{j\geq 0} \sqrt{2}^j \left( \id_{H_\pi}-\varphi_{s_{j}} (\pi)\right) .$$
	Then the following assertions hold.
	
	\emph{(1)} $\ell(\mathbf{1})=0$ and for any $\pi\neq \mathbf1$,		
	$$\|\ell(\pi)\|_{B(H_\pi)}\asymp \sqrt{2}^{J(\pi)}$$ where  $J(\pi):=\min \{j\in \bN: \pi\in E_{s_j} \}$.
	
	\emph{(2)} $S_t:x\mapsto \cF^{-1}\left(e^{-t\ell(\pi)}\cF(x)\right)$ is a semigroup of unital completely positive $h$-preserving   maps on $L_\infty(\bG)$ and for any $t\geq 0$, $r\in \bR$, 
	$$S_t\circ\sigma_r=\sigma_r\circ S_t \add{and}   h(S_t(x)^*y)=h(x^*S_t(y)), \qquad x, y\in L_\infty(\bG).$$

	\emph{(3)} Denote $m_N=\psi_{k_{s_N}}$. Then $(m_N)_{N\in \bN}$ satisfies that for any $\pi \in \Irr(\bG)$,
	\begin{equation*} 
		\|\id_{H_\pi}-m_N (\pi)\|_{B(H_\pi)}\lesssim \frac{\|\ell(\pi)\|_{B(H_\pi)}^2}{2^N},\quad \|m_N (\pi)\|_{B(H_\pi)}\lesssim  \frac{2^N}{\|\ell(\pi)\|_{B(H_\pi)}^2}. 
	\end{equation*}
	In particular, for any $2\leq p<\infty$,
	$$\|{\sup_{N\in \bN}}^+ T_{m_N} x\|_p\lesssim_p \|x\|_p, \qquad  x\in L_p(\bG).$$
	For all $x\in L_p(\bG)$ with $2< p<\infty$ (resp. $p=2$), $T_{m_N}(x)$ converges a.s. (resp. b.a.s.) to  $x$ as $N\to\infty$.
	
	Moreover, if $\widehat{\bG}$ is  amenable, then the  above results hold  for all $1<p<\infty$ (with the b.a.s. convergence for $p\leq 2$). If $\bG$ is of Kac type, all the convergences above are a.u. 
\end{theorem}
\begin{proof}
	By Remark~\ref{def: ACPAP+} and \eqref{eq:sigma m}, we have $\ell(\pi)Q_\pi=Q_\pi\ell(\pi)$ for any $\pi \in \Irr(\bG)$ and $S_t\circ \sigma_r=\sigma_r\circ S_t$ for any $t\in \bR_+, r\in \bR$.
	
	Recall that for any $N$,  $T_{\varphi_{s_N} }$ is unital and in particular $\varphi_{s_N} (\textbf{1})=1$.
	As a consequence we get $\ell(\textbf{1})=0$. In the following we consider $\pi \neq \mathbf 1$ and estimate the quantity $\|\ell (\pi)\|_{B(H_\pi)}$. Recall that $k_0=s_0=0$ and $\varphi_0(\pi)=\psi_0(\pi)=\delta_{\mathbf{1}}(\pi)$, so $E_0=\{\mathbf{1}\}$, which implies $J(\pi)\geq 1$ if $\pi \neq \mathbf 1$. %
	By the definition of $J$,  we have $\pi\in E_{s_{J(\pi)}}\subset E_{s_{j-1}}$ if   $j\geq J(\pi)+1$. Recall that $T_{\varphi_s}$ is unital completely positive and hence $\|\varphi_s (\pi)\|_{B(H_\pi)}\leq 1$ by Remark \ref{remark:multiplier and state}.
	Therefore, by \eqref{eq:condition of s N subsequence}, we have 
	\begin{align*}
		\|\ell(\pi)\|_{B(H_\pi)}&\leq \sum_{j=0}^{J(\pi)} \sqrt{2}^j  \| \id_{H_\pi}-\varphi_{s_{j}} (\pi)\|_{B(H_\pi)}+\sum_{j\geq J(\pi)+1 } \sqrt{2}^j \| \id_{H_\pi}-\varphi_{s_{j}} (\pi)\|_{B(H_\pi)}\\
		&\leq\sum_{j=0}^{J(\pi)} 2\sqrt{2}^j+\sum_{j\geq J(\pi)+1 } \sqrt{2}^{j}\cdot 2^{-j}\lesssim \sqrt{2}^{J(\pi)}.
	\end{align*}
	For  $j\leq J(\pi)-1$, we have $\pi \notin E_{s_j}$  and hence $\psi_{k_{s_j}}(\pi)=0$. We also have $\id_{H_\pi}-\varphi_{s_j} (\pi)\geq 0$ for all $j$ since $\varphi_s  (\pi)$ is selfadjoint  and $\|\varphi_s (\pi)\|_{B(H_\pi)}\leq 1$ as mentioned previously. Also by Remark~\ref{remark:multiplier and state}, we have
	$$\|\varphi_s(\pi)-\psi_{k_s}(\pi)\|_
	{B(H_\pi)}\leq \|T_{\varphi_s}-T_{\psi_{k_s}}\|_{cb}\leq \frac{1}{2^{s+1}}.$$  So
	\begin{align*}
		\|\ell(\pi)\|_{B(H_\pi)}&= \|\sum_{j=0}^{J(\pi)-1} \sqrt{2}^j \left(  \id_{H_\pi}  -\varphi_{s_j}  (\pi)\right) +\sum_{j\geq J(\pi)} \sqrt{2}^j (\id_{H_\pi}-\varphi_{s_j}  (\pi))\|_{B(H_\pi)}\\
		&\geq \|\sum_{j=0}^{J(\pi)-1} \sqrt{2}^j \left(  \id_{H_\pi}  -\varphi_{s_j}  (\pi)\right)\|_{B(H_\pi)}\\
		&
		=\|\sum_{j=0}^{J(\pi)-1} \sqrt{2}^j \left(  \id_{H_\pi}  -\varphi_{s_j}  (\pi) +\psi_{k_{s_j}}  (\pi) \right)  \|_{B(H_\pi)}\\
		&\geq	\sum_{j=0}^{J(\pi)-1} \sqrt{2}^j \left(  \id_{H_\pi}-\frac{1}{2^{s_j+1}}\right) \gtrsim \sqrt{2}^{J(\pi)}.
	\end{align*}
	Hence $\ell$ is well defined and  (1) is verified.

	Recall that $\ell(\textbf{1})=0$. Therefore $S_t(1)=e^{-t\ell(\textbf{1})}1=1$.
	On the other hand, recall that $\varphi_s (\pi)$ is selfadjoint, which means that $\ell(\pi)$ is also selfadjoint. 
	Thus by \eqref{eq:fourier isom},
	\begin{align}
		\label{eq:show symmetric of multiplier}	h(S_t(x)^*y)	&=\sum_{\pi\in \Irr(\bG)} \Tr(Q_\pi)  \Tr\left(Q_\pi\left( e^{-t\ell(\pi)}\cF(x)(\pi)\right)^* \cF(y)(\pi)  \right)\\ 	\nonumber&=\sum_{\pi\in \Irr(\bG)}\Tr(Q_\pi)\Tr\left(Q_\pi\left(\cF(x)(\pi)\right)^* e^{-t\ell(\pi)} \cF(y)(\pi)  \right)\\
		\nonumber	& =h(x^*S_t(y)).
	\end{align}
	In particular,
	$$h(S_t(x))=h(S_t(1)x)=h(x). $$
	Now let us verify the complete  positivity of $S_t$. We define the functionals $\epsilon$, $L$ and $\Phi_s  $ on $\Pol(\bG)$ by
	$$\epsilon(\uu{ij}{\pi}) = \delta_{ij},\quad
	L(\uu{ij}{\pi}) = \ell(\pi)_{ij},\quad 
	\Phi_s (\uu{ij}{\pi}) = \varphi_s (\pi)_{ij}.$$
	Note that the functional $\epsilon$ is a $*$-homomorphism on $\Pol(\bG)$, usually called the counit. By Remark~\ref{remark:multiplier and state}, $\Phi_s $ are  states and 
	$$  L=\sum_{j\geq 0} \sqrt{2}^j \left( \epsilon -\Phi_{s_j} \right) $$
	with the convergence understood pointwise on $\Pol(\bG)$. In particular 
	$$L(a^* a)=-\sum_{j\geq 0} \sqrt{2}^j  \Phi_{s_j} (a^*a)\leq 0,\quad a\in\ker \epsilon .$$ This means that $L$ is a generating functional in the sense of \cite{dawsfimaskalskiwhit16haagerup} and there is a  state $\mu_t $ with $ \mu_t (\uu{ij}{\pi}) =(e^{-t\ell(\pi) })_{ij} $ for all $t$ and $\pi$ by    \cite[Lemma 7.14 and  Equality (7.4)]{dawsfimaskalskiwhit16haagerup}. So by Remark \ref{remark:multiplier and state}, $S_t = T_{e^{-t\ell}}$ is completely positive. 
	
	We take $m_N=\psi_{{k_{s_{_N}}}}$. (If $\widehat{	\bG}$ is amenable, we take $m_N=\psi_{{k_{s_{_N}}}}=\varphi_{s_N}$.) 
	If $\pi=\mathbf{1}$, then for any $N\in \bN$, $1-m_N(\mathbf{1})=0$  since $T_{\psi_{k_{s_N}}}$ is unital.
	Note that $s_{N}\geq N$. 
	If $1\leq J(\pi)\leq N-1$, i.e. $\pi\in E_{s_{N-1}}$, then
	\begin{align*}
		\|\id_{H_\pi}-m_{N}(\pi)\|_{B(H_\pi)}&\leq \|\id_{H_\pi}-\varphi_{s_N}(\pi)\|_{B(H_\pi)} + \|\varphi_{s_N}(\pi)-\psi_{{k_{s_{_N}}}}(\pi)\|_{B(H_\pi)}\\
		&\leq (2^{-N}+2^{-s_{_N}-1})\lesssim  2^{J(\pi)-N}\lesssim \frac{\|\ell(\pi)\|_{B(H_\pi)}^2}{2^{N}}.
	\end{align*}
	If $J(\pi)\geq N$, then
	$$\|\id_{H_\pi}-m_{N}(\pi)\|_{B(H_\pi)}\leq 1\leq 2^{J(\pi)-N}\lesssim \frac{\|\ell(\pi)\|_{B(H_\pi)}^2}{2^{N}} .$$
	On the other hand, $$\|m_{N}(\pi)\|_{B(H_\pi)}\leq \dsone_{[0, N]}(J(\pi))\leq 2^{N-J(\pi)}\lesssim \frac{2^N}{\|\ell(\pi)\|_{B(H_\pi)}^2}, \quad \pi\in \Irr(\bG).$$	
	A computation	similar to \eqref{eq:show symmetric of multiplier} yields that the map $T_{m_N}$ is symmetric and $h$-preserving for any $N$.  Applying Theorem~\ref{theorem:criterion2+: tracial noncentral case} and Theorem~\ref{theorem:criterion2+: Haagerup case}, we obtain the desired maximal inequalities and pointwise convergences.
\end{proof}

\begin{remark} \label{rmk:condtions for subsequence s N can be modified}
	The above proof indeed shows that Theorem~\ref{lemma: def of conditionally negative definite function on G} also holds for any subsequence
	$(s_{_{N}})_{N\in \bN_+}$ satisfying
	\begin{equation}\label{eq:condtions for subsequence s N can be modified}
		\|\id_{H_\rho}-\varphi_{s_{_{N+1}}}(\rho)\|_{B(H_\rho)}\leq 2^{J(\rho)-N}, \qquad   \rho\in E_{s_{_N}}
	\end{equation}
	where  $(E_s)_s$ are determined increasing finite sets and $J(\rho)=\min \{j\in \bN: \rho\in E_{s_{{j}}} \}$.
	This more general formulation will be useful in the next section.
\end{remark}

In particular,  we obtain the following pointwise convergence theorem in the  general setting of quantum groups.
\begin{corollary}\label{prop:general case for ACPAP case subsequence}
	\emph{(1)} Assume that $\widehat{\bG}$ has the ACPAP. Then $\widehat{\bG}$  admits a sequence of  completely contractive Fourier multipliers $(T_{m_N})_{N\in \mathbb N}$ on $L_\infty(\bG)$ so that $m_N$ are finitely supported and 
	for any $2\leq p<\infty$,
	$$\|{\sup_{N\in \bN}}^+ T_{m_N} x\|_p\lesssim_p \|x\|_p, \qquad  x\in L_p(\bG),$$
	and	$ T_{m_N} x$ converges to $x$ a.s. as $N\to \infty$ for all $x\in L_p (\bG)$ with $2< p< \infty$ and $ T_{m_N} x$ converges to $x$ b.a.s. (a.u. if $\bG$ is of Kac type)  as $N\to \infty$ for all $x\in L_2 (\bG)$.	
	
	\emph{(2)} Assume that $\widehat{\bG}$ is  amenable.  Then $\widehat{\bG}$ admits a sequence of unital completely positive Fourier multipliers $(T_{m_N})_{N\in \mathbb N}$ on $L_\infty(\bG)$ such that $m_N$ are finitely supported and for any $1<p<\infty$
	$$\|{\sup_{N\in \bN}}^+ T_{m_N} x\|_p\lesssim_p \|x\|_p, \qquad  x\in L_p(\bG),$$
	and $ T_{m_N} x$ converges to $x$ a.s. as $N\to \infty$ for all $x\in L_p (\bG)$ with $2< p< \infty$ and  $ T_{m_N} x$ converges to $x$ b.a.s. (a.u. if $\bG$ is of Kac type and $p=2$) as $N\to \infty$ for all $x\in L_p (\bG)$ with $1<p\leq 2$.
\end{corollary}

Moreover, we have the following a.s. convergence of Fourier series of Dirichlet type on $L_2(\bG)$. Let  $p_\pi$ be the projection from $\Pol(\bG)$ onto $\{\uu{ij}{\pi}: 1\leq i,j\leq d_\pi  \}$. It is easy to see that  $p_\pi$ can  be extended to   an orthogonal projection on $L_2(\bG)$.
\begin{proposition}\label{prop:dirichlet}
	Let $\widehat{\bG}$ be a discrete quantum group with the ACPAP. Then there exists an increasing sequence of finite  subsets $(K_N)_N\subset \Irr(\bG)$ such that the maps
	$x\mapsto \sum_{\pi\in K_N} p_\pi(x)$
	is of strong type $(2, 2)$.
	Moreover, for all $x\in L_2 (\mathbb G)$,
	$$\sum_{\pi\in K_N} p_\pi(x) \to x \text{ b.a.s. (a.u. if $\bG$ is of Kac type) as }N\to\infty.$$ 
\end{proposition}
\begin{proof}
	Let $\ell$, $(k_s)_{s}$, $(s_N)_N$, $(E_s)_s$ be given by Theorem~\ref{lemma: def of conditionally negative definite function on G}.
	We show that $K_N=E_{s_{_N}}$ satisfies this proposition.
	Let $N\in \bN$ and $D_N=T_{\dsone_{K_N}}$ be the multiplier associated with the characteristic function $$\dsone_{K_N}(\pi)=\begin{cases}
	\id_{H_\pi}, &\text{ if } \pi \in K_N;\\
	0,& \text{ otherwise. } 
	\end{cases}$$
Define	$$\phi_N(\pi)=\dsone_{K_N}(\pi)-e^{-\frac{\ell(\pi)^{1/2}}{2^{N/4}}}.$$
For any $\pi \in K_N$,  i.e. $J(\pi)\leq N$, we have 
$$\|\phi_N(\pi)\|_{B(H_\pi)}\lesssim \frac{\|\ell(\pi)\|_{B(H_\pi)}^{1/2}}{2^{N/4}}\lesssim \frac{\|\ell(\pi)\|_{B(H_\pi)}^{1/2}2^{N/4}}{(2^{N/4}+\|\ell(\pi)\|_{B(H_\pi)}^{1/2})^2},$$
where the last inequality follows from the fact that $\frac{\|\ell(\pi)\|_{B(H_\pi)}}{\sqrt2^N}\lesssim 1$.	
Also, for any $\pi \notin K_N$, i.e. $J(\pi)>N$, we have
$$\|\phi_N(\pi)\|_{B(H_\pi)}\lesssim \frac{2^{N/4}}{\|\ell(\pi)\|_{B(H_\pi)}^{1/2}}\lesssim \frac{\|\ell(\pi)\|_{B(H_\pi)}^{1/2}2^{N/4}}{(2^{N/4}+\|\ell(\pi)\|_{B(H_\pi)}^{1/2})^2}.$$
As mentioned previously,  Proposition~\ref{prop:general case for p=2} remains valid for the nontracial setting. Together with Proposition~\ref{claim: sup+ < CRp norm},
we get 
$$\|(T_{\phi_N}(x))_N\|_{L_2(\bG;\ell_{\infty})}\lesssim \|x\|_2,\quad \|(T_{\phi_N}(x))_N\|_{L_2(\bG;\ell_{\infty}^c)}\lesssim \|x\|_2, \quad   x\in L_2(\bG).$$
Recall that by the choice of $(E_s)_s$, for any finite subset $F\subset \mathrm{Irr} (\mathbb G)$ there exists $M\geq 1$ with ${F\subset K_N}$ for all $N\geq M$.  Hence for any $x\in \Pol(\bG)$, there is an index $M$ large enough such that  for any $N\geq M$, we have $D_N(x)=x$, and in particular $D_N(x)\to x$ a.s. Then arguing as in Subsection \ref{sec:proof ot criterion1}, by Proposition \ref{prop:Phix-x in Lp M C0} (2) (or its nontracial analogue for $L_2(\bG;\ell_{\infty})$ and $L_2(\bG;\ell_{\infty}^c)$ mentioned after Theorem \ref{theorem:criterion2+: Haagerup case}) and the density of $\Pol(\bG)$, we obtain the desired pointwise convergence of $D_N(x)$ for any $x\in L_2(\bG)$ as $N\to \infty$.
\end{proof}

Note that we also have the corresponding a.u. convergence on $L_\infty(\bG)$ in all the previous results by Remark \ref{rem:au on m}.

\begin{remark}
Our results also provide a general abstract answer to the classical pointwise convergence problems. Indeed, if we take $\mathbb G = \mathbb Z ^d$, then $L_p(\mathbb G)$ coincides with $L_p (\mathbb T ^d)$ and Corollary~\ref{prop:general case for ACPAP case subsequence} (1) amounts to the following fact:
Let $(\Phi_N)_{N\in \bN} \subset L_1 (\mathbb T ^d)$ be an arbitrary sequence of positive trigonometric polynomials  with $\lim_{N}\|\Phi_N *f - f\|_1 =0$ for all $f\in L_1 (\mathbb T ^d)$. Then there exists a subsequence $(N_k)_{k \in \mathbb{N}}$ such that for all $1<p\leq \infty$ and all $f\in L_p (\mathbb T ^d)$, we have
$$\lim_{k\to \infty}\Phi_{N_k} * f =f\text{ a.e.} $$  
Note that when reduced to this classical setting, our method is still novel, which unavoidably involves the ergodic theory and the bootstrap method in Theorem~\ref{theorem:criterion2} as well as the genuinely abstract Markov semigroups in Theorem~\ref{lemma: def of conditionally negative definite function on G}. It seems unclear how to prove this result simply using the classical methods based on Hardy-Littlewood maximal functions. Indeed, as we will illustrate in a forthcoming paper \cite{hongwangwang21inprogress}, we will see some concrete new examples of approximate identities whose kernels seem intricate to be dominated  by  Hardy-Littlewood maximal functions. 

For the Dirichlet means studied in Proposition~\ref{prop:dirichlet}, the corresponding classical counterpart is much easier. One may easily check from the proof that we may take $K_N$ to be ${ \{j\in\mathbb Z^d:\sqrt{|j_1|^2+\dotsm+|j_d|^2}\leq 2^N \}}$ in Proposition~\ref{prop:dirichlet} and we recover the usual pointwise convergence of lacunary Dirichlet means of Fourier series on $L_2 (\mathbb T ^d)$. 
\end{remark}	


\chapter{More concrete examples}\label{sec:more ex}
In this last chapter, we apply our theorems to various explicit examples of multipliers on noncommutative $L_p$-spaces. In particular, we would like to generalize several typical classical pointwise summation methods to the noncommutative setting, such as Fej\'{e}r means and Bochner-Riesz means. In the last section of the chapter we will also discuss briefly the dimension free bounds of noncommutative Hardy-Littlewood maximal operators. 
\section{Generalized Fej\'{e}r means}
\label{sec:fejer nc}
Our first class of examples generalizes the Fej\'{e}r means. Note that the  usual Fej\'{e}r means on the $d$-dimensional torus correspond to the symbols
$$m_N(\mathbf k)=\prod_{i=1}^d (1-\frac{|k_i|}{2N+1})\dsone_{[-2N, 2N]}(k_i),\quad \mathbf k \coloneqq (k_1, k_2, \cdots, k_d) \in \bZ^d.$$
Setting $K_N=[-N, N]^d\cap \bZ^d$, we may rewrite the above symbol as
$$m_N(\mathbf k) =\frac{|K_N\cap (\mathbf k + K_N)|}{|K_N|}.$$
The quantity on the right hand side commonly appears in the study of the geometry of amenable groups. This leads us to consider a number of similar summation methods of noncommutative Fourier series. In particular, we will present some results for amenable groups, and coamenable compact quantum groups.

\subsection{Case of nilpotent groups  and amenable groups}\label{subsec:F-multiplier on G}
Let  $\Gamma$ be a discrete amenable group.
The amenability yields the existence of a   F\o{}lner sequence of $\Gamma$, that is, we may find a sequence $(K_N)_{N\in \bN}$ of subsets of $\Gamma$ such that
$$ \lim _{N\to \infty} \frac{|K_N\cap gK_N|}{|K_N|}=1, \qquad g\in \Gamma.$$ For convenience we set $K_0=\{e\}$. We define a sequence of multipliers $(m_N)_{N\in \bN}$ by 
\begin{equation}\label{eq:symbol folner}
	m_N(g)=\frac{|K_N\cap gK_N|}{|K_N|}.
\end{equation}

It is easy to see that $m_N$ is finitely supported, indeed $\supp m_N=K_N\cdot K_N^{-1}$.
By the F\o{}lner condition, we have $m_N\to 1$ pointwise.
For any $g\in \Gamma$, we have
$$m_N(g)=\langle\lambda(g) \frac{\dsone_{K_N}}{|K_N|^{1/2}},  \frac{\dsone_{K_N}}{|K_N|^{1/2}}\rangle_{\ell_2(\Gamma)}.$$
As a consequence $m_N$ is  positive definite and therefore $T_{m_N}$ is  unital completely positive on $VN(\Gamma)$ for all $N\in \bN$ (see e.g. \cite[Theorem  2.5.11]{brownozawa08bookC*}). Note that $T_{m_N}$ is also $\tau$-preserving. 
In particular, by Theorem~\ref{lemma: def of conditionally negative definite function on G}, we have the following results.
\begin{cor}
Let $\Gamma$ be a discrete amenable group with a F\o{}lner sequence $(K_N)_N$ and  $(m_N)_{N\in \bN}$ a sequence of multipliers given by \eqref{eq:symbol folner}. Then
there exists a subsequence $(N_j)_{j\in \bN}$, such that for all $1<p<\infty$ and all $x\in L_p(\vN)$, 
$$\|{\sup_{j\in \bN}}^+ T_{m_{N_j}}x\|_p\lesssim_p \|x\|_p \add{and} T_{m_{N_j}}x \to x \text{ b.a.u. (a.u. if $p\geq 2$) as } j\to \infty.$$
\end{cor}

In the following, we would like to  give a refined study in the case of nilpotent groups.
First we consider  a $2$-step (or $1$-step) nilpotent group $\Gamma$  generated by a finite set  $S$. We assume that $e\in S$ and  $S=S^{-1}$. 
Due to  \cite{stoll982step}, we have the following estimates:
\begin{equation}\label{eq:volume of Ball and sphere: nilpotent group}
	\beta^{-1} N^d\leq |S^N|\leq \beta N^d \add{and} \beta^{-1} N^{d-1}\leq |S^N \backslash S^{N-1}| \leq \beta N^{d-1},
\end{equation}
where $d\geq1$ is called the \emph{degree} of $\Gamma$, and $\beta<\infty$ is a positive constant depending only on $\Gamma$ and $S$.

For an element $g\in \Gamma$, denote by $|g|$  the word length   of $g$ with respect to the generator set $S$, i.e. $|g|=\min\{k: g\in S^k\}$.   
Let $(m_N)_N$ be a sequence of symbols given by (\ref{eq:symbol folner}) with $K_N=S^N$. 
By $(\ref{eq:volume of Ball and sphere: nilpotent group})$, we have
\begin{equation}\label{eq:general condition 1-varphi n on G}
	1-m_N(g)=\frac{|gS^N\backslash S^N|}{|S^N|}\leq \frac{|S^{N+|g|}\backslash S^N|}{|S^N|}
	\leq \beta\frac{\sum_{i=N+1}^{N+|g|} i^{d-1}}{N^d}
	\lesssim \beta\frac{|g|}{N}, \qquad g\in \Gamma.
\end{equation}
In particular, this shows that $(S^N)_N$ is a F\o{}lner sequence.
On the other hand, 
\begin{equation}\label{eq:general condition varphi n on G}
	|m_N(g)|\leq \dsone_{[0, 2N]}(|g|)\leq 2 \frac{N}{|g|}, \qquad   \ g\in \Gamma.
\end{equation}
Set $J(g)=\min\{j\in\bN: g\in S^{2^{j+1}}\}$. In other words, $J(g)$ is the unique integer with ${ 2^{J(g)}<|g|\leq 2^{J(g)+1}}$.
We have 
$$|1-m_{2^j}(g)|\leq \beta\frac{|g|}{2^j}\leq 2\beta 2^{J(g)-j}.$$
This shows that the  subsequence $(2^j)_{j\in \bN}$ satisfies the inequality in  Remark~\ref{rmk:condtions for subsequence s N can be modified}.
Define $$\ell(g)=\sum_{j\geq 0} \sqrt{2}^j |1-m_{2^j}(g)|.$$  
By Theorem~\ref{lemma: def of conditionally negative definite
	function on G}, for any $g\neq e$, $$\ell(g)\asymp \sqrt{2}^{J(g)}\asymp \sqrt{|g|},$$
and  $(S_t)_{t\in \bR_+}:\lambda(g)\mapsto e^{-t\ell(g)}\lambda(g)$ is a semigroup of unital completely positive trace preserving and symmetric maps. We remark that there are other natural choices of conditionally negative definite  functions with polynomial growth (see e.g. \cite{ciprianisauvageot16negative}), but our above construction is self-contained and useful for the further purpose.  Moreover, for any $t\in \bR_+$,  $S_t$ satisfies Rota's dilation property according to Lemma~\ref{lem:eg rota}.
Inequalities (\ref{eq:general condition 1-varphi n on G}) and (\ref{eq:general condition varphi n on G}) can be written as
$$|1-m_N(g)|\lesssim \beta\frac{\ell(g)^2}{N} \add{and}  |m_N(g)|\lesssim \beta \frac{N}{\ell(g)^2}.$$
Moreover, 
$$m_N(g)=\frac{\dsone_{S^N}*\dsone_{S^N}(g)}{|S^{N}|}.$$
By (\ref{eq:volume of Ball and sphere: nilpotent group}), we obtain
	\begin{align*}
	&\quad\ |m_{N+1}(g)-m_N(g)|\\
	& =\left|\frac{1}{|S^{N+1}|}  { \sum_{\gamma\in\Gamma}\dsone_{S^{N+1}}(\gamma)\dsone_{S^{N+1}}(g^{-1}\gamma)}-\frac{1}{|S^{N}|}( { \sum_{\gamma\in\Gamma}\dsone_{S^{N}}(\gamma)\dsone_{S^{N}}(g^{-1}\gamma) })\right| \\
	&\leq  \frac{1}{|S^{N+1}|} \left({ \sum_{\gamma\in\Gamma}\dsone_{S^{N}}(\gamma) \dsone_{S^{N+1} \setminus S^N}(g^{-1}\gamma)  }+ \sum_{\gamma\in\Gamma}\dsone_{S^{N+1}\setminus S^N}(\gamma) \dsone_{S^{N+1}}(g^{-1}\gamma) \right)\\
	&\qquad +\left|\left( \frac{1}{|S^{N+1}|} -\frac{1}{|S_N|}\right) \sum_{\gamma\in\Gamma}\dsone_{S^{N}}(\gamma)\dsone_{S^{N}}(g^{-1}\gamma)\right|\\
	&\leq 3\frac{|S^{N+1}\backslash S^N|}{|S^{N+1}|}\lesssim \beta ^2\frac{1}{N+1}.
\end{align*}
Therefore $m_N$ satisfies \eqref{eq:criterion a2}. Applying Theorem~\ref{theorem:criterion2}, we have the following corollary.
\begin{corollary}\label{cor:maximal inequlity for step 2}
	Let $\Gamma$ be  a $2$-step (or $1$-step) nilpotent group generated by a finite  symmetric set  $S$. Define $m_N(g)=\frac{|S^N\cap gS^N|}{|S^N|}$. Then
	
	\emph{(1)} $(T_{m_{2^j}})_{j\in \bN}$ is of strong type $(p, p)$ for all $ 1<p<\infty$. Moreover, for any $x\in L_p(\vN)$ with $1<p<\infty$, $T_{m_{2^j}}(x)$ converges b.a.u. to  $ x$ as $j\to \infty$ and for $2 \leq p < \infty $ the b.a.u. convergence can be strengthened to a.u. convergence. 
	
	\emph{(2)} $(T_{m_N})_{N\in \bN}$ is of strong type $(p, p)$ for all $ \frac{3}{2}<p<\infty.$ Moreover, for any $x\in L_p(\vN)$ with $3/2<p<\infty$, $T_{m_N}(x)$ converges b.a.u.     to  $ x$ as $N\to \infty$, and for $2 \leq p < \infty $ the b.a.u. convergence can be strengthened to a.u. convergence.
\end{corollary}

Let us give some remarks on the case of general groups with polynomial growth. Indeed, it is conjectured in \cite{breuillard14poly} that (\ref{eq:volume of Ball and sphere: nilpotent group}) remains true for general groups with polynomial growth. If the conjecture has a positive answer, then  the above corollary still holds in  this general setting by the same arguments. Moreover a partial result was given in \cite{breuillardledoone13nilpotent} for a general $r$-step  nilpotent group $\Gamma$  generated by  a finite set $S$. It asserts that 
\begin{equation}\label{eq:volume of Sphere: nilpotent group step r}
	\beta^{-1}	N^{d-1}\leq |S^N\backslash S^{N-1}|\leq \beta N^{d-\frac{2}{3r}},
\end{equation}
where $\beta$ is a constant depending only on $\Gamma$ and $S$. 
Therefore, as in the arguments for  the case of $2$-step nilpotent groups, we have 
$$1-m_N(g)\leq \frac{|S^{N+|g|}\backslash S^N|}{|S^N|}\lesssim \beta \frac{|g|}{N^{\frac{2}{3r}}}.$$
Let $k=k(r)$ be the minimum integer with $k\geq\frac{3r}{2}$.
Then $N_j(r)=2^{k+k^2+\cdots k^j}$ satisfies Lemma~\ref{lemma:def of t j subsequnece for amenable group}.
Indeed,  set $J(g)=\min \{j\in \bN: g\in E_{N_j(r)} \}$, that is, the integer $J(g)$ satisfying $2N_{J(g)-1}(r)<|g|\leq 2N_{J(g)}(r)$. 
Then for any  $g\in E_{N_j(r)}$, i.e. $J(g)\leq j$, we have
\begin{align*}
	|1-m_{N_{j+1}(r)}(g)|&\lesssim \beta\frac{|g|}{N_{j+1}(r)^{\frac{2}{3r}}}\lesssim 
	\beta \frac{2^{k+k^2+\cdots k^{J(g)}}}{2^{\frac{2}{3r} (k+k^2+\cdots k^{j+1})}}\lesssim 
	\beta \frac{2^{k+k^2+\cdots k^{J(g)}}}{2^{\frac{1}{k} (k+k^2+\cdots k^{j+1})}}\\& \lesssim \beta 2^{-\left( k^{J(g)+1}+\cdots k^{j+1}\right)}\lesssim \beta 2^{J(g)-j }.
\end{align*}
Therefore, by Theorem~\ref{lemma: def of conditionally negative definite function on G} and Remark~\ref{rmk:condtions for subsequence s N can be modified}, we get a conditionally negative definite function $\ell$ on $\Gamma$ such that 
$$|1-m_{N_{j+1}(r)}(g)|\lesssim \beta \frac{\ell(g)^2}{2^j}\add{and}|m_{N_{j+1}(r)}(g)|\lesssim \beta\frac{2^j}{\ell(g)^2}.$$
By Theorem~\ref{theorem:criterion2} (1), we have
\begin{corollary}
	Let $\Gamma$ be  a $r$-step nilpotent group  generated by a finite symmetric set of elements $S$. Let $m_N$, $k(r)$ and $N_j(r)$ be defined above. Then for all $ \ 1<p<\infty$,
	$(T_{m_{N_j(r)}})_{j\in \bN}$ is of strong type $(p, p)$ and for any $x\in L_p(\vN)$ with $1<p<\infty$, $T_{m_{N_j(r)}}(x)$ converges b.a.u. to  $x$ as $j\to \infty$ and for $2 \leq p < \infty $ the b.a.u. convergence can be strengthened to a.u. convergence.
\end{corollary}

\begin{remark}
As we mentioned in the beginning of this section, the symbols of classical Fej\'er means in Euclidean spaces correspond to averages along F\o{}lner sequences of cubes. 
From the viewpoint of geometric theory of amenable groups, even in this Euclidean setting it is still natural to  consider  other F\o{}lner sets besides cubes, such as balls and rectangles. 	We will carry out this in a forthcoming paper.
\end{remark}

\subsection{Case of amenable discrete quantum groups}\label{subsec:F multiplier on bG}
Let $\bG$ be a   compact quantum group of Kac type.  As before, we assume that  $\Irr(\bG)$ is countable for convenience. We keep the notation introduced in Chapter~\ref{sec:multipliers on CQG}. We may study the F\o{}lner sequences and the corresponding multipliers in this quantum setting as follows.

For any  $\alpha, \beta\in \Irr(\bG)$, we use the notation $\alpha \otimes \beta$ denote the unitary representation $u^{(\alpha)}\ot u^{(\beta)}$,  where $u^{(\alpha)}\ot u^{(\beta)}$ refers to the tensor product of representations $u^{(\alpha)}$ and $ u^{(\beta)}$, which is a representation of the form $(u^{(\alpha)}_{ij}u^{(\beta)}_{kl})_{i,j,k,l}$. Recall that any unitary representation could be decomposed into irreducible representations, which could be expressed as following 
\begin{equation*} 
	u^{(\alpha)}\ot u^{(\beta)}=\oplus_{\gamma \in \Irr(\mathbb{G})}N_{\alpha\beta}^\gamma u^{(\gamma)},
\end{equation*}
 where $N_{\alpha\beta}^\gamma\in \bN$ means the number of  representation $u^{(\gamma)}$ in $u^{(\alpha)}\ot u^{(\beta)}$. Denote by $\overline{\alpha}$  the equivalent  class of the representation $((\uu{ij}{\alpha})^*)_{ij}$. We have the following Frobenius reciprocity law (see \cite{woronowicz87matrix}, \cite[Example 2.3]{kyed08coamenable})
\begin{equation}\label{eq:rule of N alpha beta gamma}
	N_{\alpha\beta}^{\gamma}=N^{\alpha}_{\gamma\overline{\beta}}=N^{\beta}_{\overline{\alpha}\gamma}
\end{equation}
for all $\alpha, \beta, \gamma \in \Irr(\bG)$.
We write $\gamma \in \alpha \otimes \beta$ if $N_{\alpha\beta}^\gamma>0$. 
The \textit{weighted cardinality} of a finite subset $F\subset \Irr(\mathbb{G})$ is defined to be
$$|F|_w=\sum_{\alpha\in F}d_\alpha^2,$$
where we recall that $d_\alpha$ denotes the dimension of the representation $\alpha$. 
On the other hand, for a finite subset $F\subset \Irr(\mathbb{G})$ and a representation $\pi\in \Irr(\mathbb{G}) $, the \emph{boundary of $F$ related to $\pi$} is defined by
\begin{align*}
	\partial_\pi F&=\{\alpha \in F: \exists \beta \in F^c, \beta \in \alpha \otimes \pi\}  \cup \{\beta \in F^c: \exists \alpha \in F, \alpha \in  \beta \otimes \pi\}.
\end{align*}
Kyed \cite{kyed08coamenable} proved that there exists a sequence of finite subsets ${(K_n)_{n\in \bN}\subset \Irr(\bG)}$ such that for any $\pi \in \Irr(\bG)$,
\begin{equation}\label{eq:folner cond q}
	\frac{|\partial_\pi K_n|_w}{|K_n|_w}\to 0\add{as} n\to \infty,
\end{equation} 
as soon as $\mathbb G$ is coamenable. Note that the coamenability of $\mathbb G$ is nothing but a property equivalent to the  amenability of the discrete quantum group $\widehat{\mathbb G}$ (see e.g. \cite{brannan17approximation}). The above sequence $(K_n)_{n\in \bB}$ is called a \emph{F\o{}lner sequence}. We associate a sequence of multipliers
\begin{equation}\label{eq:def:F-multiplier function of compact group}
	\varphi_{n}(\pi)=\frac{\sum_{\alpha , \beta \in K_{n}} N_{\bar{\alpha}\beta}^{\pi}d_\alpha d_\beta}{d_\pi (\sum_{\xi \in K_{n}}d_\xi ^2)},\quad \pi\in \mathrm{Irr} (\mathbb G).
\end{equation}
It is easy to see that if $\widehat{\mathbb G} = \Gamma$ for a discrete group $\Gamma$, then the above function coincides with the symbol introduced in \eqref{eq:symbol folner}.
\begin{lemma}
	\emph{(1)} The maps	$T_{\varphi_n}$ are  unital completely positive on  $L_\infty(\bG)$ for all $n\in \bN$.
	
	\emph{(2)} The functions $\varphi_n $ converge  to $1$ pointwise.
\end{lemma}
\begin{proof}
	(1) It is obvious that $\varphi_n(\textbf{1})=1$ and therefore $T_{\varphi_n}$ is unital.
	
	Denote by $\chi(\pi)=\sum_i u_{ii}^\pi \in \mathrm{Pol}(\bG)$ the \textit{character} of $\pi$. We have for any $\alpha, \beta \in \Irr(\bG)$,
	$$h(\chi(\bar{\beta})\chi(\alpha))=\delta_{\alpha\beta}1,\qquad\chi(\alpha)^*=\chi(\bar{\alpha})\add{and}\chi(\alpha)\chi(\beta)=\sum_{\gamma\in \Irr(\bG)} N^{\gamma}_{\alpha\beta}\chi(\gamma).$$ We write $$\Pol_0(\bG)=\text{span}\{\chi(\pi): \pi \in \Irr(\bG) \} $$ and let $\cA_0$ be the w*-closure of $\Pol_0(\bG)$ in $L_\infty(\bG)$. Let $\bE:L_\infty(\bG)\to \cA_0$ be the canonical conditional expectation  preserving the Haar state $h$. Recall that we have assumed that $\mathbb G$ is of Kac type. It is well-known that the conditional expectation $\mathbb E$ can be given by the following explicit formula (see e.g. the proof of \cite[Lemma 6.3]{wang17sidon})
	$$\bE(\uu{ij}{\pi})=\frac{\delta_{ij}}{d_\pi}\chi(\pi),\quad \pi\in \mathrm{Irr} (\mathbb G).$$ 
	Set 
	$$\chi_n=\frac{1}{|K_n|_{ _w}^{1/2}}\sum_{\alpha \in K_n}d_\alpha\chi(\alpha)\in \Pol(\bG)_0.$$
	Then, we have 
	$$T_{\varphi_n}(x)=(h\otimes \id)\left[ (\chi_n^*\otimes 1)\cdot[(\bE\otimes \id)\circ\Delta(x)]\cdot(\chi_n\otimes 1)\right],\quad x\in  L_\infty(\bG) . $$
	Indeed, by linearity and normality, we only need to check the equality for the case $x=\uu{ij}{\pi}$. In this case we see that
	\begin{align*}
		&(h\otimes \id)\left[ (\chi_n\otimes 1)\cdot[(\bE\otimes \id)\circ\Delta(\uu{ij}{\pi})]\cdot(\chi_n^*\otimes 1)\right]\\
		=&\frac{\sum_{\alpha, \beta \in K_n} d_\alpha d_\beta \cdot h\left( \chi({\alpha})\chi(\pi)\chi(\bar{\beta}) \right)}{d_\pi |K_n|_{ _w}}\uu{ij}{\pi}\\
		=&\frac{\sum_{\alpha, \beta \in K_n} d_\alpha d_\beta N_{{\alpha}\pi}^{{\beta}}}{d_\pi |K_n|_{ _w}}\uu{ij}{\pi}\\
		=&T_{\varphi_n}(\uu{ij}{\pi}).
	\end{align*}
	Since $\bE$ and $\Delta$ are completely positive, so is $T_{\varphi_n}$.
	
	(2) The support of $\varphi_n$ is given by
	\begin{equation}\label{eq:sup of F mul on bG}
		\Lambda_n=\{\pi \in \Irr(\mathbb{G}): \exists \alpha , \beta \in K_n \text{ such that }\pi \in \bar{\alpha} \otimes \beta\}.
	\end{equation}
	We denote $ \Theta_\pi^n=\{\alpha \in K_n: \forall\  \beta \in K_n^c, N_{\alpha\pi} ^\beta=0  \}\subset K_n $. Note that $N_{\alpha\pi} ^\beta=N_{\bar{\alpha}\beta}^\pi$ and that $\alpha\in \Theta_\pi^n$ implies $\sum_{\beta\in K_n} N^\beta_{\alpha \pi}d_\beta=d_\alpha d_\pi$ by the choice of $N^\beta_{\alpha \pi}$.  Then
	$$\varphi_{n}(\pi)\geq \frac{\sum_{\alpha \in\Theta_\pi^n} d_\alpha \left( \sum_{\beta \in K_{n}} N_{\alpha\pi}^{\beta} d_\beta\right) }{d_\pi (\sum_{\xi \in K_{n}}d_\xi ^2)}=
	\frac{ \sum_{\alpha \in\Theta_\pi^n} d_\alpha ^2}{\sum_{\xi \in K_{n}}d_\xi ^2}  = \frac{|\Theta_\pi^n|_w}{|K_{n}|_w}
	.$$
	Therefore,
	$$
	1-\varphi_{n}(\pi)\leq \frac{|\{\alpha\in K_{n}: \exists \beta \in K_{n}^c \text{ such that } \beta \in \alpha \otimes \pi\}|_w}{|K_{n}|_w}
	\leq \frac{|\partial_{\pi} (K_{n})|_w}{|K_{n}|_w}.
	$$
	By the F\o{}lner condition \eqref{eq:folner cond q}, $\varphi_n\to 1$ pointwise.
\end{proof}
Therefore, by Theorem~\ref{lemma: def of conditionally negative definite function on G} we get the following result.
\begin{corollary}\label{cor:F multiplier for CQG}
	Assume that $\mathbb G $ is of Kac type and that $\widehat{\mathbb G}$ is  amenable. Let $(K_n)_{n\in \bN}\subset \mathrm{Irr} (\mathbb G)$ be a    F\o{}lner sequence and $\varphi_n$ be the symbols given by  (\ref{eq:def:F-multiplier function of compact group}). Then  there is a subsequence $(n_j)_{j\in \bN}$ such that $(T_{\varphi_{n_j}})_j$ is of strong type $(p, p)$ for any $1<p< \infty$. Moreover for all $x\in L_p(\bG)$ with $1<p<\infty$, $T_{\varphi_{n_j}}(x)$ converges b.a.u. to  $ x$ as $j\to \infty$, and for $2 \leq p < \infty $ the b.a.u. convergence can be strengthened to a.u. convergence. 
\end{corollary}

	\begin{remark}
	Our method applies to Fourier series of non-abelian compact groups as well and yields general pointwise convergence theorem without using the Lie algebraic structure, which is a new approach compared to previous works.  We will carry out this in a forthcoming paper.
	\end{remark}


\section[Smooth radial multipliers on some hyperbolic groups]{Smooth radial multipliers on some hyperbolic groups and free semicircular systems}
Note that the symbols of classical Fej\'{e}r means are not smooth, and it is natural to consider some smooth variants such as Bochner-Riesz means. In this section we will consider the pointwise convergence of Bochner-Riesz means in the noncommutative setting. Also, the smooth radial completely positive multipliers on free groups are fully classified in \cite{haagerupknudby15LKfreegroups}, so we will also discuss the corresponding pointwise convergence problem in the second part of this section.
\subsection{Bochner-Riesz means on hyperbolic groups and free semicircular systems}\label{subsec:hyperbolic}
In this subsection we briefly discuss a noncommutative analogue of Bochner-Riesz means for the setting of hyperbolic groups. We refer to \cite{ghysdelaharpe90sur,gromov87hyperbolic}  for a complete description of hyperbolic groups. We merely remind that all hyperbolic groups are weakly amenable and the completely bounded radial Fourier multipliers have been characterized in \cite{ozawa08weak,meidelasalle17cbofheatsemigroups}. In particular, we denote by $|\  |$ the usual word length function on a hyperbolic group $\Gamma$, then the Fourier multipliers
$$ B_N^\delta (x) = \sum_{g\in \Gamma : |g|\leq N} \left(1-\frac{|g|^2}{N^2}\right)^\delta \hat x (g) \lambda(g),\quad x\in  \vN $$
define a family of completely bounded maps on $\vN$ with $\sup_N \|B_N^\delta \|_{cb} <\infty$ as soon as $\real (\delta) >1$ (see \cite[Example 3.4]{meidelasalle17cbofheatsemigroups}). Let
$$b_N^\delta(g)=(1-\frac{|g|^2}{N^2})^\delta \dsone_{[0, N]}(|g|) $$ 
be the corresponding symbols of the maps $B_N^\delta$.
It is easy to check that for real parameters $\delta>1$,
\begin{equation}\label{eq:symbolconditionbochnerriesz}
|1-b_N^\delta(g)|\leq c \frac{|g|}{N},\quad  |b_N^\delta(g)|\leq c\frac{N}{|g|},\quad  |b_{N+1}^\delta(g) - b_N^\delta(g)|\leq c\frac{1}{N} ,\quad g\in\Gamma	
\end{equation}
for some constant $c>0$. We are interested in the case where the word length function $|\  |$ is conditionally negative definite. This is the case if $\Gamma$ is a non-abelian free group or a hyperbolic Coxeter group. With the help of Theorem~\ref{theorem:criterion1}, we may deduce the following result, which can be viewed as the noncommutative analogue of Stein's theorem \cite{stein58localization}. 
See also \cite{chenxuyin13harmonic} for results about Bochner-Riesz means on quantum tori.
\begin{theorem}\label{thm:bochnerrieszhyperbolic}
	Let $2\leq p\leq\infty$ and $\Gamma$ and $B_N^\delta$ be as above with real parameter $\delta>1-\frac2p$. Assume additionally that the word length function $|\  |$ is conditionally negative definite.  Then there exists a constant $c_{p,\delta}$ such that 
	$$\|{\sup_{N\in \bN}}^+B_N^\delta( x)\|_p\leq c_{p,\delta}\|x\|_p,\qquad x\in L_p(\vN),$$
	and for any $x\in L_p(\vN)$, $B_N^\delta(x)$ converges a.u. to  $x$ as $N\to \infty$.
\end{theorem}
To show the maximal inequalities, we will follow Stein's strategy based on square function inequalities and interpolation. However, in order to interpolate, we have to make use of Theorem~\ref{theorem:criterion1} to get the maximal inequalities for large $\delta$ ($\delta>1$ in the present setting). This idea, which depends finally on ergodic theory, is quite different from Stein's one.

\begin{proof}
For convenience of presentation, we will allow $N$ to take any positive real numbers: we will prove that for $2\leq p\leq\infty$ and $\delta>1-\frac2p$,
$$\|{\sup_{R>0}}^+B_R^\delta( x)\|_p\leq c_{p,\delta}\|x\|_p,\qquad x\in L_p(\vN).$$

Step1. For $\delta\in\mathbb C$ with $\real (\delta)>1$ and $2\leq p\leq\infty$, we have
\begin{align}\label{p for inter}\|{\sup_{R>0}}^+B_R^\delta( x)\|_p\lesssim\|x\|_p,\qquad x\in L_p(\vN).\end{align}
Choose $\alpha>0$ and $\beta\in\mathbb C$ such that $\real (\delta)>\alpha>1$ and $\delta=\alpha+\beta$. One gets the identity
\begin{align}\label{identity}
B^\delta_R=C_{\alpha,\beta}R^{-2\delta}\int^R_0(R^2-t^2)^{\beta-1}t^{2\alpha+1}B^\alpha_t dt,
\end{align}
where $C_{\alpha,\beta}=\frac{2\Gamma(\alpha+\beta+1)}{\Gamma(\alpha+1)\Gamma(\beta)}$. Let $(R_n)$ be a sequence in $(0,\infty)$ and take an element ${(y_n)\in L_{p'}{(\vN;\ell_1)}}$ with norm not bigger than 1. Then applying Theorem~\ref{theorem:criterion1} and \eqref{eq:symbolconditionbochnerriesz}, we obtain that for any $x\in L_p(\vN)$,
\begin{align*}
\left|\tau\left(\sum_n B^\delta_{R_n}(x)y_n\right)\right|&=|C_{\alpha,\beta}|\left|\sum_nR^{-2\delta}_n\int^{R_n}_0(R^2_n-t^2)^{\beta-1}t^{2\delta+1}\tau(B^\alpha_t(x)y_n)dt\right|\\
&\leq|C_{\alpha,\beta}|\int^{1}_0|(1-t^2)^{\beta-1}t^{2\delta+1}|\left|\tau(B^\alpha_{tR_n}(x)y_n)\right|dt\\
&\leq|C_{\alpha,\beta}|\int^{1}_0|(1-t^2)^{\beta-1}t^{2\delta+1}|dt\|{\sup_{R>0}}^+B_R^\alpha( x)\|_p\lesssim\|x\|_p.
\end{align*}
By the duality in Proposition \ref{prop: prop of LMinfty and LMl1} (1), we then deduce the desired maximal inequality in this step.

Step 2. For $\delta>0$, we have
\begin{align}\label{2 for inter}\|{\sup_{R>0}}^+B_R^\delta( x)\|_2\lesssim\|x\|_2,\qquad x\in L_2(\vN).\end{align}
We first  assume $\delta>1/2$. Choose $\alpha>-1/2$ and $\beta>1$ such that $\delta=\alpha+\beta$. By  \eqref{identity} we get
\begin{align*}
B^{\alpha+\beta}_R&=-C_{\alpha,\beta}R^{-2(\alpha+\beta)}\int^R_0\left(\int^t_0B^\alpha_rdr\right)\frac{d}{dt}[(R^2-t^2)^{\beta-1}t^{2\alpha+1}]dt\\
&=C_{\alpha,\beta}\int^1_0\varphi(t)M^\alpha_{Rt}dt,
\end{align*}  
where 
$$M^\delta_t=\frac{1}{t}\int^t_0B^\alpha_rdr$$
and
$$\varphi(t)=2(\beta-1)(1-t^2)^{\beta-2}t^{2\alpha+3}-(2\alpha+1)(1-t^2)^{\beta-1}t^{2\alpha+1}$$
which is absolutely integrable over $[0,1]$. 
Fix $x\in L_2(\vN)$ and let $(R_n)$ be a sequence in $(0,\infty)$ and $(y_n)\in L_{2}{(\vN;\ell_1)}$ with norm not bigger than 1. 
We get by the previous identity and the duality
\begin{align*}
\left|\tau(\sum_n B^\delta_{R_n}(x)y_n)\right|&=|C_{\alpha,\beta}|\left|\tau(\sum_n\left(\int^1_0\varphi(t)M^\alpha_{R_nt}(x)dt\right)y_n)\right|\\
&\leq |C_{\alpha,\beta}|\int^1_0|\varphi(t)|\left|\tau(\sum_nM^\alpha_{R_nt}(x)y_n)\right|dt\\
&\lesssim  \|{\sup_{R>0}}^+M^\alpha_{R}(x)\|_2.
\end{align*}
By the arbitrariness of $(R_n)$ and $(y_n)$, we obtain
$$\|{\sup_{R>0}}^+B^\delta_{R}(x)\|_2\lesssim \|{\sup_{R>0}}^+M^\alpha_{R}(x)\|_2.$$
Next we have to show
\begin{align}\label{average}
\|{\sup_{R>0}}^+M^\alpha_{R}(x)\|_2\lesssim\|x\|_2
\end{align}
for $\alpha>-1/2$. 
We use again the duality to obtain
\begin{align*}
\left|\tau(\sum_nM^\alpha_{R_n}(x)y_n)\right|&\leq  \left|\tau(\sum_nM^{\alpha+1}_{R_n}(x)y_n)\right|+ \left|\tau(\sum_n(M^{\alpha+1}_{R_n}(x)-M^\alpha_{R_n}(x))y_n)\right|\\
&\lesssim \|{\sup_{R>0}}^+M^{\alpha+1}_{R}(x)\|_2+ \|{\sup_{n}}^+G^\alpha_{R_n}(x)\|_2,
\end{align*}
where $G^\alpha_R(x)=M^{\alpha+1}_{R}(x)-M^\alpha_{R}(x)$.
By Lemma \ref{theorem:interpolation l infty}, we have
\begin{align*}
\|{\sup_{n}}^+G^\alpha_{R_n}(x)\|_2&\leq \|{\sup_{n}}^+|G^\alpha_{R_n}(x)|^2\|^{\frac14}_1 \|{\sup_{n}}^+|(G^\alpha_{R_n}(x))^*|^2\|^{\frac14}_1.
\end{align*}
Note that 
\begin{align*}
|G^\alpha_R(x)|^2&=\left|\frac{1}{R}\int^R_0(B^{\alpha+1}_r(x)-B^{\alpha}_r(x))dr\right|^2\\
&\leq \int^R_0|B^{\alpha+1}_r(x)-B^{\alpha}_r(x)|^2\frac{dr}{R}\\
&\leq \int^\infty_0|B^{\alpha+1}_r(x)-B^{\alpha}_r(x)|^2\frac{dr}{r}=:(G^\alpha(x))^2.
\end{align*}
Similarly, set
$$G^\alpha_*(x)=\left(\int^\infty_0|(B^{\alpha+1}_r(x)-B^{\alpha}_r(x))^*|^2\frac{dr}{r}\right)^{\frac12}.$$
Therefore, we get 
$$\|{\sup_{R>0}}^+M^{\alpha}_{R}(x)\|_2\leq \|{\sup_{R>0}}^+M^{\alpha+1}_{R}(x)\|_2+\|G^\alpha(x)\|_2^{\frac12}\|G_*^\alpha(x)\|_2^{\frac12}.$$
We claim that for $\alpha>-1/2$,
$$\|G^\alpha(x)\|_2=\|G^\alpha_*(x)\|_2\lesssim \|x\|_2.$$
The first identity is trivial; on the other hand, by Plancherel's theorem 
\begin{align*}
\|G^\alpha(x)\|_2^2&=\int^\infty_0\tau(|B^{\alpha+1}_r(x)-B^{\alpha}_r(x)|^2)\frac{dr}{r}\\
&=\int^\infty_0\sum_{|g|\leq R}\left|\left(1-\frac{|g|^2}{R^2}\right)^{\alpha+1}-\left(1-\frac{|g|^2}{R^2}\right)^{\alpha}\right|^2|\hat{x}(g)|^2\frac{dr}{r}\\
&=\sum_{g\neq e}|\hat{x}(g)|^2\int^\infty_{|g|}\frac{|g|^4}{r^4}\left(1-\frac{|g|^2}{r^2}\right)^\alpha\frac{dr}{r}\lesssim \|x\|_2
\end{align*}
since when $\alpha>-1/2$,
$$\int^\infty_{|g|}\frac{|g|^4}{r^4}\left(1-\frac{|g|^2}{r^2}\right)^\alpha\frac{dr}{r}=\int^\infty_1r^{-5}(1-r^2)^{2\alpha}<\infty.$$ 
As a result,
$$\|{\sup_{R>0}}^+M^{\alpha}_{R}(x)\|_2\lesssim \|{\sup_{R>0}}^+M^{\alpha+1}_{R}(x)\|_2+\|x\|_2,$$
which yields \eqref{average} by iteration
$$\|{\sup_{R>0}}^+M^{\alpha}_{R}(x)\|_2\lesssim \|{\sup_{R>0}}^+M^{\alpha+2}_{R}(x)\|_2+\|x\|_2\lesssim\|x\|_2$$
where we have used Step 1. And thus we obtain \eqref{2 for inter} for $\delta>1/2$.

We now handle the general case of $\delta>0$. Let $\alpha>-1/2$ and $\beta>1/2$ such that $\delta=\alpha+\beta$. Then using \eqref{identity}, we have
\begin{align*}
B^{\alpha+\beta}_R-\frac{C_{\alpha,\beta}}{C_{\alpha+1,\beta}}B^{\alpha+\beta+1}_R&=C_{\alpha,\beta}R^{-2(\alpha+\beta)}\Big[\int^R_0(R^2-t^2)^{\beta-1}t^{2\alpha+1}B^\alpha_tdt\\
&\quad-R^{-2}\int^R_0(R^2-t^2)^{\beta-1}t^{2\alpha+3}B^{\alpha+1}_t dt\Big]\\
&=C_{\alpha,\beta}R^{-2(\alpha+\beta)}\Big[\int^R_0(R^2-t^2)^{\beta-1}t^{2\alpha+1}(B^\alpha_t-B^{\alpha+1}_t)dt\\
&\quad+\int^R_0(R^2-t^2)^{\beta-1}t^{2\alpha+1}(1-R^{-2}t^2)B^{\alpha+1}_t dt\Big]\\
&:=A_R+B_R.
\end{align*}
Let us first deal with the first part. Using Lemma \ref{theorem:interpolation l infty}, we have
\begin{align*}
\|{\sup_{n}}^+A_{R_n}(x)\|_2&\leq \|{\sup_{n}}^+|A_{R_n}(x)|^2\|^{\frac14}_1 \|{\sup_{n}}^+|(A_{R_n}(x))^*|^2\|^{\frac14}_1.
\end{align*}
Observe that
\begin{align*}
|A_R(x)|&=|C_{\alpha,\beta}|R^{-2(\alpha+\beta)}\Big|\int^R_0(R^2-t^2)^{\beta-1}t^{2\alpha+1}(B^\alpha_t(x)-B^{\alpha+1}_t(x)) dt\Big|\\
&\leq|C_{\alpha,\beta}|R^{-2(\alpha+\beta)}\Big(\int^R_0|(R^2-t^2)^{\beta-1}t^{2\alpha+1}|^2dt\Big)^{\frac12}\\
&\quad\times\Big(\int^R_0|B^\alpha_t(x)-B^{\alpha+1}_t(x)|^2dt\Big)^{\frac12}\lesssim G^\alpha(x),
\end{align*}
since that $\beta>1/2$ yields
$$R^{1-4(\alpha+\beta)}\int^R_0|(R^2-t^2)^{\beta-1}t^{2\alpha+1}|^2dt=\int^1_0|(1-t^2)^{\beta-1}t^{2\alpha+1}|^2dt<\infty.$$
Similarly,
$$|(A_R(x))^*|\lesssim G_*^\alpha(x),$$
and thus
\begin{align*}
\|{\sup_{R>0}}^+A_{R}(x)\|_2\lesssim \|G^\alpha(x)\|^\frac{1}{2}_2 \|G_*^\alpha(x)\|^\frac{1}{2}_2= \|G^\alpha(x)\|_2\lesssim\|x\|_2.
\end{align*}
Now we handle $B_R$. Because
\begin{align*}
B_R&=C_{\alpha,\beta}R^{-2(\alpha+\beta)}\int^R_0(R^2-t^2)^{\beta-1}t^{2\alpha+1}(1-R^{-2}t^2)B^{\alpha+1}_t dt\\
&=C_{\alpha,\beta}R^{-2(\alpha+\beta)-2}\int^R_0(R^2-t^2)^{\beta}t^{2\alpha+1}B^{\alpha+1}_t dt
\end{align*}
and $\beta>1/2$, this term can be estimated as $B^\delta_R$ in the case of $\delta>1/2$. Therefore we get
\begin{align*}
\|{\sup_{R>0}}^+B_{R}(x)\|_2\lesssim\|x\|_2.
\end{align*}
So we conclude
\begin{align*}
\|{\sup_{R>0}}^+B^{\alpha+\beta}_{R}(x)\|_2&\leq \frac{|C_{\alpha,\beta}|}{|C_{\alpha+1,\beta}|}\|{\sup_{R>0}}^+B^{\alpha+\beta+1}_{R}(x)\|_2\\
&\quad\|{\sup_{R>0}}^+A_{R}(x)\|_2+\|{\sup_{R>0}}^+B_{R}(x)\|_2\\
&\lesssim\|x\|_2,
\end{align*}
which finishes the proof of Step 2.

Step 3. The result in the case $p=\infty$ is contained in Theorem~\ref{theorem:criterion1}, while the one in the case of $p=2$ is given by Step 2. Then for $2<p<\infty$, the desired inequality follows from Stein's complex interpolation. This interpolation is now well-known also in the noncommutative framework, we omit it here. Therefore, we conclude the proof of the whole theorem.
\end{proof}

\begin{remark}
	The pointwise convergence of the above Bochner-Riesz means  for  ${1<p<2}$  seems  more delicate.
	However, we can still  construct some finitely supported multipliers satisfying the pointwise convergence in this case. 
	For any $n\in \bN$, we define a multiplier $p_n$ on $\vN$  by
	$$p_n(x)=\sum_{|g|=n}\hat{x}(g)\lambda(g), \qquad x\in \vN,$$
	which is the projection onto the subspace $\mathrm{span}\{\lambda(g): |g|=n\}$.
	Ozawa \cite{ozawa08weak} showed that these operators satisfy  the following estimate
	\begin{equation}\label{eq:character n cb bounded}
		\|p_n: \vN\to \vN\|\leq \beta(n+1)
	\end{equation}
	where $\beta$ is a positive constant independent of $n$.
	For any $N\in \bN$, we set $$m_N(g)=\dsone_{[0, N^2]}(|g|) e^{-\frac{|g|}{N}}, \qquad g\in \Gamma.$$  
	On the other hand, for $t\in \bR_+$, we consider the multiplier $S_t: \lambda(g)\mapsto e^{-t|g|}\lambda(g)$. By  assumption, $(S_t)_{t\in \bR_+}$ is a semigroup of  unital completely positive trace preserving and symmetric maps.
	We write
	\begin{align*}
		e^{-\frac{|g|}{N}}-m_{N}(g)=\dsone_{[N^2+1,\infty)}(|g|) e^{-\frac{|g|}{N}}=\sum_{r=N^2+1}^\infty e^{-\frac{r}{N}}\dsone_{|g|=r}.
	\end{align*}
	Hence, by the estimate (\ref{eq:character n cb bounded}), we have
	\begin{align*}
	\|S_{1/N}-T_{m_{N}}:VN(\Gamma)\to VN(\Gamma)\|&\leq \sum_{r=N^2+1}^\infty \|e^{-\frac{r}{N}}p_r\|	\lesssim \beta\sum_{r=N^2+1}^\infty \frac{N^6}{r^6}(r+1)  
	\\
	&   \lesssim  \beta\frac{1}{(N+1)^2}.		
	\end{align*}By duality and interpolation, for any $1\leq p\leq \infty$, we have 
	$$\|S_{1/N}-T_{m_N}:L_p(\vN)\to L_p(\vN)\|\lesssim  \beta\frac{1}{N^2}.$$
	By Proposition~\ref{prop:maximal inequality of semigroup}, $S_{1/N}$ is of strong type $(p, p)$ for any $1<p<\infty$.
	Hence for any selfadjoint element $x\in L_p(\vN)$ with $1<p<\infty$,
	\begin{align*}
		\|{\sup_{N\in \bN}}^+ T_{m_N}(x)\|_p&\leq \|{\sup_{N\in \bN}}^+S_{1/N}(x)\|_p+\|{\sup_{N\in \bN}}^+ (S_{1/N}-T_{m_N})(x)\|_p\\
		&\lesssim_{\beta,p} \|x\|_p+ \sum_{N\geq 0} \frac{1}{(N+1)^2}\|x\|_p\\
		&\lesssim_{\beta,p} \|x\|_p.
	\end{align*}
	Similarly, for any $2\leq p<\infty$, we have 
	$$\|(T_{m_N}(x))_N\|_{L_p(\cM;\ell_\infty^c)}\lesssim_{\beta, p}  \|x\|_p.$$
	Then by Proposition~\ref{prop:Phix-x in Lp M C0}  it is easy to check that $T_{m_N}x$ converges a.u. to $x$ as $N\to \infty$ for    $x\in L_p(\vN)$ with $2\leq  p<\infty$ and   converges b.a.u.  for      $1<  p<2$. Note that we may also obtain the weak type $(1,1)$ estimate and b.a.u. convergence on $L_1(VN(\Gamma))$ for multipliers $\lambda(g) \mapsto \dsone_{[0,N]}(|g|) e^{-\sqrt{|g|/N}}\lambda(g)$ by the same argument, since the subordinated Poisson semigroup $\lambda(g)\mapsto e^{-t\sqrt{|g|}}\lambda(g)$ is of weak type $(1,1)$ according to \cite[Remark 4.7]{jungexu07ergodic}.
	%
	
	The above arguments work for all groups with the rapid decay property with respect to a conditionally negative definite length function. 
\end{remark}

Let us present briefly some similar results in the setting of free semicircular systems (\cite{voiculescu85sym}) as mentioned in Example \ref{eg:more vna} (3). Let $H_{\mathbb R}$ be a real Hilbert spaces with an orthonormal basis $(e_i)_{i\in I}$ and $H_{\mathbb C}$ be its complexification. We consider the free Fock spaces defined by the Hilbertian direct sum
$$\mathcal F _0 (H_\mathbb{C}) = \oplus_{n\geq 0} H_\mathbb{C}^{\otimes n}$$
and the creation operators
$$l^*(f):\mathcal F _0 (H_\mathbb{C}) \to \mathcal F _0 (H_\mathbb{C}),\quad \xi\mapsto f\otimes \xi.$$
The adjoint of a creation operator $l^*(f)$ is denoted by $l(f)$. Let $\mathit{\Omega}$ denote the unit element in $H_\mathbb{C}^{\otimes 0} = \mathbb C$. By definition the free Gaussian von Neumann algebra $\Gamma_0 (H_{\mathbb R})$ is the von Neumann subalgebra generated by $W(f)= l(f)+l^*(f)$ for $f\in H_\mathbb{R}$ in $B(\mathcal F _0 (H_\mathbb{C}))	$. This algebra admits a natural faithful normal tracial state 
$$\tau(x) = \langle x\mathit{\Omega},\mathit{\Omega} \rangle_{\mathcal F _0 (H_\mathbb{C})},\qquad x\in \Gamma_0 (H_{\mathbb R}).$$
It is well known that each $\xi \in H_\mathbb{C}^{\otimes n} $ corresponds to an element $W(\xi)\in \Gamma_0 (H_{\mathbb R})$ such that $W(\xi)\mathit{\Omega} = \xi$. The semigroup of maps
$$S_t:{W(\xi)\mapsto e^{-tn}W(\xi)} ,\quad \xi\in H_\mathbb{C}^{\otimes n},\,n\in \mathbb N$$
extends to a semigroup of  unital completely positive $\tau$-preserving and symmetric maps on $\Gamma_0 (H_{\mathbb R})$, which is called the free Ornstein-Uhlenbeck semigroup (see e.g. \cite{biane97hc}). As before, we may consider the radial multipliers
$$ B_N^\delta (W(\xi)) =   \left(1-\frac{n^2}{N^2}\right)^\delta \dsone_{[0,N]}(n) W(\xi),\quad \xi\in H_\mathbb{C}^{\otimes n},\,n\in \mathbb N.$$ 
Note that by \cite[Example 3.4, Equality (4)]{meidelasalle17cbofheatsemigroups}, the trace norm of the Hankel matrix associated with the symbols of $ B_N^\delta$ is uniformly bounded when ${\real (\delta)>1}$, thus by \cite[Theorem 3.5 and Proposition 3.3]{houdayerricad11freeaw}, $ B_N^\delta$ defines a family of completely bounded maps on $\Gamma_0 (H_{\mathbb R})$ with $\sup_N \|B_N^\delta \|_{cb} <\infty$ as soon as $\real (\delta) >1$. Then repeating the proof of Theorem \ref{thm:bochnerrieszhyperbolic}, we obtain the following analogous result.
	\begin{theorem}
		Let $2\leq p\leq\infty$ and $\delta>1-\frac2p$. Then there exists a constant $c_{p,\delta}$ such that 
		$$\|{\sup_{N\in \bN}}^+B_N^\delta( x)\|_p\leq c_{p,\delta}\|x\|_p,\qquad x\in L_p(\Gamma_0 (H_{\mathbb R}),\tau),$$
		and for any $x\in L_p(\Gamma_0 (H_{\mathbb R}),\tau)$, $B_N^\delta(x)$ converges a.u. to  $x$ as $N\to \infty$.
	\end{theorem}
	
\subsection{Smooth positive definite radial kernels on free groups}\label{sec:radial free}
Using our main results, we may  provide a wide class of  completely positive smooth multipliers on free groups satisfying the pointwise convergence apart from Poisson semigroups. 
To this end we will need the following characterization of radial positive definite functions on free groups. In the following $\mathbb F _d$ will denote the free group with $d$ generators ($2\leq d\leq \infty$).
\begin{theorem}[{\cite[Theorem 1.1]{haagerupknudby15LKfreegroups} and \cite[Theorem 1.2]{vergara2019positive}}]\label{theorem:vergera}
	Let $\nu$ be a positive Borel measure on $[-1, 1]$. Define  a function $\varphi$ on $\bN$ by 
	$$\varphi(k)=\int_{-1}^{1}x^kd\nu(x), \qquad k\in \bN.$$
	Then $\dot{\varphi}(g):=\varphi(|g|)$ is a  positive definite function on $\bF_\infty$, where $|\  |$ is the word length function.
\end{theorem}

Then we get the following proposition. Note that if $\nu$ is the Dirac measure on $0$ in this proposition, then this statement amounts to the almost uniform convergence of Poisson semigroups on $VN(\mathbb F _d)$ proved in \cite{jungexu07ergodic}.
\begin{prop}\label{cor:p.d. multipliers on F infty}
	Let $\nu$ be an arbitrary positive Borel measure supported on $[-1, 1]$ with $\nu([-1,1])=1$ and write $d\nu_t (x)=  d\nu ( tx)$ for all $t>0$.
	For any $t>0$, set $$m_t( g )=\int_{\mathbb R} x^{|g|} d\nu_t (x-e^{-\frac{2}{t}}),  \qquad g\in\mathbb F _d,$$
	where  $|\  |$ is the usual word length function.
	Then there exist an absolute positive number $t_0>0$ and  a constant $c>0$ such that
	for all $1<p<\infty $ and all ${x\in L_p (VN(\mathbb F _d))}$,
	$$\|(T_{m_t} x)_{t\geq t_0}\|_{L_p(\cM;\ell_\infty)} \leq c \|x\|_p  \quad\text{and}\quad  T_{m_t} x\to x \text{ b.a.u. (a.u. if $p\geq 2$) as }t\to\infty.$$
\end{prop}
\begin{proof}
	Using integration by substitution with $y= t(x-e^{-\frac{2}{t}})$, we have
	\begin{equation*}
		m_t( g )=\int_{\mathbb R} x^{|g|} d\nu_t (x-e^{-\frac{2}{t}}) =\int_{-1}^1 \left(\frac{y}{ t}+e^{-\frac{2}{t}} \right)^{|g|} d\nu(y).
	\end{equation*}	
	By some fundamental analysis, we can find a number $t_0$ large enough such that for any $t\geq t_0$,
	$$e^{-\frac{2}{t}}-\frac{1}{ t}\geq e^{-\frac{4}{t}}\add{and} \frac{1}{ t}+e^{-\frac{2}{t}}<e^{-\frac{2}{3t}}.$$
	Then for  any $t\geq t_0$, 
	\begin{equation}\label{eq:condition1 for p.d. on free group}
		|m_t(g)|\leq\left(\frac{1}{ t}+e^{-\frac{2}{t}}\right)^{|g|}\leq e^{-\frac{2|g|}{3t}} \lesssim \frac{t}{|g|}, 
	\end{equation}
	$$|1-m_t(g)|\leq 1- \left( -\frac{1}{ t}+ e^{-\frac{2}{t}}\right) ^{|g|}\leq\left( 1- e^{-\frac{4|g|}{t}}\right) \lesssim \frac{|g|}{t}. $$
	Moreover we note that
	\begin{align*}
		\frac{d^v}{ dt ^v} m_t(|g|)=\int_{-1}^{1}\frac{d^v}{ dt ^v}\left(\frac{y}{ t}+e^{-\frac{2}{t}} \right)^{|g|}  d\nu(y).
	\end{align*}
	Set $f(t)=\frac{y}{ t}+e^{-\frac{2}{t}}$. By a straightforward computation,  
	\begin{equation}\label{eq:high div for f(t)}
	|\frac{d^v}{d t^v} f(t)|\lesssim_{v}  \frac{1}{ t^{v+1}}.
	\end{equation}
	Recall the Fa\'a di Bruno formula:
	$$\frac{d^v}{dt^v} F(f(x))=\sum_{P \in \mathscr{P} (v)} F^{(|P|)}(f(t))\cdot\prod_{B\in P} f^{(|B|)}(t)  $$
	where  $\mathscr P (v)$ is the set of all partitions of  $\{1, \cdots, v\}$,  $B\in P$ means that $B$ is a block of the partition $P$, and $|B|$ denotes the size of the block $B$ and $|P|$ means the number of blocks. 
	Similar as \eqref{eq:condition1 for p.d. on free group}, $|f(t)|\leq e^{-\frac{2}{3t}}$. Then using the Fa\'a di Bruno  formula and  \eqref{eq:high div for f(t)},  we see that for any $v\geq 1$ and $y\in[-1,1]$,
	\begin{align*}
		&\quad\ \frac{d^v}{ dt ^v}\left(\frac{y}{ t}+e^{-\frac{2}{t}} \right)^{k}\\&=\sum_{P \in \mathscr P (v), |P|\leq \min\{v,k\}} k(k-1)\cdot (k-|P|+1)f(t)^{k-|P|}\cdot\prod_{B\in P} f^{(|B|)}(t) \\
		&\lesssim_v \sum_{P \in \mathscr P (v), |P|\leq \min\{v,k\}} k(k-1)\cdot (k-|P|+1) e^{-\frac{2(k-|P|)}{3t}} \cdot\prod_{B\in P} \frac{1}{ t^{|B|+1}}\\
		&\lesssim_v \sum_{P \in \mathscr P (v), |P|\leq \min\{v,k\}} k(k-1)\cdot (k-|P|+1)\frac{t^{|P|}}{(k-|P|)^{|P|}} \cdot  \frac{1}{t^{v+|P|}}\\
		&\lesssim_v  \sum_{P \in \mathscr P (v)} 2^{|P|} \frac{1}{t^v} \lesssim_v \frac{1}{t^v}.
	\end{align*}
	Therefore, $(m_t)_{t\geq t_0}$ satisfies \eqref{eq:criterion a3} in \textbf{(A2)}.
	
	Note that $\mathbb F _d$ is a subgroup of $\mathbb F_\infty$, and $\nu_t(\cdot -e^{-t/2})$ is supported in $[-1,1]$ for large $t>0$. So $m_t $  is positive definite on $\mathbb F _d$ by Theorem~\ref{theorem:vergera} for $t\geq t_0$. Also, for any $t\geq t_0$
	$T_{m_t}(1)=1$ since  $\nu([-1,1])=1$. Thus for any $t\geq t_0$, $T_{m_t}$ is unital completely positive and it extends to a contraction on $L_p(VN(\bF_d))$. Moreover, the natural length function $|\  |$ is conditionally negative definite. On the other hand, for all $1<p<\infty$  we can always find a positive integer $\eta$ depending on $p$  such that   $1+\frac{1}{2\eta}<p$. Then the proof is complete by  applying Theorem~\ref{theorem:criterion2}. 
\end{proof}

We remark that the totally same argument applies to many other examples of groups acting on homogeneous trees.
	
\section[Noncommutative Hardy-Littlewood maximal operators]{Dimension free bounds of noncommutative Hardy-Littlewood maximal operators}\label{sec:dimension free}
Our results in particular apply to the problem of dimension free estimates for Hardy-Littlewood maximal operators. 
Let $B$ be a symmetric convex body in $\bR^d$. We consider the  associated averages
$$\Phi _t(f)(x)=\frac{1}{\mu(B)}\int_{B} f(x-\frac{y}{t})dy,\quad f\in L_p(\mathbb R ^d), \ x\in \bR^d ,\  t>0.$$ 

Let $\cN$ be a semifinite von Neumann algebra equipped with a normal semifinite trace $\nu$ and $L_p(\bR^d;L_p(\cN))$ be the Bochner $L_p(\cN)$-valued $L_p$-spaces. 
Recall that we may view the space $L_p(\bR^d;L_p(\cN))$ as a noncommutative $L_p$-space associated with the von Neumann algebra $L_\infty (\bR^d) \overline{\otimes} \cN$: for any $1\leq p<\infty$,
$$L_p(L_\infty (\bR^d) \overline{\otimes} \cN, \int {\otimes} \nu)\cong L_p(\bR^d; L_p(\cN)).$$
We could then extend Bourgain's results for the corresponding Hardy-Littlewood maximal operators on noncommutative $L_p$-spaces  $\LpR$.

\begin{theorem}\label{cor:dimension free of HL maximal inequality in text}
	Let $B$ be a  symmetric convex body in $\bR^d$ and $\cN$ a semifinite von Neumann algebra. Define $\Phi_r :\LpR\to\LpR$ by $$\Phi _r(f)(x)=\frac{1}{\mu(B)}\int_{B} f(x-ry)dy.$$
	Then there exist constants $C_p>0$ independent of $d$ and $B$ such that the following holds:
	
	\emph{(1)} For any $1<p< \infty$,
	$$\|{\sup_{j\in \bZ}}^+ \Phi _{2^j}(f)\|_p\leq C_p\|f\|_p, \qquad  f\in \LpR.$$
	
	\emph{(2)} For any $\frac{3}{2}<p< \infty$,
	$$\|{\sup_{r> 0}}^+ \Phi _{r}(f)\|_p\leq C_p\|f\|_p, \qquad  f\in \LpR.$$	

	\emph{(3)} If $B$ is the $\ell_q$-ball $\{(x_i)_{i=1}^d: \sum_{i=1}^d |x_i|^q\leq 1\}$  with  $q\in 2\bN$, then for any $1<p<\infty$,
	$$\|{\sup_{r> 0}}^+ \Phi _{r}(f)\|_p\leq C_p\|f\|_p, \qquad  f\in \LpR.$$	
\end{theorem}
\begin{proof}
	Without loss of generality, we assume $\mu(B)=1$.	Let
	$$m_t(\xi) =t^d\widehat{\dsone_{t^{-1}B}}(\xi)=\widehat{\dsone_{B}}(\xi/t)$$ be the Fourier transform of the kernel of the above operator $\Phi _t$. Then we may view $\Phi _t$ as the Fourier multiplier so that $\widehat{\Phi _t f} = m_t \hat f $.
	
	Since the bound of a maximal operator is invariant under invertible linear transforms, thus by a transform, we can assume that $B$ is in the isotropic position with isotropic constant $L=L(B)$ (see Lemma 2 in \cite{bourgain86onhigh}) and $\widehat{	\dsone_{B}}$ satisfies the following estimates (see the computations in Section 4 of \cite{bourgain86onhigh}):  there exists a constant $c$  independent of $d$ and $B$ such that
	\begin{align*}
		|\widehat{	\dsone_{B}}(\xi)|\leq c\frac{1}{|\xi|L}, \quad 	|1-\widehat{	\dsone_{B}}(\xi)|\leq cL|\xi|,\quad
		|\langle \nabla \widehat{ \dsone_{B}}(\xi), \xi  \rangle|\leq c, \quad \xi \in \bR^d.
	\end{align*}
	On the other hand, setting $\zeta=\xi/t=(\xi_l/t)_{l=1}^d$, we have 
	$$\frac{d}{dt}m_t(\xi)=\frac{d }{dt}\widehat{	\dsone_{B}}(\zeta)=\sum_{l=1}^d \frac{\partial \zeta_l}{\partial t}\cdot \frac{\partial \widehat{\dsone_{B}}(\zeta)}{\partial \zeta_l}
	=-\sum_{l=1}^d \frac{ \zeta_l}{ t} \partial_l \widehat{\dsone_{B}}(\zeta)=-\frac{1}{t}\langle \nabla \widehat{ \dsone_{B}}(\zeta), \zeta \rangle.$$
	Therefore, $(m_t)_t$ satisfies \textbf{(A2)} with $\eta=1$. Then, applying Theorem~\ref{theorem:criterion2},  Remark~\ref{rmk: A1 prime and change index set N to Z}~(2), Remark \ref{rk:mtheorem for cb}, and  Lemma~\ref{lem:eg rota}~(2) to the modified Poisson semigroup $(P_{tL})_{t\in\mathbb R_+}$ on $\mathbb R ^d$, we immediately get the desired assertions (1) and (2) of Theorem \ref{cor:dimension free of HL maximal inequality in text}.

	We use the similar arguments for higher derivatives: 
	\begin{align*}
		\frac{d^2}{dt^2}\dsone_B(\zeta)&=
		\sum_{l=1}^d\frac{\partial \zeta_l}{\partial t}\cdot \frac{\partial}{\partial \zeta_l}\left(\sum_{l=1}^d\frac{\partial \zeta_l}{\partial t}\cdot \frac{\partial \dsone_B(\zeta)}{\partial \zeta_l} \right)
		=\sum_{l=1}^d\frac{ \zeta_l}{ t}\cdot \frac{\partial}{\partial \zeta_l}\left(\sum_{l=1}^d\frac{ \zeta_l}{ t}\cdot \frac{\partial \dsone_B(\zeta)}{\partial \zeta_l} \right)  
	\\
		&	=\frac{1}{t^2}\left( \sum_{l=1}^d\zeta_{l} \partial_{l}\left( \dsone_B(\zeta)\right) + \sum_{l_1, l_2=1}^d\zeta_{l_1} \zeta_{l_2}\partial_{l_1}\partial_{l_2}\left( \dsone_B(\zeta)\right) \right) ,
	\end{align*}
	and
	\begin{align*}
		  \frac{d^3}{dt^3}\big(\dsone_B(\zeta)\big) 
		&=\frac{1}{t^3}\Bigg( \sum_{l=1}^d\zeta_{l} \partial_{l}\big(\dsone_B(\zeta)\big)+2 \sum_{l_1, l_2=1}^d\zeta_{l_1} \zeta_{l_2}\partial_{l_1}\partial_{l_2}\big(\dsone_B(\zeta)\big)\\
		&\qquad\qquad
		+\sum_{l_1, l_2,l_3=1}^d\zeta_{l_1} \zeta_{l_2}\zeta_{l_3}\partial_{l_1}\partial_{l_2}\partial_{l_3}\big(\dsone_B(\zeta)\big) \Bigg) .
	\end{align*}
	Repeating this process,  we get
	$$\frac{d^v}{dt^v}\big(\dsone_B(\zeta)\big) =\frac{1}{t^v}\sum_{k=1}^v\left(c_k \sum_{l_1, l_2, \dots, l_k=1}^d \zeta_{l_1} \zeta_{l_2}\cdots \zeta_{l_k}\partial_{l_1}\partial_{l_2}\cdots \partial_{l_k}\right)\big(\dsone_B(\zeta)\big) $$
	where $c_k$ are constants only depending  on $k$. Recall that $$\partial_{l_1}\partial_{l_2}\cdots \partial_{l_k} \widehat{	\dsone_{B}}(\zeta)=\left((-2\pi i)^k x_{l_1}x_{l_2}\cdots x_{l_k}\dsone_B \right)^{\widehat{	\quad}} (\zeta).$$
	Then 
	\begin{align*}
		\left|\frac{d^v}{dt^v}m_t(\xi)\right|&
		=\frac{1}{t^v}\left|\sum_{k=1}^v \left(c_k \sum_{l_1, l_2, \dots, l_k=1}^d \zeta_{l_1} \zeta_{l_2}\cdots \zeta_{l_k}\partial_{l_1}\partial_{l_2}\cdots \partial_{l_k}\widehat{	\dsone_{B}}(\zeta)\right)\right|\\
		&\lesssim_v\frac{1}{t^v}\left|\sum_{k=1}^v\left(  \sum_{l_1, l_2 \cdots, l_k=1}^d  \zeta_{l_1}\zeta_{l_2}\cdots \zeta_{l_k}  \left((2\pi i)^k x_{l_1}x_{l_2}\cdots x_{l_k}{\dsone_{B}}(x)\right) ^{\widehat{\quad  }}(\zeta) \right)  \right|\\
		&\lesssim_v\frac{1}{t^v} \sum_{k=1}^v\left|\int_{\bR^d} e^{-2\pi i \langle x, \zeta\rangle} \langle x, \zeta \rangle^k\dsone_{B}(x)dx\right|.
	\end{align*}
	In particular, if  $B$ is the $\ell_q$-ball  with  $q\in 2\bN$, the computations in \cite{bourgain87ondimensionfree}  assert that for any $k\geq 0$ there exists a constant $c_k^\prime$ independent of $d$ such that 
	$$\left|\int_{\bR^d} e^{-2\pi i \langle x, \zeta\rangle} \langle x, \zeta \rangle^k\dsone_{B}(x)dx\right|\leq c_k^\prime.$$	 
	Hence for $\ell_q$-ball  with  $q\in 2\bN$, $(m_t)_{t\in\mathbb R_+}$ satisfies \textbf{(A2)} for any $\eta> 0$.
	For any ${1<p\leq 2}$, choosing an $\eta$ large enough (depending on $p$) and applying  Theorem~\ref{theorem:criterion2}~(2), Remark \ref{rk:mtheorem for cb} and  Lemma~\ref{lem:eg rota}~(2), we obtain Theorem \ref{cor:dimension free of HL maximal inequality in text} (3).
\end{proof}

\begin{remark}
	At the moment of writing, it seems that our approach is not enough to establish dimension free bounds for other $\ell_q$-balls with $q\in[1,\infty]\setminus 2\bN$ and new ideas are certainly required. The corresponding results in the classical setting are given by M\"uller \cite{muller90maximal}   (for $q\in[1,\infty)\setminus 2\bN$) and Bourgain \cite{bourgain14HLcube} (for $q=\infty$). From Theorem \ref{cor:dimension free of HL maximal inequality in text}, it is naturally conjectured that the noncommutative analogues of their results should still hold.
\end{remark}


\backmatter
\newcommand{\etalchar}[1]{$^{#1}$}
\providecommand{\bysame}{\leavevmode\hbox to3em{\hrulefill}\thinspace}
\providecommand{\MR}{\relax\ifhmode\unskip\space\fi MR }
\providecommand{\MRhref}[2]{%
	\href{http://www.ams.org/mathscinet-getitem?mr=#1}{#2}
}
\providecommand{\href}[2]{#2}



\begin{thebibliography}{CdPSW00}
	
	\bibitem[AdLW18]{arnodelaatwahl18Fourier}
	Y.~Arano, T.~de~Laat, and J.~Wahl, \emph{The {F}ourier algebra of a rigid
		{$C^*$}-tensor category}, Publ. Res. Inst. Math. Sci. \textbf{54} (2018),
	no.~2, 393--410. \MR{3784875}
	
	\bibitem[BC09]{bedosconti09twisted}
	E.~B\'{e}dos and R.~Conti, \emph{On twisted {F}ourier analysis and convergence
		of {F}ourier series on discrete groups}, J. Fourier Anal. Appl. \textbf{15}
	(2009), no.~3, 336--365. \MR{2511867}
	
	\bibitem[BC12]{bedosconti12twisted2}
	\bysame, \emph{On discrete twisted {$\rm C^*$}-dynamical systems, {H}ilbert
		{$\rm C^*$}-modules and regularity}, M\"{u}nster J. Math. \textbf{5} (2012),
	183--208. \MR{3047632}
	
	\bibitem[BCO17]{bekjanchenoscekowski17ncmaximal}
	T.~N. Bekjan, Z.~Chen, and A.~Os{\k e}kowski, \emph{Noncommutative maximal
		inequalities associated with convex functions}, Trans. Amer. Math. Soc.
	\textbf{369} (2017), no.~1, 409--427. \MR{3557778}
	
	\bibitem[Bia97]{biane97hc}
	P.~Biane, \emph{Free hypercontractivity}, Comm. Math. Phys. \textbf{184}
	(1997), no.~2, 457--474. \MR{1462754}
	
	\bibitem[Bou86a]{bourgain86onhigh}
	J.~Bourgain, \emph{On high-dimensional maximal functions associated to convex
		bodies}, Amer. J. Math. \textbf{108} (1986), no.~6, 1467--1476. \MR{868898}
	
	\bibitem[Bou86b]{bourgain86ontheLp}
	\bysame, \emph{On the {$L^p$}-bounds for maximal functions associated to convex
		bodies in {${\bf R}^n$}}, Israel J. Math. \textbf{54} (1986), no.~3,
	257--265. \MR{853451}
	
	\bibitem[Bou87]{bourgain87ondimensionfree}
	\bysame, \emph{On dimension free maximal inequalities for convex symmetric
		bodies in {${\bf R}^n$}}, Geometrical aspects of functional analysis
	(1985/86), Lecture Notes in Math., vol. 1267, Springer, Berlin, 1987,
	pp.~168--176. \MR{907693}
	
	\bibitem[Bou14]{bourgain14HLcube}
	\bysame, \emph{On the {H}ardy-{L}ittlewood maximal function for the cube},
	Israel J. Math. \textbf{203} (2014), no.~1, 275--293. \MR{3273441}
	
	\bibitem[BMSW18]{bourgainmireksteinwrobel18dimfreevar}
	J.~Bourgain, M.~Mirek, E.~M. Stein, and B.~Wr\'{o}bel, \emph{On dimension-free
		variational inequalities for averaging operators in {$\mathbb R^d$}}, Geom.
	Funct. Anal. \textbf{28} (2018), no.~1, 58--99. \MR{3777413}
	
	
	\bibitem[BF06]{bozejkofendler06divergence}
	M.~Bo\.{z}ejko and G.~Fendler, \emph{A note on certain partial sum operators},
	Quantum probability, Banach Center Publ., vol.~73, Polish Acad. Sci. Inst.
	Math., Warsaw, 2006, pp.~117--125. \MR{2423120}
	
	
	\bibitem[BS91]{bozejkospeicher91anexample}
	M.~Bo\.{z}ejko and R.~Speicher, \emph{An example of a generalized {B}rownian
		motion}, Comm. Math. Phys. \textbf{137} (1991), no.~3, 519--531. \MR{1105428}
	
	\bibitem[BS94]{bozejkospeicher94cp}
	\bysame, \emph{Completely positive maps on {C}oxeter groups, deformed
		commutation relations, and operator spaces}, Math. Ann. \textbf{300} (1994),
	no.~1, 97--120. \MR{1289833}
	
	\bibitem[Bra17]{brannan17approximation}
	M.~Brannan, \emph{Approximation properties for locally compact quantum groups},
	Topological quantum groups, Banach Center Publ., vol. 111, Polish Acad. Sci.
	Inst. Math., Warsaw, 2017, pp.~185--232. \MR{3675051}
	
	\bibitem[Bre14]{breuillard14poly}
	E.~Breuillard, \emph{Geometry of locally compact groups of polynomial growth
		and shape of large balls}, Groups Geom. Dyn. \textbf{8} (2014), no.~3,
	669--732. \MR{3267520}
	
	\bibitem[BLD13]{breuillardledoone13nilpotent}
	E.~Breuillard and E.~Le~Donne, \emph{On the rate of convergence to the
		asymptotic cone for nilpotent groups and sub{F}insler geometry}, Proc. Natl.
	Acad. Sci. USA \textbf{110} (2013), no.~48, 19220--19226. \MR{3153949}
	
	
	\bibitem[BO08]{brownozawa08bookC*}
	N.~P. Brown and N.~Ozawa, \emph{{$C^*$}-algebras and finite-dimensional
		approximations}, Graduate Studies in Mathematics, vol.~88, American
	Mathematical Society, Providence, RI, 2008. \MR{2391387}
	
	\bibitem[Car86]{carbery89convexbody}
	A.~Carbery, \emph{An almost-orthogonality principle with applications to
		maximal functions associated to convex bodies}, Bull. Amer. Math. Soc. (N.S.)
	\textbf{14} (1986), no.~2, 269--273. \MR{828824}
	
	\bibitem[CS15]{casperaskalski15Haagerup}
	M.~Caspers and A.~Skalski, \emph{The {H}aagerup approximation property for von
		{N}eumann algebras via quantum {M}arkov semigroups and {D}irichlet forms},
	Comm. Math. Phys. \textbf{336} (2015), no.~3, 1637--1664. \MR{3324152}
	
	\bibitem[CWW15]{chenwang15truncationfourier}
	X.~Chen, Q.~Wang, and X.~Wang, \emph{Truncation approximations and spectral
		invariant subalgebras in uniform {R}oe algebras of discrete groups}, J.
	Fourier Anal. Appl. \textbf{21} (2015), no.~3, 555--574. \MR{3345366}
	
	\bibitem[CDHX17]{chendinghongxiao17somejump}
	Y.~Chen, Y.~Ding, G.~Hong, and J.~Xiao, \emph{Some jump and variational
		inequalities for the {Calder\'{o}n} commutators and related operators}, 2017,
	\href{http://arxiv.org/abs/1709.03127}{{\ttfamily arXiv:1709.03127
			[math.CA]}}.
	
	\bibitem[CXY13]{chenxuyin13harmonic}
	Z.~Chen, Q.~Xu, and Z.~Yin, \emph{Harmonic analysis on quantum tori}, Comm.
	Math. Phys. \textbf{322} (2013), no.~3, 755--805. \MR{3079331}
	
	\bibitem[CCJ{\etalchar{+}}01]{cherixetal01haagerup}
	P.-A. Cherix, M.~Cowling, P.~Jolissaint, P.~Julg, and A.~Valette, \emph{Groups
		with the {H}aagerup property}, Modern Birkh\"{a}user Classics,
	Birkh\"{a}user/Springer, Basel, 2001, Gromov's a-T-menability, Paperback
	reprint of the 2001 edition. \MR{3309999}
	
	\bibitem[CS13]{chifansinclair13strongsolidhyperbolic}
	I.~Chifan and T.~Sinclair, \emph{On the structural theory of {${\rm II}_1$}
		factors of negatively curved groups}, Ann. Sci. \'{E}c. Norm. Sup\'{e}r. (4)
	\textbf{46} (2013), no.~1, 1--33 (2013). \MR{3087388}
	
	\bibitem[CL21]{chilinlitvinov2016ergodic}
	V.~Chilin and S.~Litvinov, \emph{On individual ergodic theorems for semifinite
		von {N}eumann algebras}, J. Math. Anal. Appl. \textbf{495} (2021), no.~1,
	124710, 16. \MR{4172864}
	
	\bibitem[Chr88]{Chr88}
	M.~Christ, \emph{Weak type (1,1) bounds for rough operators}, 
	Ann. of Math. (2) \textbf{128} (1988), no.~1, 19-42. \MR{951506}
	
	\bibitem[CS16]{ciprianisauvageot16negative}
	F.~Cipriani and J.-L. Sauvageot, \emph{Negative definite functions on groups
		with polynomial growth}, Noncommutative analysis, operator theory and
	applications, Oper. Theory Adv. Appl., vol. 252, Birkh\"{a}user/Springer,
	[Cham], 2016, pp.~97--104. \MR{3526954}
	
	\bibitem[CS17]{ciprianisauvageot17amenability}
	\bysame, \emph{Amenability and subexponential spectral growth rate of
		{D}irichlet forms on von {N}eumann algebras}, Adv. Math. \textbf{322} (2017),
	308--340. \MR{3720800}
	
	\bibitem[CdPSW00]{clementdepagtersukochev00Schauder}
	P.~Cl\'{e}ment, B.~de~Pagter, F.~Sukochev, and H.~Witvliet, \emph{Schauder
		decomposition and multiplier theorems}, Studia Math. \textbf{138} (2000),
	no.~2, 135--163. \MR{1749077}
	
	\bibitem[CAGPP20]{condegonzalezparcet20sigma}
	J.~Conde-Alonso, A.~Gonz\'{a}lez-P\'{e}rez and J. Parcet, \emph{Noncommutative strong maximals and almost uniform convergence	in several directions}, Forum Math. Sigma, \textbf{8} (2020), Paper No. e57, 39. \MR{4179649}

	\bibitem[CH89]{cowlinghaagerup89cb}
	M.~Cowling and U.~Haagerup, \emph{Completely bounded multipliers of the
		{F}ourier algebra of a simple {L}ie group of real rank one}, Invent. Math.
	\textbf{96} (1989), no.~3, 507--549. \MR{996553}
	
	\bibitem[Dab10]{dabrowski2010dilation}
	Y.~Dabrowski, \emph{A non-commutative path space approach to stationary free
		stochastic differential equations}, 2010,
	\href{http://arxiv.org/abs/1006.4351}{{\ttfamily arXiv:1006.4351 [math.OA]}}.
	
	\bibitem[Daw12]{daws12cpmultiplier}
	M.~Daws, \emph{Completely positive multipliers of quantum groups}, Internat. J.
	Math. \textbf{23} (2012), no.~12, 1250132, 23. \MR{3019431}
	
	\bibitem[DFSW16]{dawsfimaskalskiwhit16haagerup}
	M.~Daws, P.~Fima, A.~Skalski, and S.~White, \emph{The {H}aagerup property for
		locally compact quantum groups}, J. Reine Angew. Math. \textbf{711} (2016),
	189--229. \MR{3456763}
	
	\bibitem[DCFY14]{decommerfreslon14CCAP}
	K.~De~Commer, A.~Freslon, and M.~Yamashita, \emph{C{CAP} for universal discrete
		quantum groups}, Comm. Math. Phys. \textbf{331} (2014), no.~2, 677--701, With
	an appendix by Stefaan Vaes. \MR{3238527}
	
	\bibitem[DJ04]{defantjunge04maximal}
	A.~Defant and M.~Junge, \emph{Maximal theorems of {M}enchoff-{R}ademacher type
		in non-commutative {$L_q$}-spaces}, J. Funct. Anal. \textbf{206} (2004),
	no.~2, 322--355. \MR{2021850}
	
	\bibitem[DGM18]{deleavalguedonmaurey18dimfreesurvey}
	L.~Deleaval, O.~Gu\'{e}don, and B.~Maurey, \emph{Dimension free bounds for the
		{H}ardy-{L}ittlewood maximal operator associated to convex sets}, Ann. Fac.
	Sci. Toulouse Math. (6) \textbf{27} (2018), no.~1, 1--198. \MR{3771542}
	
	\bibitem[DHL17]{dinghongliu17jump}
	Y.~Ding, G.~Hong, and H.~Liu, \emph{Jump and variational inequalities for rough
		operators}, J. Fourier Anal. Appl. \textbf{23} (2017), no.~3, 679--711.
	\MR{3649476}
	
	\bibitem[Dir15]{dirksen15interpolation}
	S.~Dirksen, \emph{Weak-type interpolation for noncommutative maximal
		operators}, J. Operator Theory \textbf{73} (2015), no.~2, 515--532.
	\MR{3346135}
	
	\bibitem[DZ19]{duzhang19sharpl2schrodinger}
	X.~Du and R.~Zhang, \emph{Sharp {$L^2$} estimates of the {S}chr\"{o}dinger
		maximal function in higher dimensions}, Ann. of Math. (2) \textbf{189}
	(2019), no.~3, 837--861. \MR{3961084}
	
	\bibitem[DRdF86]{duoandikoetxearubio86max}
	J.~Duoandikoetxea and J.~L. Rubio~de Francia, \emph{Maximal and singular
		integral operators via {F}ourier transform estimates}, Invent. Math.
	\textbf{84} (1986), no.~3, 541--561. \MR{837527}
	
	
	\bibitem[GdlH90]{ghysdelaharpe90sur}
	E.~Ghys and P.~de~la Harpe (eds.), \emph{Sur les groupes hyperboliques
		d'apr\`es {M}ikhael {G}romov}, Progress in Mathematics, vol.~83,
	Birkh\"{a}user Boston, Inc., Boston, MA, 1990, Papers from the Swiss Seminar
	on Hyperbolic Groups held in Bern, 1988. \MR{1086648}
	
	\bibitem[GPJP20]{gonzalezjungeparcet19singular}
	A.~Gonz\'{a}lez-P\'{e}rez, M.~Junge, and J.~Parcet, \emph{Singular integrals in
		quantum {E}uclidean spaces}, Mem. Amer. Math. Soc. (2020), in press.
	
	\bibitem[Gra08]{grafakos08classical}
	L.~Grafakos, \emph{Classical {F}ourier analysis}, second ed., Graduate Texts in
	Mathematics, vol. 249, Springer, New York, 2008. \MR{2445437}
	
	\bibitem[Gro87]{gromov87hyperbolic}
	M.~Gromov, \emph{Hyperbolic groups}, Essays in group theory, Math. Sci. Res.
	Inst. Publ., vol.~8, Springer, New York, 1987, pp.~75--263. \MR{919829}
	
	\bibitem[Haa79]{haagerup78map}
	U.~Haagerup, \emph{An example of a nonnuclear {$C\sp{\ast} $}-algebra, which
		has the metric approximation property}, Invent. Math. \textbf{50} (1978/79),
	no.~3, 279--293. \MR{520930 (80j:46094)}
	
	\bibitem[HJX10]{haagerupjungexu10reduction}
	U.~Haagerup, M.~Junge, and Q.~Xu, \emph{A reduction method for noncommutative
		{$L_p$}-spaces and applications}, Trans. Amer. Math. Soc. \textbf{362}
	(2010), no.~4, 2125--2165. \MR{2574890}
	
	\bibitem[HK15]{haagerupknudby15LKfreegroups}
	U.~Haagerup and S.~Knudby, \emph{A {L}\'{e}vy-{K}hinchin formula for free
		groups}, Proc. Amer. Math. Soc. \textbf{143} (2015), no.~4, 1477--1489.
	\MR{3314063}
	
	\bibitem[Her54]{Her54} C. S. Herz, \emph{On the mean inversion of Fourier and Hankel transforms}, Proc. Nat. Acad. Sci. U.S.A.
	\textbf{40} (1954), 996--999. \MR{63477}
	
	\bibitem[Hon20]{hong18hl}
	G.~Hong, \emph{Non-commutative ergodic averages of balls and spheres over
		{E}uclidean spaces}, Ergodic Theory Dynam. Systems \textbf{40} (2020), no.~2,
	418--436. \MR{4048299}
	
	\bibitem[HLW21]{hongwangliao2017noncommutative}
	G.~Hong, B.~Liao, and S.~Wang, \emph{Noncommutative maximal ergodic
		inequalities associated with doubling conditions}, Duke Math. J. \textbf{170}
	(2021), no.~2, 205--246. \MR{4202493}
	
	\bibitem[HM17]{hongma17qvariation}
	G.~Hong and T.~Ma, \emph{Vector valued {$q$}-variation for differential
		operators and semigroups {I}}, Math. Z. \textbf{286} (2017), no.~1-2,
	89--120. \MR{3648493}
	
	\bibitem[HWW21]{hongwangwang21inprogress}
	G.~Hong, S.~Wang, and X.~Wang, \emph{Some new {F}ej\'er type means on
		{E}uclidean spaces and general compact groups}, in progress (2021).
	
	\bibitem[HR11]{houdayerricad11freeaw}
	C.~Houdayer and E.~Ricard, \emph{Approximation properties and absence of
		{C}artan subalgebra for free {A}raki-{W}oods factors}, Adv. Math.
	\textbf{228} (2011), no.~2, 764--802. \MR{2822210}
	
	\bibitem[Jaj91]{jajte91aus}
	R.~Jajte, \emph{Strong limit theorems in noncommutative {$L_2$}-spaces},
	Lecture Notes in Mathematics, vol. 1477, Springer-Verlag, Berlin, 1991.
	\MR{1122589}
	
	\bibitem[JW17]{jiaowang17semigroups}
	Y.~Jiao and M.~Wang, \emph{Noncommutative harmonic analysis on semigroups},
	Indiana Univ. Math. J. \textbf{66} (2017), no.~2, 401--417. \MR{3641481}
	
	\bibitem[JM04]{jolissaintmartin04haagerupsemigroup}
	P.~Jolissaint and F.~Martin, \emph{Alg\`ebres de von {N}eumann finies ayant la
		propri\'et\'e de {H}aagerup et semi-groupes {$L^2$}-compacts}, Bull. Belg.
	Math. Soc. Simon Stevin \textbf{11} (2004), no.~1, 35--48. \MR{2059174}
	
	\bibitem[Jun02]{junge02doob}
	M.~Junge, \emph{Doob's inequality for non-commutative martingales}, J. Reine
	Angew. Math. \textbf{549} (2002), 149--190. \MR{1916654}
	
	\bibitem[JLMX06]{jungelemerdyxu06Hinftycalut}
	M.~Junge, C.~Le~Merdy, and Q.~Xu, \emph{{$H^\infty$} functional calculus and
		square functions on noncommutative {$L^p$}-spaces}, Ast\'{e}risque (2006),
	no.~305, vi+138. \MR{2265255}
	
	\bibitem[JMP14]{jungemeiparcet14smooth}
	M.~Junge, T.~Mei, and J.~Parcet, \emph{Smooth {F}ourier multipliers on group
		von {N}eumann algebras}, Geom. Funct. Anal. \textbf{24} (2014), no.~6,
	1913--1980. \MR{3283931}
	
	\bibitem[JMP18]{jungemeiparcet18riesz}
	\bysame, \emph{Noncommutative {R}iesz transforms---dimension free bounds and
		{F}ourier multipliers}, J. Eur. Math. Soc. (JEMS) \textbf{20} (2018), no.~3,
	529--595. \MR{3776274}
	
	\bibitem[JNR09]{jungeneufangruan09replcqg}
	M.~Junge, M.~Neufang, and Z.-J. Ruan, \emph{A representation theorem for
		locally compact quantum groups}, Internat. J. Math. \textbf{20} (2009),
	no.~3, 377--400. \MR{2500076}
	
	\bibitem[JNRX04]{jungenielsenetal04schauderbasis}
	M.~Junge, N.~J. Nielsen, Z.-J. Ruan, and Q.~Xu, \emph{{$\mathscr C\mathscr
			O\mathscr L_p$} spaces---the local structure of non-commutative {$L_p$}
		spaces}, Adv. Math. \textbf{187} (2004), no.~2, 257--319. \MR{2078339}
	
	\bibitem[JP10]{jungeparcet10mixed}
	M.~Junge and J.~Parcet, \emph{Mixed-norm inequalities and operator space
		{$L_p$} embedding theory}, Mem. Amer. Math. Soc. \textbf{203} (2010),
	no.~953, vi+155. \MR{2589944}
	
	
	\bibitem[JRS]{jungericardshlyakhtenko2014noncommutative}
	M.~Junge, E.~Ricard, and D.~Shlyakhtenko, \emph{Noncommutative diffusion
		semigroups and free probability}, in progress.
	
	
	\bibitem[JX03]{jungexu03burkholder}
	M.~Junge and Q.~Xu, \emph{Noncommutative {B}urkholder/{R}osenthal
		inequalities}, Ann. Probab. \textbf{31} (2003), no.~2, 948--995. \MR{1964955
		(2004f:46078)}
	
	\bibitem[JX05]{jungexu05ncmartigale}
	\bysame, \emph{On the best constants in some non-commutative martingale
		inequalities}, Bull. London Math. Soc. \textbf{37} (2005), no.~2, 243--253.
	\MR{2119024}
	
	\bibitem[JX07]{jungexu07ergodic}
	\bysame, \emph{Noncommutative maximal ergodic theorems}, J. Amer. Math. Soc.
	\textbf{20} (2007), no.~2, 385--439. \MR{2276775}
	
	\bibitem[Kad52]{kadison52schwarz}
	R.~V. Kadison, \emph{A generalized {S}chwarz inequality and algebraic
		invariants for operator algebras}, Ann. of Math. (2) \textbf{56} (1952),
	494--503. \MR{0051442}
	
	\bibitem[KT02]{katztao02kakeya}
	N.~Katz and T.~Tao, \emph{Recent progress on the {K}akeya conjecture},
	Proceedings of the 6th {I}nternational {C}onference on {H}armonic {A}nalysis
	and {P}artial {D}ifferential {E}quations ({E}l {E}scorial, 2000), no. Vol.
	Extra, 2002, pp.~161--179. \MR{1964819}
	
	\bibitem[KR99]{krausruan99approximation}
	J.~Kraus and Z.-J. Ruan, \emph{Approximation properties for {K}ac algebras},
	Indiana Univ. Math. J. \textbf{48} (1999), no.~2, 469--535. \MR{1722805}
	
	\bibitem[Kus01]{kustermans01LCQGuniversal}
	J.~Kustermans, \emph{Locally compact quantum groups in the universal setting},
	Internat. J. Math. \textbf{12} (2001), no.~3, 289--338. \MR{1841517}
	
	\bibitem[Kye08]{kyed08coamenable}
	D.~Kyed, \emph{{$L^2$}-{B}etti numbers of coamenable quantum groups},
	M\"{u}nster J. Math. \textbf{1} (2008), 143--179. \MR{2502497}
	
	\bibitem[LdlS11]{delasallelafforgueduke}
	V.~Lafforgue and M. de la Salle, \emph{Noncommutative {$L^p$}-spaces without the completely bounded
	approximation property}, Duke Math. J., \textbf{160} (2011), no.~1, 71--116. \MR{2838352}

	\bibitem[Lan76]{lance76ergodic}
	E.~C. Lance, \emph{Ergodic theorems for convex sets and operator algebras},
	Invent. Math. \textbf{37} (1976), no.~3, 201--214. \MR{0428060}
	
	
	\bibitem[LS15]{leeseeger15radial}
	S.~Lee and A.~Seeger, \emph{On radial {F}ourier multipliers and almost
		everywhere convergence}, J. Lond. Math. Soc. (2) \textbf{91} (2015), no.~1,
	105--126. \MR{3338611}
	
	\bibitem[LW20]{liwu2019bochnerriesz}
	X.~Li and S.~Wu, \emph{New estimates of the maximal {B}ochner-{R}iesz operator
		in the plane}, Math. Ann. \textbf{378} (2020), no.~3-4, 873--890.
	\MR{4163515}
	
	\bibitem[LP86]{lustpiquard86khintchine}
	F.~Lust-Piquard, \emph{In\'{e}galit\'{e}s de {K}hintchine dans
		{$C_p\;(1<p<\infty)$}}, C. R. Acad. Sci. Paris S\'{e}r. I Math. \textbf{303}
	(1986), no.~7, 289--292. \MR{859804}
	
	\bibitem[LPP91]{lustpisier91nonKhintchine}
	F.~Lust-Piquard and G.~Pisier, \emph{Noncommutative {K}hintchine and {P}aley
		inequalities}, Ark. Mat. \textbf{29} (1991), no.~2, 241--260. \MR{1150376}
	
	\bibitem[MdlS17]{meidelasalle17cbofheatsemigroups}
	T.~Mei and M.~de~la Salle, \emph{Complete boundedness of heat semigroups on the
		von {N}eumann algebra of hyperbolic groups}, Trans. Amer. Math. Soc.
	\textbf{369} (2017), no.~8, 5601--5622. \MR{3646772}
	
	\bibitem[Mei07]{mei07operatorvalued}
	T.~Mei, \emph{Operator valued {H}ardy spaces}, Mem. Amer. Math. Soc.
	\textbf{188} (2007), no.~881, vi+64. \MR{2327840}
	
	\bibitem[M{\"{u}}l90]{muller90maximal}
	D.~M{\"{u}}ller, \emph{A geometric bound for maximal functions associated to
		convex bodies}, Pacific J. Math. \textbf{142} (1990), no.~2, 297--312.
	\MR{1042048}
	
	
	\bibitem[NSW78]{nagelsteinwainger78difflacunary}
	A.~Nagel, E.~M. Stein, and S.~Wainger, \emph{Differentiation in lacunary
		directions}, Proc. Nat. Acad. Sci. U.S.A. \textbf{75} (1978), no.~3,
	1060--1062. \MR{466470}
	
	\bibitem[NR11]{neuwirthricard11transfer}
	S.~Neuwirth and E.~Ricard, \emph{Transfer of {F}ourier multipliers into {S}chur
		multipliers and sumsets in a discrete group}, Canad. J. Math. \textbf{63}
	(2011), no.~5, 1161--1187. \MR{2866074}
	
	\bibitem[Oza08]{ozawa08weak}
	N.~Ozawa, \emph{Weak amenability of hyperbolic groups}, Groups Geom. Dyn.
	\textbf{2} (2008), no.~2, 271--280. \MR{2393183}
	
	\bibitem[OP10]{ozawapopa10cartan}
	N.~Ozawa and S.~Popa, \emph{On a class of {${\rm II}\sb 1$} factors with at
		most one {C}artan subalgebra}, Ann. of Math. (2) \textbf{172} (2010), no.~1,
	713--749. \MR{2680430}
	
	\bibitem[Pau02]{paulsen02completely}
	V.~Paulsen, \emph{Completely bounded maps and operator algebras}, Cambridge
	Studies in Advanced Mathematics, vol.~78, Cambridge University Press,
	Cambridge, 2002. \MR{1976867}
	
	\bibitem[PRdlS22]{delasalleparcetricardduke}
	J.~Parcet, E.~Ricard and M. de la Salle, \emph{Fourier multipliers in {${\rm SL}_n({\bf R})$}}, Duke Math. J., \textbf{171} (2022), no.~6, 1235--1297. \MR{4408121}

	\bibitem[Pis82]{pisier82holomor}
	G.~Pisier, \emph{Holomorphic semigroups and the geometry of {B}anach spaces},
	Ann. of Math. (2) \textbf{115} (1982), no.~2, 375--392. \MR{647811}
	
	\bibitem[Pis96]{pisier96OH}
	\bysame, \emph{The operator {H}ilbert space {${\rm OH}$}, complex interpolation
		and tensor norms}, Mem. Amer. Math. Soc. \textbf{122} (1996), no.~585,
	viii+103. \MR{1342022}
	
	\bibitem[Pis98]{pisier98noncommutativeLp}
	\bysame, \emph{Non-commutative vector valued {$L_p$}-spaces and completely
		{$p$}-summing maps}, Ast\'{e}risque (1998), no.~247, vi+131. \MR{1648908}
	
	\bibitem[PX03]{pisierxu03nclp}
	G.~Pisier and Q.~Xu, \emph{Non-commutative {$L^p$}-spaces}, Handbook of the
	geometry of {B}anach spaces, {V}ol. 2, North-Holland, Amsterdam, 2003,
	pp.~1459--1517. \MR{1999201}
	
	
	\bibitem[PW90]{podlesworonowicz90quantum}
	P.~Podle\'{s} and S.~L. Woronowicz, \emph{Quantum deformation of {L}orentz
		group}, Comm. Math. Phys. \textbf{130} (1990), no.~2, 381--431. \MR{1059324}
	
	\bibitem[Pop07]{popa07deformation}
	S.~Popa, \emph{Deformation and rigidity for group actions and von {N}eumann
		algebras}, International {C}ongress of {M}athematicians. {V}ol. {I}, Eur.
	Math. Soc., Z\"{u}rich, 2007, pp.~445--477. \MR{2334200}
	
	\bibitem[PV14]{popavaes14cartanii1}
	S.~Popa and S.~Vaes, \emph{Unique {C}artan decomposition for {$\rm II_1$}
		factors arising from arbitrary actions of free groups}, Acta Math.
	\textbf{212} (2014), no.~1, 141--198. \MR{3179609}
	
	\bibitem[PV15]{popavaes15reprentation}
	\bysame, \emph{Representation theory for subfactors, {$\lambda$}-lattices and
		{$\rm C^*$}-tensor categories}, Comm. Math. Phys. \textbf{340} (2015), no.~3,
	1239--1280. \MR{3406647}
	
	\bibitem[Ran02]{rand02ncmartigale}
	N.~Randrianantoanina, \emph{Non-commutative martingale transforms}, J. Funct.
	Anal. \textbf{194} (2002), no.~1, 181--212. \MR{1929141}
	
	\bibitem[Ste58]{stein58localization}
	E.~M. Stein, \emph{Localization and summability of multiple {F}ourier series},
	Acta Math. \textbf{100} (1958), 93--147. \MR{105592}
	
	\bibitem[Ste61]{stein61limit}
	E.~M. Stein, \emph{On limits of seqences of operators}, Ann. of Math. (2)
	\textbf{74} (1961), 140--170. \MR{125392}
	
	\bibitem[Ste70]{stein70booktopics}
	E.~M. Stein, \emph{Topics in harmonic analysis related to the
		{L}ittlewood-{P}aley theory}, Annals of Mathematics Studies, No. 63,
	Princeton University Press, Princeton, N.J.; University of Tokyo Press,
	Tokyo, 1970. \MR{0252961}
	
	\bibitem[Sto98]{stoll982step}
	M.~Stoll, \emph{On the asymptotics of the growth of {$2$}-step nilpotent
		groups}, J. London Math. Soc. (2) \textbf{58} (1998), no.~1, 38--48.
	\MR{1666070}
	
	\bibitem[Tao99a]{tao99restrictionbochnerriesz}
	T.~Tao, \emph{The {B}ochner-{R}iesz conjecture implies the restriction
		conjecture}, Duke Math. J. \textbf{96} (1999), no.~2, 363--375. \MR{1666558}
	
	\bibitem[Tao99b]{tao99restriction}
	\bysame, \emph{Restriction theorems and applications}, 1999, Lecture Notes Math
	254B, UCLA, URL: \url{http://www.math.ucla.edu/~tao/254b.1.99s/}.
	
	\bibitem[Tao02]{tao03bochnerrieszplane}
	\bysame, \emph{On the maximal {B}ochner-{R}iesz conjecture in the plane for
		{$p<2$}}, Trans. Amer. Math. Soc. \textbf{354} (2002), no.~5, 1947--1959.
	\MR{1881025}
	
	\bibitem[Tao04]{tao04restrictionconj}
	\bysame, \emph{Some recent progress on the restriction conjecture}, Fourier
	analysis and convexity, Appl. Numer. Harmon. Anal., Birkh\"{a}user Boston,
	Boston, MA, 2004, pp.~217--243. \MR{2087245}
	
	\bibitem[Ter81]{terp81lpspace}
	M.~Terp, \emph{{$L_p$} spaces associated with von neumann algebras}, Notes,
	Report No. 3a+3b, K\o{}benhavns Universitets Matematiske Insitut, 1981.
	
	\bibitem[Tim08]{timmermann08qgbook}
	T.~Timmermann, \emph{An invitation to quantum groups and duality}, EMS
	Textbooks in Mathematics, European Mathematical Society (EMS), Z\"urich,
	2008, From Hopf algebras to multiplicative unitaries and beyond. \MR{2397671
		(2009f:46079)}
	
	\bibitem[VD98]{van98alge}
	A.~Van~Daele, \emph{An algebraic framework for group duality}, Adv. Math.
	\textbf{140} (1998), no.~2, 323--366. \MR{1658585}
	
	\bibitem[Ver20]{vergara2019positive}
	I.~Vergara, \emph{Positive definite radial kernels on homogeneous trees and
		products}, J. Operator Theory (2020), In press.

	\bibitem[Voi85]{voiculescu85sym}
	D.~Voiculescu, \emph{Symmetries of some reduced free product
		{$C^\ast$}-algebras}, Operator algebras and their connections with topology
	and ergodic theory ({B}u\c{s}teni, 1983), Lecture Notes in Math., vol. 1132,
	Springer, Berlin, 1985, pp.~556--588. \MR{799593}
	
	\bibitem[Wan17]{wang17sidon}
	S.~Wang, \emph{Lacunary {F}ourier series for compact quantum groups}, Comm.
	Math. Phys. \textbf{349} (2017), no.~3, 895--945. \MR{3602819}
	
	\bibitem[Wei01]{lutz01fourier}
	L.~Weis, \emph{Operator-valued {F}ourier multiplier theorems and maximal
		{$L_p$}-regularity}, Math. Ann. \textbf{319} (2001), no.~4, 735--758.
	\MR{1825406}
	
	\bibitem[Wor87]{woronowicz87matrix}
	S.~L. Woronowicz, \emph{Compact matrix pseudogroups}, Comm. Math. Phys.
	\textbf{111} (1987), no.~4, 613--665. \MR{901157}
	
	\bibitem[Wor98]{woronowicz98compact}
	\bysame, \emph{Compact quantum groups}, Sym\'{e}tries quantiques ({L}es
	{H}ouches, 1995), North-Holland, Amsterdam, 1998, pp.~845--884. \MR{1616348}
	
	\bibitem[Xu21]{xu21holomorphic}
	Q.~Xu, \emph{Holomorphic functional calculus and vector-valued Littlewood-Paley-Stein theory for semigroups}, 2021.
	\href{https://arxiv.org/abs/2105.12175}{{\ttfamily arXiv:2105.12175 [math.FA]}}.

	\bibitem[XZ21]{xuzhang21inprogress}
	Z.~Xu and H.~Zhang, \emph{From the Littlewood-Paley-Stein inequality to the Burkholder-Gundy inequality}, 2021.
	\href{https://arxiv.org/abs/2111.05164}{{\ttfamily arXiv:2111.05164 [math.FA]}}.
	
	\bibitem[Yos95]{yosida95fa}
	K.~Yosida, \emph{Functional analysis}, Classics in Mathematics,
	Springer-Verlag, Berlin, 1995, Reprint of the sixth (1980) edition.
	\MR{1336382}
	
\end{thebibliography}
\end{document}